\documentclass{amsart}
\usepackage{amsmath, amsthm, amsfonts, amssymb}
\usepackage{mathrsfs,graphicx}
\usepackage{verbatim}
\usepackage{mathtools}
\providecommand{\noopsort[1]{}}
\usepackage{bbm}
\usepackage{dsfont}
\usepackage{color}
\usepackage{mathbbol}
\usepackage{ifthen}
\numberwithin{equation}{section}
\usepackage{a4}
\usepackage[colorlinks=true,linkcolor=blue, citecolor=blue]{hyperref}
\usepackage{enumerate}
\usepackage{tikz}

\newtheorem{thm}{Theorem}[section]
\newtheorem{cor}[thm]{Corollary}
\newtheorem{prop}[thm]{Proposition}
\newtheorem{lem}[thm]{Lemma}

\theoremstyle{remark}
\newtheorem{rem}[thm]{Remark}

\newtheorem{example}[thm]{Example}

\theoremstyle{definition}
\newtheorem{defn}[thm]{Definition}

\newcommand{\coloneqq}{\mathrel{\mathop:}=}

\renewcommand{\Re}{{\rm Re}\,}

\newcommand{\eps}{\varepsilon}

\newcommand{\one}{\mathds{1}}
\newcommand{\nstar}{0^\star}
\newcommand{\tstar}{t^\star}

\newcommand{\weak}{\rightharpoonup}

\newcommand{\R}{\mathds{R}}

\newcommand{\N}{\mathds{N}}

\newcommand{\cD}{\mathscr{D}}
\newcommand{\cF}{\mathscr{F}}
\newcommand{\cX}{\mathscr{X}}
\newcommand{\cXc}{\mathscr{X}_{\mathrm{C}}}
\newcommand{\cXd}{\mathscr{X}_{\mathrm{D}}}

\newcommand{\cL}{\mathscr{L}}
\newcommand{\cB}{\mathscr{B}}
\newcommand{\cM}{\mathscr{M}}
\newcommand{\cT}{\mathscr{T}}

\newcommand{\cG}{\mathscr{G}}

\newcommand{\md}{\mathbb{d}}
\newcommand{\x}{\mathbb{x}}
\newcommand{\y}{\mathbb{y}}
\newcommand{\z}{\mathbb{z}}
\newcommand{\p}{\mathbb{p}}
\let\k\undefined
\newcommand{\k}{\mathbb{k}}
\let\S\undefined
\newcommand{\E}{\mathbb{E}}
\newcommand{\T}{\mathbb{T}}
\newcommand{\D}{\mathbb{D}}
\newcommand{\S}{\mathbb{S}}
\newcommand{\Df}{\mathbb{D}_{\mathsf{full}}}
\newcommand{\A}{\mathbb{A}}
\newcommand{\B}{\mathbb{B}}
\newcommand{\Af}{\mathbb{A}_{\mathsf{full}}}
\let\P\undefined
\newcommand{\P}{\mathbb{P}}
\newcommand{\bP}{\mathbf{P}}
\newcommand{\expect}{\mathbf{E}}
\newcommand{\bo}{\mathfrak{B}}

\newcommand{\vt}{\vartheta}
\newcommand{\Th}{\Theta}

\newcommand{\cm}{C_{b}(\cX, \cF_{0-})}

\newcommand{\bm}{B_{b}(\cX, \cF_{0-})}

\newcommand{\chat}{\hat{C}_b(\cX^-)}
\newcommand{\cXxt}{C_{b,\mathrm{ext}}(\cX^-)}

\newcommand{\bb}[1]{B_{b}(\mathscr{X}, \cF(#1))}
\newcommand{\cb}[1]{C_{b}(\mathscr{X}, \cF(#1))}
\newcommand{\Ch}{\mathscr{C}_h}
\newcommand{\Dh}{\mathscr{D}_h}

\newcommand{\Lf}{L_{\mathsf{full}}}
\DeclareMathOperator{\tr}{tr}

\begin{document}
\title{Evolutionary semigroups on path spaces}
\author{Robert Denk}
\email{robert.denk@uni-konstanz.de}
\author{Markus Kunze}
\email{markus.kunze@uni-konstanz.de}
\author{Michael Kupper}
\email{kupper@uni-konstanz.de}
\address{Universit\"at Konstanz, Fachbereich Mathematik und Statistik, 78357 Konstanz, Germany}

\begin{abstract}
    We introduce the concept evolutionary semigroups on  path spaces,
    generalizing the notion of transition semigroups to possibly
    non-Markovian stochastic processes. 
    We study the basic properties 
    of evolutionary semigroups and, in particular, prove that they
    always arise as the composition of the shift semigroup and 
    a single operator called the expectation operator of the
    semigroup. We also prove that the transition semigroup of
    a Markov process can always be extended to an evolutionary 
    semigroup on the path space whenever the Markov process 
    can be realized with the appropriate path regularity.
    As first examples of evolutionary semigroups associated to
    non-Markovian processes, we discuss deterministic evolution
    equations and stochastic delay equations.
\end{abstract}

\date{\today}

\maketitle

\section{Introduction}

An important object in the study of a Markov process (say with a Polish state space $X$) is its transition semigroup. 
Indeed, encoded in the transition semigroup are the transition probabilities of the process and, together with the initial distribution, these determine the finite
dimensional marginals and thus the distribution of the process as a random variable with values in the space of
all $X$-valued functions on $[0, \infty)$. In particular, in view of Kolmogorov's extension theorem, we can construct a Markov process with 
a prescribed transition semigroup. Depending on the Markov process in question, it is sometimes possible to replace the space of all
$X$-valued functions on $[0,\infty)$ with a space of more regular functions. The most important examples are the space
$C([0,\infty); X)$ of all continuous $X$-valued functions and $D([0,\infty); X)$, the space of all c\`adl\`ag $X$-valued functions.
But in all cases we can say that all information about the stochastic process is encoded in the transition semigroup and vice versa. 

In this article, we introduce the concept of \emph{evolutionary semigroups on path spaces} to extend the notion of  transition semigroups
to possibly non-Markovian stochastic processes. Our main guidelines in establishing this concept are two statements which may be 
considered  `mathematical folk wisdom' -- not proven facts but rather an intuitive understanding on how a 
sensible theory should look like. These are as follows:\smallskip

\emph{If we enlarge the state space appropriately, every stochastic process becomes Markovian}. This suggests that we should incorporate
enough `history' of the stochastic process into the state space to make it Markovian. Consequently, as an `ultimate state space', we should
consider a space of $X$-valued functions on the interval $(-\infty, 0]$, where we consider the value at time $t=0$ as the present state
of the process and the value at time $t<0$ as the position of the process at the past time $t<0$. Thus, if we seek to describe
a stochastic process with continuous paths (and we will assume this throughout this introduction; however our general setup
also allows for different `path spaces', in particular for c\`adl\`ag paths), we should use $\cX^-=C((-\infty, 0];X))$ as  state space.
This strategy of `incorporating the past' into the state space is rather standard in the semigroup approach to (deterministic) delay
equations, see \cite{bp05} and the references therein, but it was also used for stochastic delay equations, see \cite{mohammed84}.
We should point out that often (in particular in the references just mentioned) 
only the history in a finite time horizon is considered, i.e.\ one uses $C([-h, 0]; X)$ instead of $C((-\infty, 0];X)$.

However, there is a fundamental difference between semigroup theory for deterministic delay equations and the theory
that we want to develop here. In the former, $X$ itself is a Banach space and one constructs a semigroup on the space $C([-h,0]; X)$ 
(or a similar space such as
$L^p([-h, 0];X)$) that describes the evolution of the solution of the
delay equation. 
On the other hand, in this article, $X$ is not a vector space in general and (similar to \cite{mohammed84}) we want to
use the `path space' $\cX^- = C((-\infty, 0]; X)$ as state space of a then Markovian process.
Thus, the transition semigroup acts on the space $C_b(\cX^-) = C_b\big(C((-\infty, 0];X))$. In a sense, the semigroup we seek to construct acts on
a function space that is `one level higher' than that considered in the classical theory. 
This is similar to considering the \emph{Koopman operator} (or, in the time continuous setting, \emph{Koopman semigroups}) in the study of dynamical systems, see \cite{efhr15}. Typically, the Koopman operator/semigroup acts either on an $L^p$-space or (closer to our situation) on the space
$C(K)$ of continuous functions on a compact space $K$.
Only recently, there was a generalization of this 
theory to the setting of completely regular spaces
in \cite{fk2020}.
\smallskip

As a second guideline we require that, \emph{as time passes, the past of the process is merely shifted in time}. While this
guideline does not need additional explanation, it raises an important question. What should
happen with the past once it is shifted into the future? In the semigroup approach to delay equations
mentioned above, this question is answered at an infinitesimal level. The generator of the shift semigroup
is (roughly speaking) the first derivative. To obtain a semigroup governing the evolution of a delay equation, one considers
a realization of the first derivative, defined on a subspace of $C^1([-h,0]; X)$, where  the delay equation itself enters the domain 
of the generator as a (Neumann-type) boundary condition at time $t=0$. In our situation, the shift semigroup 
is a rather delicate object (see Remark \ref{r.evaluation}) and this approach does not seem to work.

Therefore, in this article we follow a different approach. Setting $\cX = C(\R; X)$, we identify the space $C_b(\cX^-)$ with a subspace of
$C_b(\cX)$ as follows: Given $\x\in \cX$ and $\tilde{F}\in C_b(\cX^-)$, we can extend $\tilde{F}$ to a function $F \in C_b(\cX)$,
by setting $F(\x) = \tilde{F}(\x|_{(-\infty, 0]})$. Note that $F(\x) = F(\y)$ for any $\x, \y \in \cX$
with $\x(t) = \y(t)$ for all $t\leq 0$ and this property actually characterizes functions that appear as extension of functions in $C_b(\cX^-)$.
As it turns out (see Lemma \ref{l.measurability}) this property is equivalent to $F$ being measurable with respect to
the $\sigma$-algebra $\cF_0$ that is generated by the point evaluations $(\pi_t)_{t\leq 0}$. 
Similarly, functions $F\in C_b(\cX)$ whose
value only depends on $\x|_{(-\infty, t]}$ are exactly those that are measurable with respect to $\cF_t \coloneqq \sigma(\pi_s : s\leq t)$.
Let us write $C_b(\cX, \cF_t)$ for the space of bounded continuous functions $F : \cX \to \R$ that are $\cF_t$-measurable.
With this identification at hand, we can work throughout on the space $C_b(\cX)$ and, in particular, avoid the central question posed above.
If we shift a function $F\in C_b(\cX,\cF_0)$ (say by the time lapse $t>0$), we again obtain a function in $C_b(\cX)$, 
which is $\cF_t$-measurable rather than $\cF_0$-measurable.

After this preparation, we now 
define an \emph{evolutionary semigroup} as a semigroup $(\T(t))_{t\geq 0}$ on $C_b(\cX, \cF_0)$ such that for 
$0\leq s \leq t$ and $F\in C_b(\cX, \cF_{-t})$ it holds $\T(s)F = \Th_s F$, where $(\Th_t)_{t\geq 0}$ is the \emph{shift semigroup}
on $C_b(\cX)$, see Section \ref{sect.shiftsg}. 
This requirement is a rephrasement of our second guiding principle.
In our first main result, Theorem \ref{t.expectation_semigroup}, we prove that a semigroup on $C_b(\cX, \cF_0)$
is evolutionary if and only if it is given as $\T(t) = \E \Th_t$ for some bounded linear operator $\E$ on $C_b(\cX)$
that takes values in $C_b(\cX, \cF_0)$ and satisfies $\E F = F$ for all $F\in C_b(\cX, \cF_0)$. This operator $\E$ determines
the semigroup $(\T(t))_{t\geq 0}$ uniquely and is called the \emph{expectation operator} of the semigroup. Proposition 
\ref{p.expectation_operator} shows that $\E$ indeed has properties that are characteristic for expectations.
The expectation operator also connects (various notions of) the generator of $\T$ to that of the shift semigroup, see Propositions 
\ref{p.bpcorefora}, \ref{p.generator} and \ref{p.generatorgen}.
We point out that the expectation operator now settles the central question posed above: Once the shift operator $\Th_t$
transports information into the future (i.e.\ producing a function that is not $\cF_0$-measurable any more) the expectation operator
$\E$ transforms this into an $\cF_0$-measurable function (and thus into a function that we might view as a function on $\cX^-$).
\smallskip

Let us now discuss some examples.

\subsection*{Delay equations}

Our first example concerns deterministic delay equations. We fix $d\in \N$, $h>0$ and define
\[
\Ch \coloneqq C([-h,0]; \R^d).
\]
As usual, for a function $y \in C([-h, \infty); \R^d)$ and $t\geq 0$, we define $y_t \in \Ch$, the \emph{past at time $t$},
by setting $y_t(s) = y(t+s)$ for $s\in [-h, 0]$. 
Now, given a Lipschitz map $b: \R^d \to \R^d$ and an `initial history' $\xi \in \Ch$, we may consider the delay equation
\begin{equation}
    \label{eq.introdde}
    \left\{ 
    \begin{aligned}
        y'(t) & = b(y_t), \quad\mbox{for }t\ge 0,\\
        y_0 & = \xi.
    \end{aligned}
    \right.
\end{equation}
As $b$ is assumed to be Lipschitz continuous, an argument based on the Banach fixed point theorem shows that \eqref{eq.introdde}
has a unique solution $y^\xi$, see \cite[Theorem II.4.3.1]{bddm07}. In Section \ref{sub.deterministic}, we will see that, setting
$X= \R^d$ and $\cX = C(\R; \R^d)$ as above, we can associated an evolutionary semigroup $\T$ on $C_b(\cX, \cF_0)$ with such equations.
To describe the expectation operator, given $\x \in \cX$, we define the function $\x_- \in \cX^- = C((-\infty, 0]; \R^d)$ as the restriction
$\x_{(-\infty, 0]}$. Note that $\x_0 \in \Ch$ so that \eqref{eq.introdde} has a unique solution $y^{\x_0}$ for $\xi = \x_0$. We set
\[
\x_-\oplus_0 y^{\x_0}\coloneqq \begin{cases}
    \x(t), & \mbox{if } t \leq 0,\\
    y^{\x_0}, & \mbox{if } t> 0.
\end{cases}
\]
Then $\x_-\oplus_0y^{\x_0}\in \cX$ and we define for $F\in C_b(\cX)$,
\begin{equation} 
\label{eq.evolutiondde}
[\E F](\x) \coloneqq F(\x_-\oplus_0y^{\x_0}).
\end{equation}
We will see in Section \ref{sub.deterministic}, that this is indeed an expectation operator that induces an evolutionary semigroup 
by setting $\T(t) \coloneqq \E\Th_t$.

We point out that there are several semigroup approaches to delay equations to be found in the literature, see, for example,
\cite{bp05}, \cite[Section VI.6]{en00} or \cite[Section II.4]{bddm07}. However, to the best of our knowledge, all these approaches need the
additional assumption that the map $b$ is linear. To obtain an evolutionary semigroup $\T$ induced by $\E$, we do not need $b$ to be linear.
The semigroup $\T$ can be seen as a Koopman approach to delay equations or as an extension of the results from \cite{fk2020} to equations that depend on the past.

\subsection*{Stochastic differential equations}

Next, let $(B(t))_{t\geq 0}$ be an $m$-dimensional Brownian motion, defined on a probability space $(\Omega, \Sigma, \bP)$. Given
functions $b: \R^d \to \R^d$ and $\sigma: \R^d \to \R^{m\times d}$, we can consider the stochastic differential equation 
\begin{equation}
    \label{eq.introsde}
    \left\{
    \begin{aligned}
        dY(t) & = b(Y(t))\, dt + \sigma (Y(t))dB(t), \quad \mbox{ for } t\geq 0,\\
        Y(0) & = y.
    \end{aligned}
    \right.
\end{equation}
Under suitable assumptions on the coefficients $b$ and $\sigma$, it is well known that for every initial datum $y\in \R^d$, Equation \eqref{eq.introsde} has a unique solution $(Y^y(t))_{\geq 0}$. It turns out that also to this equation we can associate an evolutionary semigroup
on the path space $\cX = C(\R, \R^d)$. Motiaved by \eqref{eq.evolutiondde}, we set for $F\in C_b(\cX)$,
\[
[\E F](\x) \coloneqq \expect \big[F(\x_-\oplus_0Y^{\x(0)})\big],
\]
where $\expect$ refers to expectation with respect to the measure $\bP$. 
Using the fact that the solutions to \eqref{eq.introsde} are Markov processes,
we prove in Section \ref{sub.markov} that $\E$ is an expectation operator that induces an evolutionary semigroup. As a matter of fact, this is true not only 
for solutions of stochastic differential equations, but for all Markov processes that can be realized with continuous paths. Switching to the path space of c\`adl\`ag paths, the same result also applies to Markovian processes with c\`adl\`ag paths. In Theorem \ref{t.diffusion}, 
we give a complete
characterization of evolutionary semigroups that arise in this way.

\subsection*{Stochastic delay equations}

For our last example, consider stochastic delay equations of the form
\begin{equation}
    \label{eq.introsdde}
    \left\{
    \begin{aligned}
    dY(t) & = b(Y_t)dt + \sigma(Y_t)dB(t),\quad \mbox{ for } t\geq 0,\\
    Y_0 & = \xi.
    \end{aligned}
    \right.
\end{equation}
Here, once again, $(B(t))_{t\geq 0}$ is an $m$-dimensional Brownian motion on the probability space $(\Omega, \Sigma, \bP)$. However, 
compared to \eqref{eq.introsde}, the domain of the maps $b$ and $\sigma$ changes from $\R^d$ to the function space $\Ch$. Nevertheless,
under Lipschitz assumptions on the coefficients, it is known that for any $\xi\in \Ch$, Equation \eqref{eq.introsdde} has a unique solution
$Y^\xi$ and we may define
\[
[\E F](\x) \coloneqq \expect \big[F( \x_-\oplus_0Y^{\x_0})\big].
\]
We point out that the solutions $Y^\xi$ are no longer Markov processes on $\R^d$, but they turn out to be Markov processes on the space
$\Ch$. Using this fact, we prove in Section \ref{sec:SDE} that $\E$ is indeed the expectation operator of an evolutionary semigroup on
the space $\cX = C(\R, \R^d)$. In Section \ref{sub.flow}, we discuss some related results in the c\`adl\`ag setting.

\subsection*{Organization of this article}
In Section \ref{sect.pathspace}, we introduce the abstract concept of a `path space' 
and establish some preliminary measurability results. Our main examples of path spaces are
$C(\R; X)$ and $D(\R; X)$, as we will prove in Appendix \ref{ap.pathspace}. In Section \ref{sect.shiftsg}, we introduce the shift semigroup and present some results concerning its full generator and its $C_b$-generator.
We should point out that the semigroups considered in this article are not strongly continuous, whence a different semigroup
theory is used throughout. 
The main definitions and results of this theory are collected in Appendix \ref{sect.appendix}, where also references
to the literature may be found. In Section \ref{sect.measurable}, we take up our main line of study and introduce the central concepts
of `evolutionary semigroup' and `expectation operator'. In Section \ref{sect.measurable}
we do not work in $C_b(\cX)$ but rather in $B_b(\cX)$, the space of all bounded measurable functions on the path space $\cX$.
Continuity properties of evolutionary semigroups are discussed in Section \ref{sect.continuous}. The choice of
the path space plays an important role here. 
The case where $\cX = C(\R; X)$ is the easiest one and all of the expected
results hold true. The general case is much more involved and we actually impose additional assumptions on the `path space' 
(which, however, are satisfied in the case of c\`adl\`ag paths).
The basic problem in the case of c\`adl\`ag paths is that point evaluations are not continuous functions. 
Nevertheless, most results from the continuous case generalize. In the concluding Section \ref{sect.examples}, we present our examples.

\subsection*{Acknowledgement}
We are grateful to Rainer Nagel for helpful comments.

\section{Path spaces}
\label{sect.pathspace}

\begin{defn}
\label{def.pathspace}
Let $(X,d)$ be a complete separable metric space.
A \emph{path space (with state space $X$)} is a pair $((\cX, \md), \tau)$ consisting of a complete separable metric space $(\cX, \md)$ of right-continuous functions $\x : \R \to X$ and a map
$\tau: \cX \to \cX$, such that the following conditions are satisfied:
\begin{itemize}
\item[(P1)] The \emph{evaluation maps} $\pi_t : \cX \to X$, $\pi_t(\x) = \x(t)$ are Borel measurable and the Borel $\sigma$-algebra
$\bo (\cX)$ is generated by these maps, i.e.\ $\bo (\cX) = \sigma (\pi_t : t\in \R)$. Every $\x\in \cX$ is continuous at almost every
$t\in \R$. Moroever, if $\x_n\to \x$ and $\x$ is continuous at $t$, then $\pi_t(\x_n)\to \pi_t(\x)$.

We define some additional $\sigma$-algebras. Given an interval $I\subset \R$, we set
\[
\cF(I) \coloneqq  \sigma (\pi_t : t\in I).
\]
Of particular importance is the case where $I= (-\infty, t)$ or $I= (-\infty, t]$ for some $t\in \R$.
We define $\cF_{t-} \coloneqq \cF((-\infty,t))$ and $\cF_t\coloneqq \cF((-\infty,t])$.
\item[(P2)] The \emph{stopping map} $\tau : \cX \to \cX$ is $\cF_{0-}$-measurable and $\tau^{-1}(A) = A$ for every $A\in \cF_{0-}$. 
Moreover, $\tau(\x)$ is continuous at $0$ for every $\x\in \cX$ and 
the image $\cX^-\coloneqq \tau(\cX)$ is Polish, i.e.\ there exists a complete metric $\md^-$ on $\cX^-$ that induces the same 
topology as $\md$ on $\cX^-$.
\item[(P3)] For every $t\in \R$, the \emph{shift} $\vt_t$, defined by $[\vt_t\x](s) \coloneqq \x(t+s)$,  maps $\cX$ to itself and the 
map $(t,\x) \mapsto \vt_t\x$ is continuous from $\R\times\cX\to \cX$.
\end{itemize}
\end{defn}

Throughout, elements of the path space $\cX$ will be denoted by lower case blackboard letters $\x, \y, \z$ whereas elements of the state space
$X$ will be denoted by lower case roman letters $x, y, z$. Likewise, scalar-valued functions on $\cX$ will be denoted by upper case roman letters 
$F, G, H$, whereas scalar-valued functions on $X$ will be denoted by lower case roman letters $f, g, h$.

\begin{example}
Our basic examples of path spaces are $\cXc\coloneqq C(\R; X)$, the space of all \emph{continuous paths}
from $\R$ to $X$, endowed with the metric that topologizes uniform convergence on compact subsets of $\R$
and $\cXd\coloneqq D(\R; X)$, the space of all \emph{c\`adl\`ag paths} from $\R$ to $X$, endowed with 
Skorohod's $J_1$-metric. We point out that the stopping map is defined slightly different in these two examples:
for $\cXc$, we define $[\tau_\mathrm{C}(\x)](t) = \x(0)$ for $t\geq 0$, whereas for c\`adl\`ag paths
$\cXd$ we put $[\tau_{\mathrm{D}}(\x)](t) = \x(0-)$ for $t\geq 0$. 

As our main results only use the abstract assumptions of Definition \ref{def.pathspace}, we postpone
the details and proofs to Appendix \ref{ap.pathspace}.
\end{example}

\begin{example}
    \label{ex.limitpathspace}
    In the context of continuous paths, we can also consider the spaces
    \[
     \cX_{\mathrm{C}, \ell} \coloneqq \{ \x \in C(\R; X) : \lim_{t\to-\infty}\x(t) \mbox{ exists}\} = C([-\infty, \infty); X)
    \]
    and, if $X$ is additionally a vector space, 
   \[
    \cX_{\mathrm{C}, 0} \coloneqq \{\x \in C(\R; X) : \lim_{x\to -\infty}\x(t) = 0\}
    \]
    Both are path spaces with respect to the metric
    \[
    \md(\x, \y) = \sum_{n=1}^\infty 2^{-n} \big[ 1\wedge \sup_{t \in (-\infty, n]} d(\x (t), \y(t))\big].
    \]
    This choice of path spaces is helpful, for example, in the context of delay equations with infinite delay, see Equation (3.2) on page 240
    of \cite{bddm07}.
\end{example}

\begin{example}
    Another possible choice for a path space is motivated by adding a cemetary state to the space $X$. This is a standard procedure
    to turn a non-honest Markov process into an honest one. This construction can be replicated on the level of the  path space. 
    Indeed, given a locally compact $X$, we denote its one-point compactification of $X$ by $X^\dagger$, 
    where the added point $\dagger$ is used as a cemetery state.
    As $X$ is a complete separable metric space, so is $X^\dagger$. Now set
    \[
    \cX_{\mathrm{C}, \dagger} \coloneqq \{ \x \in C(\R; X^\dagger) : \x(t_0) = \dagger \mbox{ implies } \x(t) = \dagger \mbox{ for all } t\geq t_0\}.
    \]
\end{example}

We also introduce the following notation. Given an interval $I\subset \R$, we put
\[
\bb{I} \coloneqq \{ F \in B_b(\cX) : F \mbox{ is } \cF(I)\mbox{-measurable}\}
\]
and $\cb{I} \coloneqq \bb{I}\cap C_b(\cX)$.  We often write $B_b(\cX, \cF_{t-})$ and $C_b(\cX, \cF_{t-})$ instead
of $\bb{(-\infty, t)}$ and $\cb{(-\infty, t)}$. If $F: \cX \to \R$ is $\cF_{t-}$-measurable, we will say that it is
\emph{determined before $t$}, if $F$ is $\cF([t,\infty))$-measurable, we will say that it is
\emph{determined after $t$}.
Thus, $\bb{[t_1,t_2)}$ refers to the space of bounded, measurable functions that are determined before $t_2$ and after $t_1$.
Endowed with the supremum norm $\|\cdot\|_\infty$, all of these spaces are Banach spaces.

Let us note some easy consequences concerning measurability.

\begin{lem}\label{l.measurability}
Let $((\cX, \md), \tau)$ be a path space and $S$ be a Polish space with Borel $\sigma$-algebra $\bo(S)$.
\begin{enumerate}
[\upshape (a)]
\item For every $t \in \R$ and interval $I\subset \R$, the map $\vt_t$ is 
$\cF(I+t)/\cF(I)$-measurable.
\item For $t\in \R$, the map $\tau^t \coloneqq \vt_{-t}\circ\tau\circ \vt_t$ is $\cF_{t-}$-measurable and $(\tau^t)^{-1}(A) = A$
for all $A\in \cF_{t-}$. 
\item A measurable function  
$\Phi: \cX \to S$ is $\cF_{t-}$-measurable if and only if $\Phi= \Phi\circ \tau^t$.
\item It holds $\tau^2=\tau$. Moreover, $\x \in \cX^-$ if and only if $\x=\tau(\x)$.
\item It holds $\x(t) = [\tau(\x)](t)$ for all $t<0$ and $\x\in \cX$.
\end{enumerate}
\end{lem}

\begin{proof}
(a). This follows directly from the identity $\pi_r\circ\vt_t = \pi_{r+t}$ for all $r\in \R$.\smallskip

(b). If $A\in \bo (\cX)$, then $B \coloneqq \tau^{-1}((\vt_{-t})^{-1}(A)) \in \cF_{0-}$ by (P2), so that $(\vt_t)^{-1}(B) \in \cF_{t-}$ by (a), proving
the $\cF_{t-}$-measurability of $\tau^t$. If $A\in \cF_{t-}$, then $(\vt_{-t})^{-1}(A) \in \cF_{0-}$ by (a) so that $B= (\vt_{-t})^{-1}(A)$ by (P2).
At this point, (a) yields $A = (\vt_t)^{-1}(B) = (\tau^t)^{-1}(A)$.\smallskip

(c). If $\Phi= \Phi\circ \tau^t$, then $\Phi$ is $\cF_{t-}$-measurable as $\tau^t$ is. Conversely, if $\Phi$ is $\cF_{t-}$-measurable, then for
$A\in \Sigma$, $\Phi^{-1}(A) \in \cF_{t-}$ so that $\Phi^{-1}(A) = (\tau^t)^{-1}(\Phi^{-1}(A))$ by (b). As $A$ was arbitrary, it follows that
$\Phi= \Phi\circ\tau^t$.\smallskip

(d). Applying (c) to $\Phi=\tau$ yields $\tau^2=\tau$. If $\x= \tau(\x)$ then $\x\in \cX^-$. Conversely, assume that
$\x= \tau(\y)\in \cX^-$. Then $\tau(\x) = \tau^2(\y) = \tau(\y) = \x$.\smallskip

(e). Let $A \coloneqq \{ \x\in \cX : \x(t) = [\tau(\x)](t) \mbox{ for all } t< 0\}$. Then $A\in \cF_{0-}$ and thus $A=\tau^{-1}(A)$
by (P2). On the other hand, $\tau^{-1}(A) = \{ \x : [\tau(\x)](t) = [\tau^2(\x)](t) \mbox{ for all } t< 0\} = \cX$ by (d).
\end{proof}

As a consequence of Lemma \ref{l.measurability}(c), an $\cF_{0-}$-measurable function $F\in B_b(\cX)$
is uniquely determined by its values on $\cX^-$. We may thus use the map $\tau$ to identify
functions on $\cX^-$ with  functions on $\cX$ by means of the \emph{extension map} $\Phi\mapsto \hat \Phi$, defined by $\hat \Phi = 
\Phi\circ\tau$.
Concerning measurability, we have the following result.

\begin{lem}\label{l.eminusf0}
Let $((\cX, \md), \tau)$ be a path space, $S$ be a Polish space with Borel $\sigma$-algebra $\bo(S)$ and $\Phi:\cX \to S$.
\begin{enumerate}
[\upshape (a)]
\item The Borel $\sigma$-algebra $\bo(\cX^-)$ is the trace $\sigma$-algebra of $\cF_{0-}$ on $\cX^-$.
\item Then the function $\Phi$ is $\bo(\cX^-)$-measurable 
if and only if $\hat{\Phi}$ is $\cF_{0-}$-measurable.
\end{enumerate}
\end{lem}

\begin{proof}
(a). By \cite[Lemma 6.2.4]{bogachev2007}, $\bo(\cX^-)$ is the trace of $\bo(\cX)$ on $\cX^-$, i.e.\ $A\in \bo(\cX^-)$
if and only if there exits $B\in \bo(\cX)$ with $A= B\cap \cX^-$. For $A$ of this form, it 
follows from Lemma \ref{l.measurability}(d) that $\tau(A) = A$ which, by (P2), implies $A\in \cF_{0-}$. This shows that
$\bo(\cX^-)\subset \cF_{0-}$ and thus $\bo(\cX^-)$ is contained in the trace of $\cF_{0-}$ on $\cX^-$. The converse inclusion follows
from considering $B\in \cF_{0-}$.

(b). Assume that $\Phi$ is $\bo(\cX^-)$-measurable. By (a), for every $A\in \Sigma$ it is $B\coloneqq \Phi^{-1}(A)\in \bo(\cX^-)\subset
\cF_{0-}$. By (P2), $\hat{\Phi}^{-1}(A) = \tau^{-1}(\Phi^{-1}(A)) =\tau^{-1}(B) = B$, proving that $\hat{\Phi}$ is 
$\cF_{0-}$-measurable.

Conversely assume that $\hat{\Phi}$ is $\cF_{0-}$-measurable. Then for every $A\in \Sigma$ we have
$B\coloneqq \hat{\Phi}^{-1}(A)\in \cF_{0-}$, whence $\tau^{-1}(B) = B$. This implies $\tau(B) = B$ and thus
$B\subset \cX^-$ in view of Lemma \ref{l.measurability}(d). In particular, $B=B\cap\cX^-\in \bo(\cX^-)$.
It follows that $\tau^{-1}(\Phi^{-1}(A)) = \tau^{-1}(B)$ and thus $\Phi^{-1}(A) = B$ proving that $\Phi$ is $\bo(\cX^-)$-measurable.
\end{proof}

\section{The shift semigroup}\label{sect.shiftsg}

Throughout, $((\cX, \md),\tau)$ is a path space. We next introduce the \emph{shift group} $(\Th_t)_{t\in \R} \subset \cL(B_b(\cX), \sigma)$ by setting
\[
(\Th_t F)(\x) \coloneqq F(\vt_t\x)
\]
for $t\in \R$ and $F\in B_b(\cX)$. Here, $\cL(B_b(\cX), \sigma)$ refers to the space of bounded kernel operators, see Section \ref{s.a1} and, 
in particular, Lemma \ref{l.kernelop}.

Lemma \ref{l.measurability}(a) immediately yields

\begin{cor}\label{c.measurability}
Let $I\subset \R$ be an interval and $t \in \R$. If $F\in \bb{I}$, then $\Th_tF \in \bb{I+t}$.
\end{cor}

Let us briefly recall some notions concerning semigroups.
A \emph{$C_b$-semigroup} is a family $\S =(\S_t)_{t\geq 0}$ of Markovian kernel operators on $\cX$ that leave the space
$C_b(\cX)$ invariant and such that the orbit of every bounded, continuous function is jointly continuous in $t$ and $\x$.
The \emph{$C_b$-generator} $\A$ of a $C_b$-semigroup $\S$ is defined as follows. We have $F \in D(\A)$ and $\A F = G$ if and only if
$\sup_{t\in (0,1)}t^{-1}\|\S(t)F- F\|_\infty <\infty$ and $t^{-1}(\S(t)F(\x) - F(\x)) \to G(\x)$ for all $\x \in \cX$ 
as $t\to 0$. See Theorem \ref{t.genchar} for equivalent descriptions of the $C_b$-generator.

\begin{prop}\label{p.derivation}
Given (P1) and (P2), condition (P3) is fulfilled if and only if 
$(\Th_t)_{t\in \R}$ is a $C_b$-group.
In this case, the $C_b$-generator of $(\Th_t)_{t\in R}$ is denoted by $\D$ and the following hold true:
\begin{enumerate}
[\upshape (a)]
\item $\D$ is a \emph{derivation}, i.e.\ for $F,G\in D(\D)$ also the product $FG\in D(\D)$ and
$\D(FG) = (\D F)G + F(\D G)$;
\item If $F\in D(\D)$ is determined before (after) time t, then so is $\D F$.
\end{enumerate}
\end{prop}

\begin{proof}
If (P3) is satisfied then for every $F\in C_b(\cX)$ the map
$(t,\x) \mapsto (\Th_tF)(\x) = F(\vt_t\x)$ is continuous which shows that $(\Th_t)_{t\in \R}$ is a $C_b$-group.

To see the converse, assume that there exist sequences $t_n \to t$ and $\x_n \to \x$ such that
$\md(\vt_{t_n}\x_n, \vt_t\x) \geq \eps>0$ for all $n\in \N$. As $\cX$ is a metric space, we find a bounded continuous function $F$ on $\cX$
such that $F(\vt_t\x) = 1$ whereas $F(\y)= 0$ whenever $\md(\y, \vt_t\x) \geq \eps$. For this function $F$ we have
\[
0 \equiv (\Th_{t_n}F)(\x_n) \not\to (\Th_tF)(\x) = 1.
\]
Thus, for this particular $F$ the map $(t,\x)\mapsto [\Th_tF](\x)$ is not continuous, whence
 $(\Th_t)_{t\in \R}$ is not a 
$C_b$-group.\smallskip

(a).  Let $F, G\in D(\D)$. Using the characterization of the $C_b$-generator from Theorem \ref{t.genchar}(iv), we find
\[
\frac{\Th_t(FG) - FG}{t}(\x) = (\Th_tF)(\x)\frac{(\Th_tG)(\x) - G(\x)}{t} + \frac{(\Th_tF)(\x) - F(\x)}{t}G(\x)
\]
which is uniformly bounded and converges to $F(\x) [\D G](\x) + [\D F](\x) G(\x)$ as $t\to 0$.\smallskip

(b). Let $F\in D(\D)$ be $\cF_{t-}$-measurable. By Corollary \ref{c.measurability}
 $\Th_{-s}F$ is $\cF_{t-}$-measurable for every $s\geq 0$. Thus, for every $s>0$ the difference
quotient $(\Th_{-s} F - F)/(-s)$ is $\cF_{t-}$-measurable hence so is the (pointwise) limit $\D F$.

In the case where $F$ is determined after time $t$ we can proceed similarly, considering the difference quotients
$(\Th_sF-F)/s$ for $s>0$ instead.
\end{proof}

Besides the $C_b$-generator $\D$ also the \emph{full generator} $\Df$ of $(\Th_t)_{t\in \R}$ is of interest.
This operator is typically multivalued and is defined via the Laplace transform of the semigroup
on $B_b(\cX)$. Consequently, it may contain functions that are not continous, which will be important
in what follows. For more information about the full generator, we refer to Section \ref{s.a2}.

\begin{cor}\label{c.derivative}
The full generator $\Df$ of $(\Th_t)_{t\in \R}$ is a derivation in the sense that if $(F_j, G_j)\in \Df$ for 
$j=1,2$, then also $(F_1F_2, F_1G_2+F_2G_1)\in \Df$.
\end{cor}

\begin{proof}
Using the characterization of the full generator from Proposition \ref{p.awf}(iii), this follows from the multiplicativity of 
$(\Th_t)_{t\in \R}$ and the product rule for Sobolev functions.
\end{proof}

Next, we introduce some specific elements of $\Df$. Given $f\in B_b(X)$ and $a<b$ and $t\in \R$, we define $F_a^b(f)$ and $F_t(f)$ by setting
\begin{align}
\label{eq.fab}
[F_a^b(f)](\x) &\coloneqq \int_a^bf(\x(s))\, ds \quad \mbox{and}\\
\label{eq.fa}
[F_t(f)](\x) &\coloneqq f(\x(t))
\end{align}
for every $\x\in \cX$. Obviously, $F_a^b(f)$ and $F_t(f)$ are bounded measurable functions and $F_a^b(f) \in \bb{[a,b)}$. 
On the other hand, in general $F_t(f)$ is not $\cF_{t-}$-measurable. 
As this leads to technical problems in what follows, we will often replace $F_t(f)$ by the function 
$F_t^\star(f)$, defined by
\begin{equation}
\label{eq.fastar}
[F_t^\star(f)](\x)\coloneqq \limsup_{n\to\infty}n\int_{t-2/n}^{t-1/n}f(\x(s))\, ds.
\end{equation}

In the case where $\cX \in \{\cXc, \cXd\}$, we set
\[
\tstar \coloneqq \begin{cases}
t, & \mbox{if } \cX = \cXc,\\
t-, & \mbox{if } \cX = \cXd.
\end{cases}
\]

For future reference, we collect some easy properties of $F_t^\star(f)$.

\begin{lem}\label{l.fbstar}
Given $t\in \R$ and $f\in B_b(X)$, the following hold true:
\begin{enumerate}
[\upshape (a)]
\item $F_t^\star(f)$ is $\cF_{t-}$-measurable;
\item If $f\in C_b(X)$ and $t$ is a continuity point of $\x$, then $[F_t^\star(f)](\x) = [F_t(f)](\x)$;
\item  If $\cX \in \{\cXc, \cXd\}$, then $F_t^\star(f) = F_{\tstar}(f)$ for all $t\in \R$ and $f\in C_b(X)$;
\item For $s\in \R$, it is $\Th_sF_t^\star(f) = F_{t+s}^\star(f)$.
\end{enumerate}
\end{lem}

%\begin{proof}
%(a) and (b) are obvious. For (c), note that
%\begin{align*}
%[\Th_sF_t^\star(f)](\x) &=\limsup_{n\to\infty}n\int_{t-2/n}^{t-1/n}f(\x(r+s))\, dr\\
%& =
%\limsup_{n\to\infty}n\int_{t+s-2/n}^{t+s-1/n}f(\x(r))\, dr
%=[F_{t+s}^\star(f)](\x)
%\end{align*}
%for all $\x\in \cX$.
%\end{proof}

\begin{lem}\label{l.integrals}
Let $f\in C_b(X)$ and $a<b$ then $(F_a^b(f), F_{b}^\star(f)-F_a^\star(f))\in \Df$.
\end{lem}

\begin{proof}
For $t>0$, we have
\begin{align*}
[\Th_tF_a^b(f) - F_a^b(f)](\x) & = \int_a^b f(\x(t+s))\, ds - \int_a^b f(\x(s))\, ds\\
& = \int_b^{b+t} f(\x(s))\, ds - \int_a^{a+t} f(\x(s))\, ds\\
& = \int_0^t f(\x(b+s))\, ds - \int_0^t f(\x(a+s))\, ds\\
& = \int_0^t [\Th_s(F_b(f) - F_a(f))](\x)\, ds\\
& = \int_0^t[\Th_s(F_{b}^\star(f) - F_a^\star(f))](\x)\, ds.
\end{align*}
Here, we have used Lemma \ref{l.fbstar}(b) and the continuity assumption in (P1) in the last equality.
Now Proposition \ref{p.awf} yields $(F_a^b(f), F_b^\star(f)-F_a^\star(f))\in \Df$.
\end{proof}

\begin{defn}\label{d.d0}
We define the operator $\D_0$ as follows: $D(\D_0)$ is the algebra generated by all functions of the form $F_a^b(f)$, 
where $a<b$ and $f\in C_b(X)$ and
\[
\D_0\prod_{j=1}^n F_{a_j}^{b_j}(f_j) = \sum_{k=1}^n(F_{b_k}^\star(f_k) - F_{a_k}^\star(f_k))\prod_{j\neq k} F_{a_j}^{b_j}(f_j).
\]
Then $\D_0$ is a slice (see Definition \ref{def.slice}) of $\Df$ by Lemma \ref{l.integrals} and Proposition \ref{p.derivation}(b).
\end{defn}

\begin{rem}\label{r.domaind0}
\begin{enumerate}
[(a)]
\item 
If $f\in C_b(X)$, then $F_a^b(f)\in C_b(\cX)$ as a consequence of the continuity requirement in (P1). Note, however, that
in general $F_b(f)-F_a(f) \not\in C_b(\cX)$. However, if $\cX = \cXc$, see Section \ref{sub.continuous}, then
$F_a^b(f)\in D(\D)$ for $f\in C_b(X)$ and it follows that $D(\D_0) \subset D(\D)$, i.e.\ it is a subset of the domain of the $C_b$-generator
$\D$.
\item In the general case, there are several choices of functions $G\in B_b(\cX)$ such that
$(F_a^b(f), G) \in \Df$. For the slice $\D_0$, we choose $G=F_{b}^\star(f) - F_a^\star(f)$ and the latter is
$\cF_{b-}$-measurable. Thus, Proposition \ref{p.derivation}(b) generalizes to the slice $\D_0$. The proof of Lemma \ref{l.integrals} shows that another possible choice is $G= F_b(f) - F_a(f)$ which is general not $\cF_{b-}$-measurable. 

If $\cX=\cXd$, see Subsection \ref{sub.cadlag}, then for every $\lambda \in [0,1]$ 
we may choose 
\[
G(\x) = \lambda f(\x(b)) + (1-\lambda)f(\x(b-)) -[ \lambda f(\x(a)) + (1-\lambda)f(\x(a-))].
\]
\end{enumerate}
\end{rem}

Recall that a sequence $(F_n)_{n\in \N} \subset B_b(\cX)$ \emph{bp-converges} to $F\in B_b(\cX)$ if it is uniformly bounded
and and converges pointwise to $F$. Kernel operators are well-behaved with respect to bp-convergence, see Lemma \ref{l.kernelop}. A subset
$M\subset B_b(\cX)$ is called \emph{bp-closed} if for every sequence $(F_n)_{n\in\N}\subset M$ that bp-converges to $F$ it follows
that $F\in M$. The \emph{bp-closure} of a set is the smallest bp-closed set that contains it. If the bp-closure of a set is
all of $B_b(\cX)$, it is called \emph{bp-dense} in $B_b(\cX)$.

\begin{lem}\label{l.algebradense}
$D(\D_0)$ is bp-dense in $B_b(\cX)$.
\end{lem}

\begin{proof}
Denote by $M$ the bp-closure of the algebra $D(\D_0)$. Standard arguments show that $M$ is also an algebra. By the right-continuity of the paths
in $\cX$ it follows that for every $f\in C_b(X)$, it is $nF_t^{t+1/n}(f) \to F_t(f)$ pointwise. Consequently, $F_t(f)\in M$
for all $f\in C_b(X)$ and $t\in \R$. 
Using that $C_b(X)$ is bp-dense in $B_b(X)$, it follows that $F_t(f)\in M$ for all $f\in B_b(X)$
and $t\in \R$. As $M$ is an algebra, for all choices of $t_1 < \ldots < t_n \in \R$ and $A_1, \ldots, A_n\in \bo(X)$ we have
\[
\one_{\displaystyle\{\x\in \cX : \x(t_j) \in A_j \mbox{ for } j=1, \ldots, n\}} = \prod_{j=1}^n F_{t_j}(\one_{A_j})\in M,
\]
i.e.\ indicator functions of cylinder sets belong to $M$. As the cylinder sets form a generator of the Borel $\sigma$-algebra
that is stable under intersections, measure theoretic induction yields $M= B_b(\cX)$.
\end{proof}

We end this section by establishing that the shift group is uniquely determined by the slice $\D_0$. This is
done by generalizing the concept of a core of an operator, see \cite[Definition  II.1.6]{en00}. 
As we are using both the concept of the $C_b$-generator
and the concept of the full generator, there are two different concepts of a core. For the former, the appropriate concept
is that of a $\beta_0$-core (see Lemma \ref{l.cbcore}), for the latter the notion of a bp-core (see Lemma \ref{l.bpcore}) is used instead.

\begin{cor}\label{c.bpcore}
$D(\D_0)$ a bp-core for $\Df$. If $\cX=\cXc$, then $D(\D_0)$ is a $\beta_0$-core for $\D$.
\end{cor}

\begin{proof}
This follows from Corollary \ref{c.slice}. Indeed, noting that 
$\Th_t F_a^b(f) = F_{a+t}^{b+t}(f)$ and $\Th_t F_s(f) = F_{t+s}(f)$ condition (i)
is satisfied, whereas (ii) follows from Lemma \ref{l.algebradense}. Condition (iii) is an immediate consequence of the fact  
that for $f\in C_b(X)$ the map $t\mapsto f\circ\pi_t$ is continuous in almost every point as a consequence of (P1).

Now consider the case of continuous paths.
As $D(\D_0)$ is an algebra that separates the points of
$\cX$, the Stone--Weierstra{\ss} Theorem, see \cite[Theorem 11]{fremlin1972}, yields that $D(\D_0)$ is dense in $C_b(\cX)$
with respect to $\beta_0$. As $\Th_t D(\D_0)\subset D(\D_0)$ for all $t\in \R$, 
it follows from Lemma \ref{l.cbcore} that $D(\D_0)$ is a $\beta_0$-core for $\D$.
\end{proof}

\begin{rem}
\label{r.evaluation}
Even in the case where $\cX=\cXc$, we cannot expect any function of the form $F_0(f)$ (or, more generally, $F_t(f)$) which is
not constant to 
belong to $D(\D)$. To see this, we consider the case $X=\R^d$. Assume that $F_0(f)$ belongs to $D(\D)$. 
Fixing $x_0, v \in \R^d$, we consider $\x(t) = x_0 + |t|v$.  It follows that
\[
[\D F_0(f)](\x) = \lim_{t\downarrow 0} \frac{f(x_0+ tv) - f(x_0)}{t} = \lim_{t\uparrow 0} \frac{f(x_0-tv) - f(x_0)}{t} = -[\D F_0(f)](\x).
\]
This implies that $f$ has in the point $x_0$ directional derivative $0$ in direction $v$. As $x_0$ and $v$ were arbitrary, 
$f$ must be constant.
\end{rem}

\section{Evolutionary semigroups on spaces of measurable functions}\label{sect.measurable}

We are now prepared to introduce the central notions of this article, namely \emph{evolutionary semigroups} and their associated
\emph{expectation operators}. We start with the latter and first prove that several possible defining properties are in fact equivalent.
Throughout this sexction $((\cX, \md), \tau)$ is a path space.

\begin{prop}\label{p.expectation_operator}
Let $\E \in \cL (B_b(\cX), \sigma)$ be a Markovian kernel operator 
with associated kernel  $\k$.
The following are equivalent:
\begin{enumerate}
[\upshape (i)]
\item For every $F\in B_b(\cX)$, the function $\E F$ is $\cF_{0-}$-measurable
and if $F$ is $\cF_{0-}$-measurable, then $\E F = F$.
\item For every $F\in C_b(\cX)$, the function $\E F$ is $\cF_{0-}$-measurable
and if $F$ is $\cF_{0-}$-measurable, then $\E F = F$.
\item Given $A_-\in \cF_{0-}$, $A \in \bo(\cX)$
and $\x\in \cX$,
\[
\k(\x, A_-\cap A) = \delta_{\tau(\x)}(A_-) \k(\tau(\x), A).
\]
\item For every $F\in B_b(\cX)$, the function $\E F$ is $\cF_{0-}$-measurable. Moreover, 
\[
\E (FG) = F \E G,
\]
for all $F,G \in B_b(\cX)$ where $F$ is $\cF_{0-}$-measurable.
\end{enumerate}
\end{prop}

\begin{proof}
(i) $\Rightarrow$ (ii). This is trivial.\smallskip

(ii) $\Rightarrow$ (i). Set
\[
M = \{ F\in B_b(\cX) : \E F \mbox{ is } \cF_{0-}\mbox{-measurable}\}.
\]
By (ii), $C_b(\cX) \subset M$. Moreover, using that $\E$ is $\sigma$-continuous (see Lemma \ref{l.kernelop}), it is easy to see
that if $(F_n)_{n\in \N}$ is a bounded sequence in $M$ that converges pointwise to $F$, then also $F\in M$.
This proves that $M$ is bp-closed. As $\cX$ is Polish, $C_b(\cX)$ is bp-dense in
$B_b(\cX)$ (see \cite[Proposition 3.4.2]{ek}) whence $M= B_b(\cX)$.

Similarly, one also sees that the second property extends to $B_b(\cX)$.\smallskip

(i) $\Rightarrow$ (iii).  As $\E F$ is always $\cF_{0-}$-measurable it follows from Lemma \ref{l.measurability}  that
\[
\k(\x, A) = (\E\one_A)(\x) = (\E\one_A)(\tau(\x)) = \k (\tau(\x), A)
\]
for all $A\in \cB(\cX)$ and $\x\in \cX$.

Now let $A_-\in \cF_{0-}$, $A\in \bo(\cX)$ and $\x\in \cX$ be given. By (i), $\k(\x, A_-) = \E\one_{A_-}(\x) = \one_{A_-}(\x)$.
It follow that for $\x\not\in A_-$
\[
0\leq \k(\x, A_-\cap A) \leq \k(\x, A_-) =0.
\]
On the other hand, if $\x\in A_-$, the same argument shows $\k(\x, A_-^c\cap A) = 0$ and we obtain
\[
\k(\x, A) = \k (\x, A_-\cap A) + \k (\x, A_-^c\cap A) = \k (\x, A_-\cap A).
\]
Altogether,
\[
\k(\x, A_-\cap A) = \k(\tau(\x), A_-\cap A) = \delta_{\tau(\x)}(A_-) \k(\tau(\x), A).
\]

(iii) $\Rightarrow$ (iv). Condition (iii) yields $\E(F G) = F \E G$ whenever $F=\one_{A_-}$ for some $A_-\in \cF_{0-}$ and
$G=\one_{A}$ for some $A \in \bo(\cX)$. Using measure theoretic induction twice, this easily extends to arbitrary functions.\smallskip

(iv) $\Rightarrow$ (i). Let $F\in B_b(\cX)$ be $\cF_{0-}$-measurable. As $\one$ is $\bo(\cX)$-measurable and
$\E\one=\one$ by Markovianity, (iv) yields $\E F = \E (F\one) = F\E\one = F\one = F$.
\end{proof}
Note that condition (iv) of Proposition \ref{p.expectation_operator}
implies that $\E$ is local in the sense that $\E (\one_{A}G) = \one_{A} \E G$ for all $A\in \cF_{0-}$ and $G \in B_b(\cX)$, and can therefore be viewed as a conditional function/expectation, see \cite{MR3451982} and the references therein.

\begin{defn}
\label{d.expectation}
An operator $\E$ satisfying the equivalent conditions of Proposition \ref{p.expectation_operator} is called
\emph{expectation operator}. Given an expectation operator $\E$, we define 
$\E_t \coloneqq \Th_t\E\Th_{-t}$ for $t\ge 0$.
We say that $\E$ is
\emph{homogeneous} if $\E = \E \E_t$ for all $t\geq 0$.
\end{defn}

\begin{rem}\label{rem-SDE1}
In the case of continuous paths, i.e.\ $((\cX, \md), \tau) = ((\cXc, \md_\mathrm{C}), \tau_{\mathrm{C}})$, 
we give an alternative representation of $\E$ that follows from Proposition~\ref{p.expectation_operator}(iii).
To that end, we identify  $\cX^-  \coloneqq \tau(\cX)$ with $ C((-\infty,0];X)$ and define $\cX^+ \coloneqq C([0,\infty); X)$. Then the map
\[ \cX^-\times \cX^+\to \cX,\; (\x_-,\x_+)\mapsto \x_-\oplus_0\x_+ \]
with 
\[ (\x_-\oplus_0\x_+) (t) \coloneqq \begin{cases}
    \x_-(t), & \mbox{if } t\le 0,\\
    \x_+(t) -\x_+(0)+\x_-(0), & \mbox{if } t\ge 0,
\end{cases}\]
is well-defined, and its restriction to the closed subspace of all compatible pairs
\[ 
\cX_{\text{comp}} \coloneqq \{(\x_-,\x_+)\in\cX^-\times\cX^+: \x_-(0)=\x_+(0)\}
\]
is an isomorphism.  
Let $\E \in \cL (B_b(\cX), \sigma)$ be a Markovian kernel operator with associated kernel  $\k$.
From Proposition~\ref{p.expectation_operator}(iii) and identifying $\cX$ and $\cX_{\text{comp}}$, we get
\begin{equation}\label{eq-SDE1}
\k((\x_-,\x_+),A_-\times A_+) = \delta_{\x_-}(A_-) \k_+(\x_-, A_+)
\end{equation}
for $(\x_-,\x_+)\in \cX_{\text{comp}}$ and $A_\pm\in \bo(\cX^\pm)$, 
where we define
\begin{align*}
     \k_+(\x_-,A_+) &:= \k \big(\x_-, \big\{ \x_-\oplus_0 \y_+: \y_+\in A_+,\, \x_-(0) = \y_+(0)\big\}\big)\\
     & = \k\big(\x_-, (\{\x_-\}\times A_+)\cap \cX_{\text{comp}}\big).
\end{align*} 
Note that by \eqref{eq-SDE1}, the measure $\k(\x,\cdot)$ is extended by zero to $(\cX^-\times\cX^+)\setminus \cX_{\text{comp}}$. For $F\in B_b(\cX)$, \eqref{eq-SDE1} and a Fubini argument yield
\begin{equation}\label{eq-SDE2}
    \begin{aligned}
        (\E F)(\x) & = \int_{\cX} F(\y) \k (\x,d\y) 
        = \int_{\cX^-\times\cX^+} F(\y_-\otimes_0 \y_+) 
        \k((\x_-,\x_+), d(\y_-,\y_+)) \\
        & = \int_{\cX^-}\int_{\cX^+} F(\y_-\otimes_0 \y_+) \k_+(\x_-,d\y_+)
        \delta_{\x_-}(d\y_-) \\
        & = \int_{\cX^+} F(\x_-\otimes_0 \y_+) \k_+(\x_-,d\y_+).
    \end{aligned}
\end{equation}

If we want to extend this construction to c\`adl\`ag paths, it is a natural question how to glue two functions from $D([0,\infty); X)$ and
$D((-\infty, 0]; X)$ together at $0$, i.e.\ if there might be a jump at $0$ or not. If the evolutionary semigroup induced by $\E$ is a 
$C_b$-semigroup, then Corollary \ref{c.continuous} below yields that
\[
\P^\x(\{ \y : \y \mbox{ is continuous at 0}\}) =1,
\]
which suggests that we should glue paths $\x_- \in D((-\infty, 0]; X)$ and $\x_+\in D([0,\infty); X)$ continuously together in $0$.
Thus, mutatis mutandis, the same construction  as above can be extended to c\`adl\`ag paths in this situation.

The above description of $\E$ will be used for stochastic (delay) differential equations below, where we  specify $\k_+$ in terms of the unique solution of the equation, see Section~\ref{sect.examples} for details. 
\end{rem}

\begin{lem}\label{l.expectationop}
Let $\E$ be an expectation operator with  associated kernel $\k$. The following are equivalent:
\begin{enumerate}
[\upshape (i)]
\item $\E$ is homogeneous.
\item For every $\x \in \cX$, $t\geq 0$ and $A\in \cB(\cX)$,
\[
\int_{\cX} \k(\vt_t\y, \vt_t A) \k(\x, d\y) = \k(\x, A).
\]
\end{enumerate}
\end{lem}

\begin{proof}
Let $F\in B_b(\cX)$.
Note that

\[
(\E\Th_{-t}F)(\x) = \int_\cX F(\vt_{-t}\y) \k (\x, d\y) = \int_\cX F(\z) \k (\x, d\z\circ(\vt_{-t})^{-1}), 
\]
where we write $\k(\x, d\z\circ(\vt_{-t})^{-1})$ for the push-forward measure of $\k(\x, \cdot)$ under the map $\vt_{-t}$.
Therefore
\[
(\E_t F)(\x) = (\Th_t \E \Th_{-t} F)(\x) = (\E\Th_{-t}F)(\vt_t\x) = \int_\cX F(\z)\k (\vt_t\x, d\z\circ (\vt_{-t})^{-1}).
\]
It follows that
\begin{align*}
(\E \E_t F)(\x) & = \int_\cX (\E_t F)(\y) \k (\x, d\y)\\
& = \int_\cX \int_\cX F(\z) \k (\vt_t \y, d\z\circ (\vt_{-t})^{-1}) \k (\x, d\y)\\
& = \int_\cX F(\z)\int_\cX \k (\vt_t \y, d\z\circ(\vt_{-t})^{-1}) \k (\x, d\y).
\end{align*}
This is the same as $(\E F)(\x)$ for all $F\in B_b(\cX)$ if and only if
\[
\int_\cX \one_A(\z)\int_\cX \k(\vt_t\y, d\z\circ(\vt_{-t})^{-1}) \k(\x, d\y) =\int_\cX \k(\vt_t \y, \vt_tA) \k(x, d\y) = \k(\x, A)
\]
for all $\x\in \cX$ and $A\in \cB(\cX)$.
\end{proof}

\begin{thm}\label{t.expectation_semigroup}
Let $(\T (t))_{t\geq 0}$ be a semigroup of Markovian kernel operators on $\bm$.
 The following are equivalent:
\begin{enumerate}
[\upshape (i)]
\item For every $F\in \bm$ and $t\geq 0$, it is $\T(t)\Th_{-t}F =F$. In particular,
if $F$ is determined before $-t_0 < 0$, then $\T(t)F = \Th_t F$ for all $0\leq t \leq t_0$.
\item There exists a homogeneous expectation operator $\E$ such that
\begin{equation}\label{eq.representation}
\T(t) =\E \Th_t\quad \mbox{for all } t\geq 0.
\end{equation}
\end{enumerate}
Conversely, if $\E$ is a homogeneous expectation operator, then Equation \eqref{eq.representation}
defines a semigroup of Markovian kernel operators on $\bm$, which satisfies condition {\upshape (i)}.
\end{thm}

\begin{proof}
(i) $\Rightarrow$ (ii). Fix $s\geq 0$. If  $F$ is determined before $s$, Lemma \ref{l.measurability} implies that
$\Th_{-s}F$ is determined before $0$. Thus, the expression $\T(s)\Th_{-s}F$ is well-defined
and $F\mapsto \T(s)\Th_{-s}F$ defines a $\sigma$-continuous mapping from $B_b(\cX, \cF_{s-})$ to $B_b(\cX, \cF_{0-})$.
Consequently, we find a kernel $\k_s$ such that
\[
\T(s)\Th_{-s}F(\x) = \int_\cX F(\y)\k_s(\x, d\y)
\]
for all $F\in B_b(\cX, \cF_{s-})$. It follows from Lemma \ref{l.measurability} that for $F\in B_b(\cX, \cF_{s-})$ the function
 $\Th_{-(r+s)}F$ is determined before $-r <0$.
Consequently, (i) implies that $\Th_r\Th_{-(r+s)}F = \T(r)\Th_{-(r+s)}F$ and thus
\[
\T(s)\Th_{-s}F = \T(s)\Th_r\Th_{-(r+s)}F = \T(s)\T(r)\Th_{-(r+s)}F = \T(s+r)\Th_{-(r+s)}F.
\]
This shows that $\k_s(\x, A) = \k_{s+r}(\x, A)$
for all $r\geq 0$ and $A\in \cF_{s-}$. We may thus define
\[
\k(\x, A) \coloneqq \k_s(\x, A)
\]
whenever $A\in \cF_{s-}$. In this way, $\k (\x, \cdot)$ defines a finitely additive measure on $\cD = \bigcup_{s\geq 0}\cF_{s-}$.\smallskip

Note that $\k_s(\x, \cdot)$ is even $\sigma$-additive on $\cF_{s-}$ and thus can be viewed
as a Radon measure 
on the Polish space $\cX_s = \tau^s(\cX)$. Consequently, given $\eps>0$ and  $A\in \cF_{s-}$, we find a 
compact subset $K\subset \cX_s$ such that $\k_s(\x, A\setminus K) \leq \eps$. Note that $K$ is also a compact subset of $\cX$. 
This shows that given $A\in \cD$ and $\eps>0$ we find a compact
subset $K$ of $\cX$ that also belongs to $\cD$ such that $\k(\x, A\setminus K)\leq \eps$. At this point
\cite[Theorem\ 1.4.3]{bogachev2007} implies that $\k(\x, \cdot)$ is $\sigma$-additive on $\cD$, whence it can be
extended to a measure on $\bo(\cX)$ by means of the Carath\'{e}odory theorem (see \cite[Cor.\ 1.11.9]{bogachev2007}).

We note that the map $\x\mapsto \k(\x, A)$ is measurable whenever $A\in \cD$. By a monotone class argument, this
extends to arbitrary $A\in \bo(\cX)$, proving that $\k$ is a kernel. We may thus define
\[
[\E F](\x) \coloneqq \int_{\cX} F(\y) \k (\x, d\y)
\]
for all $F\in B_b(\cX)$. By construction, $\E F = \T(s)\Th_{-s} F$ whenever
$F$ is determined before $s\geq 0$. Applying this to $F= \Th_sG$ for some $G\in \bm$ 
yields $\T(s)G = \E\Th_sG$ for all $G\in \bm$. In particular,  $\E G = G$
for $G\in \bm$. 

On the other hand, as $\T(s)\Th_{-s}F$ is $\cF_{0-}$-measurable for all $F\in B_b(\cX, \cF_s)$ it follows that
$\k_s(\x, \cdot) = \k_s(\tau(\x), \cdot)$. As $s$ is arbitrary, $\k(\x, A) = \k(\tau(\x), A)$ whenever $A\in \cD$. A monotone class argument
extends this to $A\in \bo(\cX)$ and it follows that $\E F$ is $\cF_{0-}$-measurable for all $F\in B_b(\cX)$. Thus,
condition (i) of Proposition \ref{p.expectation_operator} is satisfied, proving that $\E$ is an expectation operator.\smallskip

It remains to prove that $\E$ is homogeneous. To that end, note that if $F\in \bm$, then
\[
\E\Th_t\E\Th_s F = \T(t)\T(s)F = \T(t+s)F = \E\Th_t\Th_sF
\]
for all $t,s\geq 0$. Putting $G=\Th_sF$, it follows from Lemma \ref{l.measurability} that $G$ is determined before $s$ and that
any $G$ determined before $s$ can be written in this form. This implies that $\E\Th_t\E G = \E\Th_t G$ for every $G$
that is determined before some time $s\geq 0$ and thus $\E\Th_t\E\one_A = \E\Th_t\one_A$ for
every $A\in \cD \coloneqq \bigcup_{s>0}\cF_{s-}$. As this is a generator for $\bo(\cX)$,
measure theoretic induction yields $\E\Th_t\E = \E\Th_t$ proving that $\E$ is homogeneous.\medskip

(ii) $\Rightarrow$ (i).  If $\T(t)$ is defined by Equation \eqref{eq.representation}, then
for $F\in \bm$
\[
\T(t)\Th_{-t} F = \E\Th_t\Th_{-t} F = \E F = F
\]
since $\E$ is an expectation operator. This proves (i). For the addendum observe that if $\T$ is not a priori assumed
to be a semigroup but given by Equation \eqref{eq.representation}, then the semigroup law follows from 
the homogeneity of $\E$. Indeed, as $\E$ is homogeneous,  $\E\Th_t\E = \E\Th_t$ for every $t\geq 0$ and hence
\[
\T(t)\T(s) = \E\Th_t\E\Th_s = \E\Th_t\Th_s = \E\Th_{t+s} =\T(t+s)
\]
for all $t,s\geq 0$.
\end{proof}

\begin{defn}
A semigroup $\T$ of Markovian kernel operators on $\bm$ is called \emph{evolutionary} if it satisfies the equivalent conditions of Theorem \ref{t.expectation_semigroup}. In this case, the operator $\E$ is called the \emph{associated expectation operator} of $\T$.
\end{defn}

We next show that the operator $\E_t$ can be interpreted as `conditional expectations given $\cF_{t-}$'. Moreover, we obtain a property similar
to the Markov property for evolutionary semigroups.

\begin{prop}\label{p.markov}
Let $(\T(t))_{t\geq 0}$ be an evolutionary semigroup with expectation operator $\E$. We denote the kernel of $\E$ by
$\k$ and write $\P^\x$ for the probability measure $\k(\x, \cdot)$ and $\E^\x$ for the (conditional) expectation with respect to
$\P^\x$.
\begin{enumerate}
    [\upshape (a)]
    \item For every $F\in \B_b(\cX)$ and $t\geq 0$, it is
    \[
    \big[\E^\x[F|\cF_{t-}]\big](\y) = [\E_t F](\y) \quad \mbox{for } \P^\x\mbox{-a.e. } \y.
    \]
    \item For every $s,t\ge 0$, 
    $F\in  \bm$ and $\x\in\cX$, we have
    \[
    \big[\E^{\x}[ \Th_{t+s} F \big| \cF_{t-}]\big](\y) = \big[\T(s)F\big](\vt_t\y)\quad \mbox{for } \P^\x\mbox{-a.e. } \y.
    \]
\end{enumerate}
\end{prop}

\begin{proof}
(a). Fix $F\in B_b(\cX)$ and $t\geq 0$. If $A\in \cF_{t-}$, then $\one_A \in B_b(\cX, \cF_{t-})$ and thus $\Th_{-t}\one_A \in \bm$. By Proposition \ref{p.expectation_operator}(iv), we have
\[
\E [\Th_{-t}\one_AF] = \E [(\Th_{-t}\one_A)(\Th_{-t}F)] = (\Th_{-t}\one_A)\E [\Th_{-t}F].
\]
Applying $\Th_t$ on both sides, it follows that $\E_t (\one_A F) = \one_A \E_t F$. Using this, we see that for $\x\in \cX$
\begin{align*}
    \E^\x[\one_A F] & = [\E(\one_A F)](\x) = [\E \E_t (\one_A F)](\x) = \E[\one_A\E_t F](\x) = \E^\x[\one_A \E_tF].
\end{align*}
Here, we have used homogenity of $\E$ in the second equality. As $A\in \cF_{t-}$ was arbitrary, (a) is proved.

(b). Fix  $s,t\ge 0$ and $F\in  \bm$. It follows from (a) that  for $\P^\x\mbox{-a.e. } \y$,
\[
\big[\E^\x[\Th_{t+s}F|\cF_{t-}]\big](\y) = [\E_t(\Th_{t+s}F)](\y) = [\Th_t\E\Th_s F](\y) = [\T(s)F](\vt_t \y).\qedhere
\]
\end{proof}

As already mentioned in the introduction, in (stochastic) delay equations the case of finite delay $h>0$ is of particular interest
and in this case, one often uses a space of functions on the interval $[-h, 0]$ as a state space. In our more general
framework we identify functions on the interval $[-h, 0]$ with functions on $\cX^-$ that are
measurable with respect to $\cF([-h,0))$.

\begin{lem}\label{l.invariance}
Let $\T$ be an evolutionary semigroup with expectation operator $\E$ and $h>0$. The following are equivalent:
\begin{enumerate}
[\upshape (i)]
\item $\T(t)\bb{[-h,0)}\subset \bb{[-h,0)}$ for all $t>0$.
\item $\E  \bb{[0,\infty)} \subset \bb{[-h,0)}$.
\end{enumerate}
\end{lem}

\begin{proof}
(i) $\Rightarrow$ (ii).  We prove $\E F \in \bb{[-h,0)}$ for $F\in \bb{[0,\infty)}$ in several steps.

\emph{Step 1.} We prove the assertion for $F= F_t(f)$ with $t\geq 0$ and $f\in C_b(X)$.

To that end, fix $t\geq 0$ and $f\in C_b(X)$. For every $s>0$ we have
\[
\E F_{t+s}^\star(f) = \E \Th_{t+s}F_0^\star(f) = \T(t+s)F_0^\star(f) \in \bb{[-h,0)}
\]
by (i), as $F_0^\star{f} \in \bb{[-h,0)}$. For every $\x\in \cX$ it is $[F_{t+s}^\star(f)](\x) =[F_{t+s}(f)](\x)$ for almost
every $s>0$. By right continuity of the paths, 
$F_{t+s}^\star(f)$ converges pointwise to $F_t(f)$ as $s\to 0$. Since $\E$ is bp-continuous and $\bb{[-h,0)}$
is bp-closed, it follows that $\E F_t(f) \in \bb{[-h,0)}$ as claimed. Note that this shows in particular
that $F_0^\star(f) = \E F_0^\star(f) = \E F_0(f)$.
\smallskip

\emph{Step 2.} We prove the assertion for $F = \prod_{j=1}^n F_{t_j}(f)$ with $0\leq t_1 < \ldots < t_n$ and $f_1, \ldots, f_n \in C_b(X)$.

This is proved by induction over $n$. The case $n=1$ is exactly Step 1. Assuming the claim to be proved for the product of $n$ functions,
let $0\leq t_1 < \ldots < t_n < t_{n+1}$ and $f_1, \ldots f_{n+1}\in C_b(X)$ be given. We put $G= \prod_{j=2}^{n+1}F_{t_j}(f_j)$, 
so that $F = \prod_{j=1}^{n+1}F_{t_j}(f_j)=F_{t_1}(f_1)G$. It follows that
\begin{align*}
\E F & = \E F_{t_1}(f_1)G = \E \Th_{t_1} F_0(f_1)\Th_{-t_1}G = \E \Th_{t_1}\E F_0(f_1)\Th_{-t_1}G\\
& = \E \Th_{t_1}\E F_0^\star (f_1)\Th_{-t_1}G = \E \Th_{t_1}F_0^\star (f_1)\E\Th_{-t_1}G\\ 
& = \T(t_1)\big[F_{0}^\star(f_1)\E\Th_{-t_1} G\big] \in \bb{[-h,0)}.
\end{align*}
In the above calculation, the third equality uses that $\E$ is homogeneous (whence $\E\Th_{t_1}\E = \E\Th_{t_1}$), the fourth follows from Step 1
and the fifth from Lemma \ref{l.expectationop}(iv). By induction hypothesis, $\E\Th_{-t_1}G \in \bb{[-h,0)}$
so that also $F_{0}^\star(f_1)\E\Th_{-t_1} G \in \bb{[-h,0)}$. Using our assumption (i) again, 
it follows that indeed $\E F\in \bb{[-h,0)}$.\smallskip

\emph{Step 3.} We prove the general case.

By a bp-closedness argument, it follows that the assertion of Step 2 is still valid for $f_1, \ldots, f_n \in B_b(X)$. In particular,
 $\E F \in \bb{[-h,0)}$ whenever $F=\one_{A}$ where $A$ is a cylinder set of the form
 \[
 A =\{ \x : \x(t_j) \in A_j \text{ for } j=1,\dots,n\}
 \]
for some $0\leq t_1 < \ldots < t_n$ and $A_1, \ldots, A_n \in \bo(X)$. As these sets are a generator of $\cF([0,\infty))$ which is stable under
intersections, another bp-closedness argument shows that $\E \one_A \in \bb{[-h,0)}$ for all $A\in \cF([0,\infty))$ and the general
case follows from linearity and yet another bp-closedness argument.\medskip

(ii) $\Rightarrow$ (i). Let $F= \prod_{j=1}^nF_{t_j}(f_j)$ where $-r\leq t_1< \ldots < t_n < 0$ and $f_1, \ldots f_n \in B_b(X)$.
Note that $\Th_tF = \prod_{j=1}^n F_{t_j+t}(f_j)$. We pick the index $k$ such that $t_k+t < 0 \le t_{k+1} +t$. It follows that
\[
\T(t)F = \E\Th_{t}F = \E \prod_{j=1}^n F_{t_j+t}(f) = \prod_{j=1}^k F_{t_j+t}(f)\E \prod_{j=k+1}^n F_{t_j+t}(f)
\]
by Proposition \ref{p.expectation_operator}(iv). At this point, (ii) implies $\T(t)F \in \bb{[-r,0)}$. 
As $\cF ([-h,0)) =\sigma (\pi_t : - h\leq t < 0)$,
measure theoretic induction shows $\T(t)F \in \bb{[-h,0)}$ for all $F\in \bb{[-h,0)}$.
\end{proof}

So far, we have not imposed any measurability assumption on the orbits of an evolutionary semigroup. 
As it turns out, it is automatically satisfied.

\begin{lem}\label{l.measurableorbits}
Let $\T$ be an evolutionary semigroup with expectation operator $\E$. Then for every $F\in B_b(\cX, \cF_{0-})$, the map
$(t, \x)\mapsto [\T(t)F](\x)$ is measurable.
\end{lem}

\begin{proof}
First consider $F\in C_b(\cX, \cF_{0-})$. As $(\Th_t)_{t\geq 0}$ is a $C_b$-semigroup by Proposition \ref{p.derivation}, if $t_n\to t$, then
$\Th_{t_n}F \to \Th_tF$ pointwise. As $\E$ is a kernel operator, it follows that 
\[
[\T(t_n)F](\x) = [\E\Th_{t_n}F](\x) \to [\E\Th_tF](\x) = [\T(t)F](\x)
\]
for every $\x\in \cX$. This shows that for fixed $\x\in \cX$ the map $t\mapsto \psi(t,\x)\coloneqq [\T(t)F](\x)$ is continuous.
At the same time, for fixed $t\in [0,\infty)$, the map $\x\mapsto \psi (t, \x)$ is measurable. These two facts imply that
$\psi$ is product measurable. Indeed, setting
\[
\psi_n(t,\x) = \sum_{j=0}^{n 2^n} \one_{[\frac{j}{2^n}, \frac{j+1}{2^n})}(t) \psi\Big(\frac{j}{2^n}, \x\Big),
\]
we see that $\psi_n$ is product measurable and converges pointwise to $\psi$.\smallskip

Next, consider the set 
\[
M\coloneqq \{ F\in B_b(\cX) : (t,\x) \mapsto [\T(t)(F\circ\tau)](\x) \mbox{ is measurable }\}.
\]
By the above, $C_b(\cX) \subset M$. A moments thought shows that $M$ is bp-closed. As $C_b(\cX)$
is bp-dense in $B_b(\cX)$ (see \cite[Proposition 3.4.2]{ek}), it follows that $M=B_b(\cX)$. As $F\circ\tau\in B_b(\cX, \cF_{0-})$
for every $F\in B_b(\cX)$ and every element of $B_b(\cX, \cF_{0-})$ is of this form (by Lemma \ref{l.measurability}),
the claim follows.
\end{proof}

It follows from Lemma \ref{l.measurableorbits} that we are very close to the notion of transition semigroup, see Definition \ref{d.transsg}. 
However, our setting is slightly different from that considered there, as $\cF_{0-}$ is not the Borel $\sigma$-algebra
of $\cX$. It is also clear, at least on an intuitive level, that $B_b(\cX, \cF_{0-})$ does not separate the points in $\cM(\cX)$, the space of all bounded measures on $\cX$. Note that
the duality between measurable functions and measures lies at the heart of the theory of transition semigroups. Thus, at first glance,
we cannot use the theory presented in the appendix. 

However, by means of the extension map $F\mapsto \hat{F}=F\circ \tau $, we can identify $B_b(\cX, \cF_{0-})$ with the space
$B_b(\cX^-, \bo(\cX^-))$, see Lemma \ref{l.eminusf0}. By means of this identification, we are in the situation considered in
the appendix, namely we work on the norming dual pair $(B_b(\cX^-), \cM(\cX^-))$. To switch from $\cX$ to $\cX^-$, we formally
also have to change the semigroup. To wit, if $\T$ is an evolutionary semigroup, we define 
\begin{equation}\label{eq.modified}
[\T^-(t)F](\x) \coloneqq [\T(t)\hat{F}](\x) \quad \mbox{for all }F\in B_b(\cX^-)\mbox{ and } x\in \cX^-.
\end{equation}
It then follows, that $\T^-$ defines a transition semigroup on $B_b(\cX^-)$ and we may talk of its Laplace transform, its full generator, continuity properties etc. However, in order not to overburden notation, we will not distinguish between the semigroup $\T$ and the modified
semigroup $\T^-$.\smallskip

We may now describe the full generator $\Af$ of an evolutionary semigroup in terms of the expectation
operator $\E$ and the full generator $\Df$ of the shift semigroup. It turns out that $\Af$ is already uniquely determined
by the expectation operator $\E$ and the slice $\D_0$ from Definition \ref{d.d0}. We use the following notation:

\[
D_-(\D_0) = D(\D_0)\cap B_b(\cX, \cF_{0-})\mbox{ and } D_+(\D_0) = D(\D_0)\cap B_b(\cX, \cF([0,\infty))).
\]

\begin{prop}\label{p.bpcorefora}
Let $\T$ be an evolutionary semigroup with expectation operator $\E$ and full generator $\Af$.
As usual, we denote the full generator of the
shift group by $\Df$. Then:
\begin{enumerate}
[\upshape (a)]
\item It is $(F,G)\in \Af$ if and only if there exits $(U,V)\in \Df$ with $F=\E U$ and $G= \E V$.
\item Define $\A_0$ by
\[
D(\A_0) \coloneqq \Big\{ \sum_{j=1}^n U_j \E V_j : n\in \N, U_j \in D_-(\D_0),
V_j \in D_+(\D_0) \mbox{ for } j=1, \ldots, n\Big\}
\]
and  
\[
\A_0\sum_{j=1}^n U_j\E V_j = \sum_{j=1}^n \big[(\D_0 U_j) \E V_j + U_j \E \D_0 V_j\big].
\]
Then $\A_0$ is a slice of $\Af$ and $D(\A_0)$ is a bp-core for $\Af$.
\item $D(\A_0)$ (and thus $D(\Af)$) is bp-dense in $B_b(\cX, \cF_{0-})$. In view of Lemma \ref{l.fullgenprop}, $\T$ is uniquely
determined by $\Af$.
\end{enumerate}
\end{prop}

\begin{proof}
(a) As $\E$ is homogeneous, it is $\E\Th_t = \E\Th_t\E$ for all $t\geq 0$. By Proposition \ref{p.awf}, we have
$(U,V)\in \Df$ if and only if 
\[
\Th_t U - U = \int_0^t\Th_s V\, ds
\]
for all $t\geq 0$. Applying $\E$ to this equality and using the homogeneity yields
\[
\E\Th_t\E U - \E U = \T(t)\E U - \E U = \int_0^t \E\Th_s\E V\, ds = \int_0^t \T(t)\E V\, ds.
\]
By Proposition \ref{p.awf}, this shows that $(F,G) = (\E U, \E V) \in \Af$.

To see the converse, first observe that the resolvents of $\Af$ and $\Df$ are related via
\begin{align*}
R(\lambda, \Af) F = \int_0^\infty e^{-\lambda t} \E \Th_t F\, dt = \E \int_0^\infty e^{-\lambda t} \Th_t F\, dt = \E R(\lambda, \Df)F.
\end{align*}
Consequently, $(F,G)\in \Af$ if and only if
\[
F = R(\lambda, \Af)(\lambda F-G) = \E R(\lambda, \Df)(\lambda F -  G) = \E U,
\]
for $U= R(\lambda, \Df)(\lambda F-  G) \in D(\Df)$. Note that in this case for $V= \lambda U - \lambda  F + G$ we have
$(U,V)\in \Df$. Moreover, $\E V =  G$ as $\E F = \E^2 U = \E U$ by construction and $\E G =  G$ since $ G$ is $\cF_{0-}$-measurable.\medskip

(b) 
Noting that if $a<0<c$, then $F_a^b(f) = F_a^0(f) + F_0^b(f)$, where $F_a^0(f)$ is $\cF_0$-measurable
and $F_0^b(f)$ is $\cT_0$-measurable, it is easy to see
that any element $U$ of $D(\D_0)$ can be written in the form $U=\sum_{j=1}^n U_jV_j$, where $U_j\in D_-(\D_0)$ and
$V_j\in D_+(\D_0)$ for $j=1, \ldots, n$. Using part (a) and Proposition \ref{p.expectation_operator} (taking
Remark \ref{r.domaind0} into account), it follows that $\A_0$ is a slice of $\Af$.

Using that $D(\D_0)$ is a bp-core for $\Df$ by Corollary \ref{c.bpcore}, it follows from part (a) and the $\sigma$-continuity of
$\E$ that $\A_0$ is a bp-core for $\A$.\medskip

(c) It follows from (b) that $D_-(\D_0) \subset D(\A_0)$. The proof now follows along the lines of that of Lemma \ref{l.algebradense}.
\end{proof}

\section{Continuity properties of evolutionary semigroups}\label{sect.continuous}

In this section, we study the question, whether an evolutionary semigroup actually is a $C_b$-semigroup. The answer depends on the path space $\cX$. 
The case where $\cX=\cXc$ is significantly easier than the general case and
also yields better results. This is due to the fact that both the evaluation maps $\pi_t$ and the stopping map $\tau$ are continuous
in this case. The former yields that the operator $\D_0$ from Definition \ref{d.d0} is not only a slice of $\Df$, but actually a restriction
of the $C_b$-generator $\D$, since in this case $\D_0F_a^b(f) = F_b(f)-F_a(f)\in C_b(\cX)$ whenever $f\in C_b(X)$. That $\tau$ is continuous
implies that the extension map $F\mapsto \hat{F}$ does not only establish an isomorphism between $B_b(\cX^-)$ and $B_b(\cX, \cF_{0-})$,
but also an isomorphism between $C_b(\cX^-)$ and $C_b(\cX, \cF_{0-})$.\smallskip

In the case $\cX=\cXd$ of c\`adl\`ag paths, functions of the form $F=F_t(f)$ for $f\in C_b(X)$ are not continuous on $\cX$, whence $\D_0$ is
not a restriction of the $C_b$-generator in this case. Moreover, for $t=0$ the function $F=F_0(f)$ belongs to $C_b(\cX^-)$ but its extension
$\hat{F}$ is not continuous. To overcome this problem, we will impose additional assumptions on the path space (which are satisfied for
c\`adl\`ag paths in the $J_1$-topology) which allow us to generalize at least part of our results.

\subsection{The case of continuous paths}

Throughout this subsection, we consider the case $((\cX, \md), \tau) = ((\cXc, \md_\mathrm{C}), \tau_{\mathrm{C}})$ from
Subsection \ref{sub.continuous}. However, in order to not overburden notation, we will use our generic notation. 
Note that in this case $\cF_{t-} = \sigma(\pi_s : s<t) = \sigma (\pi_s : s\leq t) = \cF_t$.\smallskip

Before stating and proving our first result, we need a little preparation. Let
us briefly recall the notion of a \emph{convergence determining set}, see, e.g.\ \cite[p.\ 112]{ek}. If $(S,d)$ is a metric space, and
$\mu_n, \mu$ are Borel measures on $S$, then $\mu_n$ converges weakly to $\mu$ if $\int_S f\, d\mu_n \to \int_S f\, d\mu$
for every $f\in C_b(S)$; we write $\mu_n \weak \mu$ in this case. A subset $M\subset C_b(S)$ is called \emph{convergence determining}
if convergence $\int_S f\, d\mu_n \to \int_S f\, d\mu$ for all $f\in M$ implies $\mu_n \weak \mu$.

A set $M \subset C_b(S)$ is said to \emph{strongly separate points} in $S$ if, given $x\in S$ and $\delta >0$, there exist
$\{f_1, \ldots, f_n\} \subset M$ such that 
\[
\inf_{y : d(x,y)\geq \delta} \max_{k=1, \ldots, n}|f_k(x)-f_k(y)|> 0.
\]
It follows from \cite[Theorem  3.4.5]{ek} that any algebra that strongly separates points in $S$ is convergence determining.

\begin{lem}\label{l.contdet}
The sets
\[
\mathscr{D} \coloneqq \bigcup_{t>0} C_b(\cX, \cF_t)
\]
and
\[
\mathscr{P}\coloneqq \Big\{ \sum_{j=1}^n F_j G_j : n\in \N, F_j \in C_b(\cX, \cF_0), G_j \in C_b(\cX, \cF([0,\infty)))\Big\}.
\]
are convergence determining for $\cX$.
\end{lem}

\begin{proof}
Fix $\x \in \cX$ and $\delta>0$. 
Pick $n_0$ so large, that $2^{-n_0}< \delta/2$ and put
\[
F(\y) \coloneqq \sup_{t\in [-n_0, n_0]} d(\x(t), \y(t)).
\]
Obviously, $F\in C_b(\cX, \cF_{n_0}) \subset \mathscr{D}$ and $F(\x) = 0$. On the other hand, if $\md(\x,\y) >\delta$, 
we must have $F(\y) \geq \delta/2$ as otherwise $\md(\x,\y) <\delta$ by choice of $n_0$. This shows that $\mathscr{D}$
strongly separates points. 
As $\mathscr{D}$ is also an algebra, \cite[Theorem\ 3.4.5]{ek} yields the claim. The proof for $\mathscr{P}$ is similar,
considering the  functions 
\[
F(\y) \coloneqq \sup_{t\in [-n_0, 0]}d(\x(t), \y(t))\quad\mbox{and}\quad G(\y) \coloneqq \sup_{t\in [0,n_0]}d(\x(t), \y(t)).\qedhere
\]
\end{proof}

We are now ready to present the main result of this subsection. In its formulation (and also in the rest of this subsection), 
we once again do not distinguish between an evolutionary semigroup $\T$ and its modification $\T^-$ on $B_b(\cX^-)$ defined
by \eqref{eq.modified}. This in particular concerns parts (a) and (b) of the theorem.

\begin{thm}\label{t.continuouscb}
Let $\T$ be an evolutionary semigroup with
expectation operator $\E$ and $C_b$-generator $\A$. The following are equivalent:
\begin{enumerate}
[\upshape (i)]
\item $\T(t)C_b(\cX, \cF_0) \subset C_b(\cX, \cF_0)$ for all $t>0$;
\item $\E C_b(\cX) \subset C_b(\cX, \cF_0)$;
\item $\E C_b(\cX, \cF([0,\infty))) \subset C_b(\cX, \cF_0)$.
\end{enumerate}
If these equivalent conditions are satisfied, the following hold:
\begin{enumerate}
[\upshape (a)]
\item $\T$ induces (in the sense of \eqref{eq.modified}) a $C_b$-semigroup on $C_b(\cX^-)$;
\item The operator $\A_0$ from Proposition \ref{p.bpcorefora} is a restriction of $\A$.
Moreover, $D(\A_0)$ is a $\beta_0$-core for $\A$.
\end{enumerate}
\end{thm}

\begin{proof}
Throughout the proof, we denote the kernel of $\E$ by $\k$ and write $\P^\x$ for the probability
measure $\k(\x, \cdot)$.

(i) $\Rightarrow$ (ii). As $\E F = \T(t)\Th_{-t}F$ whenever $F$ is determined before $t$, it follows from (i) 
that $\E F \in C_b(\cX, \cF_0)$
whenever $F$ belongs to the set $\mathscr{D}$ from Lemma \ref{l.contdet}.
This means that if $(\x_n)\subset \cX^-$ converges
to $\x\in \cX^-$, then 
\begin{equation}\label{eq.convergenced}
\int_\cX F(\y)\P^{\x_n}(d\y) = [\E F](\x_n) \to [\E F](\x) = \int_\cX F(y)\P^\x(d\y)
\end{equation}
for every $F\in \mathscr{D}$. As $\mathscr{D}$ is convergence determining by Lemma \ref{l.contdet}, 
Equation \eqref{eq.convergenced} holds true for every $F\in C_b(\cX)$, i.e.\ $\E F$ is continuous on $\cX^-$ for every $F\in C_b(\cX)$.
As $\tau$ is a continuous map taking values in $\cX^-$ and $[\E F](\x) = [\E F](\tau(\x))$ for all $\x\in \cX$, (ii) follows.\smallskip

(ii) $\Rightarrow$ (iii). This implication is trivial.\smallskip

(iii) $\Rightarrow$ (i). If $F\in C_b(\cX, \cF_0)$ and
$G\in C_b(\cX, \cF([0,\infty))$, then it follows from Proposition
\ref{p.expectation_operator}, that $\E(FG) = F\E G$ and (iii) shows
that $\E(FG)\in C_b(\cX, \cF_0)$. It follows that 
$\E H \in C_b(\cX, \cF_0)$
for every
Element $H$ in the set $\mathscr{P}$ from Lemma \ref{l.contdet}.
As this set is convergence determining,
similar arguments as in the proof of (i) $\Rightarrow$ (ii) show
that $\E F\in C_b(\cX, \cF_0)$ for every $F\in C_b(\cX)$, i.e.\ (ii)
holds true. That (ii) implies (i) follows immediately from the representation $\T(t)F = \E\Th_t F$.
\medskip

Now assume that the equivalent conditions are satisfied.

(a). Let $F\in C_b(\cX, \cF_0)$, $(t_n)\subset [0,\infty)$ converge to $t$ and $(\x_n)\subset \cX^-$ converge to $\x$. 
By Proposition \ref{p.derivation} and Theorem \ref{t.beta0char}, 
$\Th_{-t_n}F \to \Th_{-t}F$ with respect to $\beta_0$. By (ii), $\P^{\x_n} \to \P^\x$ 
with respect to $\sigma(\cM(\cX), C_b(\cX))$. It follows that
\[
[\T(t_n)F](\x_n) = \int_\cX [\Th_{-t_n}F](\y)\P^{\x_n}( d\y) \to \int_{\cX} [\Th_{-t}F](\y)\P^\x(d\y) = [\T(t)F](\x).
\]

(b). 
Note that, in the case of continuous paths, we actually have $\D_0 \subset \D$.
By the construction in Proposition \ref{p.bpcorefora} this fact together with (ii) shows that $D(\A_0) \subset C_b(\cX, \cF_0)$
and $\A_0 F \in C_b(\cX, \cF_0)$ for every $F\in D(\A_0)$. Consequently, $\A_0\subset \A$.

To prove that $D(\A_0)$ is a $\beta_0$-core for $\A$, let $F\in D(\A)$ be given. By Proposition \ref{p.bpcorefora}(a),
there is some $(U,V)\in D(\Df)$ with $F=\E U$ and $\A F = \E V$. Inspecting the proof of that proposition and noting
that $R(\lambda, \Df)$ maps $C_b(\cX)$ to itself, we see that we may choose continuous $U,V$. This implies
$U\in D(\D)$ and $V= \D U$. By the above, there is a net $(U_\alpha)\subset D(\D_0)$ with $U_\alpha \to U$ 
and $\D U_\alpha \to \D U = V$ with respect to $\beta_0$. It follows from (ii) that $\E$ is $\beta_0$-continuous
and thus $F_\alpha \coloneqq \E U_\alpha \to F$ and $\A F_\alpha = \E \D U_\alpha \to \E V = \A F$ with respect to
$\beta_0$.
\end{proof}

\begin{defn}
An \emph{evolutionary $C_b$-semigroup} is an evolutionary semigroup that satisfies the equivalent conditions
of Theorem \ref{t.continuouscb}.
\end{defn}

\begin{prop}\label{p.generator}
Let $\cX=\cXc$ and $(\T(t))_{t\geq 0}$ be a $C_b$-semigroup
on $C_b(\cX^-)$ with $C_b$-generator $\A$. Then $\T$ is induced by an evolutionary semigroup if and only if
$D_-(\D_0)\subset D(\A)$ and $\A U = \D U$ for $U\in D_-(\D_0)$.
\end{prop}

\begin{proof}
Assume that $D_-(\D_0) \subset D(\A)$ and $\A U = \D U$ for $U\in D_-(\D_0)$.
Fix $U\in D_-(\D_0)$. As $D(\D_0)$ is invariant under the shift semigroup, $\Th_{-t} U\in D_-(\D_0) \subset D(\A)$ for all $t>0$.
Define $\varphi: [0,\infty) \to C_b(\cX^-)$ by setting
$\varphi (t) = \T(t)\Th_{-t}  U$. 

Now, fix $t\geq 0$ and note that
\[
\frac{\varphi(t+h) - \varphi (t)}{h} = \T(t+h) \frac{\Th_{-(t+h)} U - \Th_{-t}U}{h} + \frac{\T(t+h)- \T(t)}{h}\Th_{-t}U.
\]
Since $\Th_{-t}U \in D(\A)$ the last term $\beta_0$-converges to
$\T(t)\A \Th_{-t} U = \T(t)\D\Th_{-t}U$ as $h\to 0$.

Next, observe that
\[
\Delta_hU\coloneqq \frac{\Th_{-(t+h)}U - \Th_{-t}U}{h} \to - \D\Th_{-t}U
\]
with respect to $\beta_0$. Now, let $\mu \in \mathscr{M}(\cX^-)$ be given. As the map $r\mapsto \T(r)'\mu$ is
$\sigma (\mathscr{M}(\cX^-), C_b(\cX^-))$-continuous, the set
\[
S\coloneqq \{ \T(r)'\mu : r\in [t-1, t+1]\cap [0,\infty) \}
\]
is $\sigma (\mathscr{M}(\cX^-), C_b(\cX^-))$-compact. By Theorems 5.8 and 5.4 of \cite{sentilles}, $\beta_0$ is not only the Mackey topology
of the dual pair $(C_b(\cX^-),\cM(\cX^-))$ but also the
topology of uniform convergence on the $\sigma(\mathscr{M}(\cX^-), C_b(\cX^-))$-compact subsets of $\mathscr{M}(\cX^-)$. 
Thus,
$p_S(F) \coloneqq \sup_{\nu\in S} |\langle F, \nu\rangle|$ is a $\beta_0$-continuous seminorm. It follows that for $|h|<\max\{t, 1\}$
\begin{align*}
|\langle \T (t+h)(\Delta_h U +\D\Th_{-t}U), \mu\rangle | & = |\langle \Delta_h U + \D\Th_{-t} U, \T(t+h)'\mu\rangle|\\
& \leq p_S(\Delta_h U + \D \Th_{-t} U) \to 0
\end{align*}
as $h\to 0$. Noting that $\T(t+h)\D\Th_{-t}U \to \T(t)\D\Th_{-t}U$ as $h\to 0$ by the continuity of $\T$, it follows that
altogether
\[
\left\langle \frac{\varphi(t+h)-\varphi (t)}{h}, \mu \right\rangle \to 0 - \langle \T(t)\D\Th_{-t} U, \mu\rangle
+\langle \T(t)\D\Th_{-t} U, \mu\rangle = 0.
\]
This implies that the map $t\mapsto \langle \varphi (t), \mu\rangle$ is constant, whence $\langle \T(t)\Th_{-t} U, \mu\rangle
= \langle U, \mu\rangle$ for all $t\geq 0$ and $\mu\in\mathscr{M}(\cX^-)$. As $\mu$ was arbitrary, the Hahn--Banach theorem
implies $\T(t)\Th_{-t}U = U$ for all $t\geq 0$.

We note that $D_-(\D_0)$ is an algebra that separates the points in $\cX^-$ and is thus dense in $C_b(\cX^-)$ by the Stone--Weierstra{\ss}
Theorem, see \cite[Theorem \ 11]{fremlin1972}. It follows that, given $F\in C_b(\cX^-)$, we find a net $(U_\alpha)_{\alpha}\subset D_-(\D_0)$ that converges to $F$ with respect
to $\beta_0$. By the continuity properties of $\T(t)$ and $\Th_{-t}$, it follows that the equality 
$\T(t)\Th_{-t}U_\alpha=U_\alpha$ implies $\T(t)\Th_{-t}F=F$. By a bp-closure argument, this equality extends
to arbitrary $F\in B_b(\cX, \cF_0)$, proving that $\T$ is induced by an evolutionary semigroup.\smallskip

The converse implication follows immediately from Theorem \ref{t.continuouscb}.
\end{proof}

\subsection{The general case}\label{sub.general}

We now turn to more general path spaces, where our main interest lies in the space of all c\`adl\`ag paths. If $\cX=\cXd$ then
for $f\in C_b(X)$ the function $F_t(f)$ is not a continuous function on $\cXd$. More precisely, $F_t(f)$ is continuous at the point
$\x\in \cXd$ if and only if $\x$ is continuous at $t$. On the other hand, $F_0(f)$ is continuous on $\cXd^-$, as every $\x\in \cXd^-$
is continuous at $0$. This is an example of a function $F\in C_b(\cXd^-)$ whose extension $\hat F$ is not continuous on
$\cX$.

To study the subtle interaction of the extension map with continuous functions, we introduce the following spaces:
\begin{align*}
\chat  \coloneqq \{ \hat F : F\in C_b(\cX^-)\}= \{ F\in \bm : F|_{\cX^-} \mbox{ is continuous }\}
\end{align*}
is the space of all extensions to $\cX$ of continuous functions on $\cX^-$.
\[
\cXxt \coloneqq \{ F\in C_b(\cX^-) : \hat F \in \cm\} = \{ F|_{\cX^-} : F\in \cm\}
\]
is the space of all continuous functions on $\cX^-$ whose extensions are continuous on all of $\cX$.

To ensure that these spaces are rich enough, we impose some additional assumptions on path spaces.

\begin{defn}
Let $((\cX, \md), \tau)$ be a path space. We say that $((\cX, \md), \tau)$ is \emph{proper}  if
\begin{itemize}
\item[(P4)] If $\x_n\to\x$, then  $\tau(\vt_t\x_n) \to \tau(\vt_t\x)$ in $\cX$ for almost every $t\in \R$.
\item [(P5)] Given a compact subset  $K\subset \cX^-$  and $\eps>0$, there exists $\delta>0$ such that
\[
\md(\x, \tau(\vt_{-t}\x)) \leq \eps 
\]
for all $0\leq t\leq \delta$ and $\x\in K$.
\item[(P6)] There is a constant $c\in (0,1)$ such that given $\delta>0$ we find $t_0>0$ such that
 $\md(\tau^{t_0}(\x), \tau^{t_0}(\y)) \geq c\delta$ for all $\x,\y\in \cX$ with $\md(\x,\y) \geq \delta$.
\end{itemize}
\end{defn}

In Section \ref{sub.proper}, we will prove that $\cXd$ is a proper path space, so that all results of this subsection
are applicable in this case. 

\begin{prop}\label{p.cext}
Let $((\cX, \md), \tau)$ be a proper path space.
\begin{enumerate}
[\upshape (a)]
\item Given $F\in \chat$ and $r>0$ we define
\[
F_r(\x) \coloneqq \frac{1}{r} \int_0^r F(\vt_{-t}\x)\, dt.
\]
Then $F_r\in \cm$ and $F_r(\x) \to F(\x)$ as $r\to 0$ for every $\x\in \cX^-$, uniformly on compact subsets of $\cX^-$.
\item Given $\x\in \cX^-$ and $\delta>0$ there is $F\in \cXxt$ with 
\[
\inf\{ |F(\x) - F(\y)| : \y \in \cX^-,  \md(\x,\y)\geq \delta \} >0
\] 
In other words,\ $\cXxt$ strongly separates the points in $\cX^-$.
\item $\cXxt$ is convergence determining for $\cX^-$.
\end{enumerate}
\end{prop}

\begin{proof}
(a). First observe that as $F$ is determined before $0$, $F\circ\vt_{-t}$ is determined before $-t$ and thus also before $0$.
Integrating, it follows that $F_r$ is $\cF_{0-}$-measurable. Now let $\x_n\to \x$. As $F$ is $\cF_{0-}$-measurable,
$F(\vt_{-t}\x_n) = F(\tau(\vt_{-t}\x_n))$ for all $t>0$ and $n\in \N$ by Lemma \ref{l.measurability}(c). It follows from (P4)
that $\tau(\vt_{-t} \x_n)\to \tau(\vt_{-t}\x)$ for almost all $t\in (0,r)$. By continuity of $F$ on $\cX^-$, $F(\tau(\vt_{-t}\x_n))
\to F(\tau(\vt_{-t}\x))$ for almost all $t\in (0,r)$ and  the dominated convergence theorem
yields $F_r(\x_n) \to F_r(\x)$, proving that $F_r\in \cm$.

Now let $K\subset \cX^-$ be compact and $\eps>0$. As $F$ is uniformly continuous on $K$ we find $\delta>0$
such that $|F(\x) - F(\y)|\leq \eps$ for all $\x, \y \in K$ with $\md(\x, \y)\leq \delta$. By (P5), there is $\rho>0$
such that $\md(\x, \tau(\vt_{-t}\x))\leq \delta$ for all $\x\in K$ and $0\leq t\leq \rho$. It follows for
$0<r<\rho$ and $\x\in K$ that
\[
|F(\x) - F_r(\x)| \leq \frac{1}{r} \int_0^r |F(\x) - F(\tau(\vt_{-t}\x))|\, dr \leq \eps.
\]
\medskip

(b). Given $\x\in \cX^-$ and $\delta>0$, pick $r_0$ such that $\md(\x, \tau(\vt_{-t}\x)) \leq \delta/4$ for all
$0\leq t\leq r_0$. This is possible by (P5). As $\cX^-$ is a metric space, we find $F\in C_b(\cX^-)$ with
$0\leq F\leq 1$ such that $F(\y) = 0$ for all $\y\in \cX^-\cap B(\x, \delta/4)$ and $F(\y) = 1$ for all
$\y\in \cX^-\setminus B(\x, \delta/2)$. 

Consider the function $F_{r_0}$ defined as in part (a). Then $F_{r_0} \in \cm$ and $F(\x) = 0$. If $\y\in \cX^-$
satisfies $\md(\x, \y)\geq \delta$,  we find by (P5) $0< \eps < r_0$ such that
$\md(\y, \tau(-\vt_{-t}\y)) \leq \delta/4$ for all $0\leq t \leq \eps$. It follows that $\md(\x, \tau(\vt_{-t}\y)) \geq \delta/2$
for all $0\leq t \leq \eps$ and thus
\[
|F(\x) - F(\y)| = F(\y) = \frac{1}{r_0}\int_0^{r_0} F(\tau(\vt_{-t}\y)\, dt \geq \frac{1}{r_0}\int_0^\eps 1\, dt = \frac{\eps}{r_0} >0.
\]

(c). This follows from part (b) and \cite[Theorem 3.4.5]{ek}.
\end{proof}

\begin{cor}\label{c.dense}
Let $((\cX, \md), \tau)$ be a proper path space.
Define
\[
\cD \coloneqq \bigcup_{t\in \R} C_b(\cX, \cF_{t-})
\]
and 
\[
\mathscr{P}\coloneqq \Big\{\sum_{j=1}^n F_jG_j :
n\in \N, F_j \in C_b(\cX, \cF_{0-}), G_j \in C_b(\cX, \cF([0,\infty)))\Big\}.
\]
Then both $\cD$ and $\mathscr{P}$ are convergence determining for $\cX$ and $\beta_0$-dense in $C_b(\cX)$.
\end{cor}

\begin{proof}
We note that $\cX^-$ is homeomorphic to $\cX^t \coloneqq \tau^t(\cX)$ via the map $\vt_{-t}$. Thus, it follows
from  from Proposition \ref{p.cext}(b) that $C_b(\cX, \cF_{t-})$ strongly separates points in $\cX^t$. 
Using condition (P6) it follows that $\cD$ strongly separates the points in $\cX$.
As $\cD$ is an algebra, \cite[Theorem 3.4.5]{ek} yields that $\cD$ is convergence determining and the Stone--Weierstra{\ss} Theorem
\cite[Theorem 11]{fremlin1972} yields the density.

The proof for $\mathscr{P}$ is similar. We note
that to prove that $\mathscr{P}$ strongly separates points, we
can consider a product $FG$, where $F\in C_b(\cX, \cF_{0-})$, i.e.\
$F$ is an extension of an element of $\cXxt$, and $G$
is of the form $G(\y) = \md_0^t(\y, \x)$ (cf.\ Equation \eqref{eq.bddskor}) for suitably chosen $t$ and $\x$.
\end{proof}

We can now generalize Theorem \ref{t.continuouscb}:

\begin{thm}\label{t.continuousgen}
Let $((\cX, \md), \tau)$ be a proper path space and 
 $(\T(t))_{t\geq 0}$ be an evolutionary semigroup with expectation operator $\E$. The following
are equivalent:
\begin{enumerate}
[\upshape (i)]
\item $\T(t)\cm \subset \chat$ for all $t\geq 0$;
\item $\T(t) \chat \subset \chat$ for all $t\geq 0$;
\item $\E C_b(\cX) \subset \chat$;
\item $\E C_b(\cX, \cF([0,\infty))) \subset \chat$.
\end{enumerate}
If these equivalent conditions are satisfied,  $\T$ induces (in the sense of \eqref{eq.modified}) a $C_b$-semigroup on $C_b(\cX^-)$.
\end{thm}

\begin{proof}
Assume condition (i) and fix a sequence $(\x_n)\subset \cX^-$ converging to $\x\in \cX^-$. We denote by $\p_t$ the kernel of $\T(t)$
and define the measures $\mu_n$ on $\cX^-$ by setting $\mu_n(A) \coloneqq \p_t(\x_n, \tau^{-1}(A))$ and $\mu$ similarly, replacing $\x_n$ by
$\x$. Condition (i) implies that
\begin{equation}\label{eq.convergence}
\int_{\cX^-} F(\y)\, d\mu_n(\y) = [\T(t)\hat{F}](\x_n) \to [\T(t)\hat{F}](\x) = \int_{\cX^-} F(\y)\, d\mu(\y)
\end{equation}
for every $F\in \cXxt$. As this set is convergence determining by Proposition \ref{p.cext}(c) , the convergence in \eqref{eq.convergence} 
still holds true for all $F\in C_b(\cX^-)$ and this means precisely that $\T(t)\chat \subset \chat$. This proves the 
implication (i) $\Rightarrow$ (ii), the converse implication (ii) $\Rightarrow$ (i) is trivial.\smallskip

To prove (i) $\Rightarrow$ (iii), let $F\in C_b(\cX, \cF_{t-})$ for some
$t>0$. Then $\E F = \T(t)\Th_{-t}F$ is continuous on $\cX^-$ by (i). As $t>0$ was arbitrary, it follows 
for all $F$ in the set $\cD$ from Corollary \ref{c.dense} that $\E F$ is continuous on $\cX^-$. As this set
is convergence determining by Corollary \ref{c.dense}, (iii) follows by a similar argument as above.\smallskip

The implication (iii) $\Rightarrow$ (i) follows immediately from the identity $\T(t)F = \E\Th_t F$.\smallskip

The remaining equivalenct (iii) $\Leftrightarrow$ (iv) follows
as in the proof of Theorem \ref{t.continuouscb}, using
the set $\mathscr{P}$ from Corollary \ref{c.dense}.

Let us now assume that (i) -- (iv) are satisfied. We fix a sequence $(\x_n) \subset \cX^-$ converging to $\x\in \cX^-$. 
Denoting the kernel of $\E$ by $\k$, it follows from (iii) that $\k(\x_n, \cdot) \weak \k(\x, \cdot)$. Given a sequence $(t_n)\subset [0,\infty)$
converging to $t$ and $F\in \cm$, note that $\Th_{-t_n}F \to \Th_{-t}F$ with respect to $\beta_0$ as a consequence of Proposition \ref{p.derivation}.
It follows that
\[
[\T (t_n) F](\x_n) = \int_{\cX} \Th_{-t_n}F(\y) \k(\x_n, d\y) \to \int_{\cX} \Th_{-t} F(\y) \k (\x, d\y) = [\T(t_n) F](\x).
\]
Using that $\cXxt$ is convergence determining for $\cX^-$, the same convergence result holds true for $F\in \chat$, proving the addendum.
\end{proof}

Similarly to the case of continuous paths, we call an evolutionary semigroup $\T$ on a proper path space an \emph{evolutionary
$C_b$-semigroup} if it satisfies the equivalent conditions of Theorem \ref{t.continuousgen}.

A characterization of evolutionary semigroups through the generator in the spirit of Proposition \ref{p.generator} is also possible
in the general case. However, as $D_-(\D_0)$ is not contained in the domain of the $C_b$-generator (as $\D F_a^b(f) \not\in C_b(\cX)$),
we have to use a suitable substitute. We put
\[
D_-(\D) \coloneqq D(\D)\cap C_b(\cX, \cF_{0-}).
\]

\begin{prop}\label{p.generatorgen}
Let $((\cX, \md), \tau)$ be a proper path space and $\T$ be a $C_b$-semigroup on $C_b(\cX^-)$ with $C_b$-generator $\A$. 
Then $\T$ is induced by an evolutionary semigroup if and only if $D_-(\D)\subset D(\A)$ and $\A U = \D U$ for all
$U\in D_-(\D)$.
\end{prop}

\begin{proof}
The proof is identical to that of Proposition \ref{p.generator} with $D_-(\D_0)$ replaced by $D_-(\D)$. We only need 
to know that $D_-(\D)$ is $\beta_0$-dense in $C_b(\cX^-)$. But this follows from Proposition \ref{p.cext}(a), noting
that the function $F_r$ defined there belongs to $D(\D)$ by Proposition \ref{p.awf}.
\end{proof}

We end this section with some results in the case where $\cX = \cXd$ (or $\cX=\cXc$, where they hold trivially). 
We recall that $t^\star = t-$ in this case and we have
$F^\star_t(f) = F_{\tstar}(f)$ whenever $f\in C_b(X)$ by Lemma \ref{l.fbstar}.

\begin{lem}\label{l.0-0}
Let $\cX=\cXd$ and $\T$ be an evolutionary $C_b$-semigroup with expectation operator $\E$. Then 
\[
\T(t) F_{\nstar}(f) = \E F_t(f) = \E F_{\tstar}(f)
\]
for all $f\in B_b(X)$ and $t\geq 0$. In particular,
$\E F_0(f) = F_{\nstar}(f)$ for all $f\in B_b(X)$.
\end{lem}

\begin{proof}
Fix $f\in C_b(X)$ and $t\geq 0$. From the relation of $\T$ and $\E$ it follows that for $s>0$
\[
\T(t+s)F_{\nstar}(f) = \E \Th_{t+s}F_{\nstar}(f) = \E F_{(t+s)^\star}(f).
\]
As $F_{\nstar}(f) \in C_b(\cX^-)$, upon $s\downarrow 0$, the left-hand side converges pointwise to $\T(t)F_{\nstar}(f)$ by continuity of the semigroup. As for the right-hand side,
note that $F_{(t+s)^\star}(f) \to F_t(f)$ as $s\downarrow 0$ by right-continuity of the paths. Using dominated convergence,
$\T(t)F_{\nstar}(f) = \E F_t(f)$ follows. By a bp-closure argument, this equality extends to $f\in B_b(X)$.
\end{proof}

\begin{cor}\label{c.continuous}
In the situation of Lemma \ref{l.0-0}, let $\k$ denote the kernel of the expectation operator $\E$ and write
$\P^\x$ for the measure $\k(\x, \cdot)$. Given $\x\in \cXd$, define
\[
C_\x=\{\y : \y(0) = \y(\nstar)=\x(\nstar)\}=
\{\y : \y \mbox{ is continuous at } 0 \mbox{ and } \y(0) = \x(\nstar)\}. 
\]
Then $\P^\x(C_\x)=1$ for every $\x\in \cX$.
\end{cor}

\begin{proof}
It is $A_- \coloneqq \{\y : \y(\nstar) \neq \x(\nstar)\} \in \cF_{0-}$. 
It follows from Proposition \ref{p.expectation_operator}(iii)
that $\k(\x, A_-) = \delta_{\x}(A_-) = 0$, proving that $\k(\x, \{\y : \y(\nstar) = \x(\nstar)\}) = 1$. Now let $n\in \N$ and
define $f_n(x) \coloneqq 1\wedge nd(x, \x(\nstar))$. Then $f_n\in C_b(X)$ and Lemma \ref{l.0-0} yields
\[
\k(\x, \{\y : d(\y(0), \x(\nstar)) \geq n^{-1}\}) \leq [\E F_0(f_n)](\x) = [F_{\nstar}(f_n)](\x) = 0
\]
for every $n\in \N$. This shows $\k(\x, \{ \y : \y(0) \neq \x(\nstar)\}) = 0$ and yields the claim.
\end{proof}

\section{Examples}\label{sect.examples}

\subsection{Deterministic evolutions}
\label{sub.deterministic}

In this subsection, we are interested in the situation where the expectation operator $\E$ is deterministic, i.e.\ for every 
$\x\in \cX$ there is a unique $\varphi (\x)$ such that $[\E F](\x) = F(\varphi (\x))$.
We note that if $\E$ is of this form, then $\E$ is multiplicative, i.e.\
$\E(FG) = (\E F)(\E G)$. It is well-known, see \cite[Theorem  II.2.3]{cooper}, that any multiplicative $\E$ is of this form.
Throughout, we will work on the path space $((\cX, \md), \tau) = ((\cXc, \md_{\mathrm{C}}), \tau_{\mathrm{C}})$ of continuous 
paths as these seem more appropriate for our purposes,
see Remark \ref{r.continuous} below.

\begin{defn}\label{def.evolutionmap}
An \emph{evolution map} is a continuous mapping $\varphi : \cX \to \cX$ such that
\begin{enumerate}
[\upshape (i)]
\item $\varphi (\x) = \varphi (\tau(\x))$ for all $\x \in \cX$;
\item $\tau(\varphi (\x)) = \tau(\x)$;
\item $\varphi (\vt_t \varphi (\x) ) = \vt_t\varphi (\x)$ for all $\x\in \cX$ and $t\geq 0$.
\end{enumerate}
\end{defn}

\begin{prop}\label{p.delay}
Let $\T$ be an evolutionary $C_b$-semigroup with expectation operator $\E$.
We denote the $C_b$-generator of $\T$ by $\A$. The following are equivalent:
\begin{enumerate}
[\upshape (i)]
\item $\A$ is a derivation;
\item $\T(t)$ is multiplicative for every $t\geq 0$;
\item $\E$ is multiplicative;
\item There is an evolution map $\varphi$ such that $[\E F](\x) = F(\varphi(\x))$ for all $\x\in \cX$.
\end{enumerate}
\end{prop}

\begin{proof}
(i) $\Leftrightarrow$ (ii). This follows from \cite[Theorem 3.4]{fk2020}.

(ii) $\Leftrightarrow$ (iii). If $\T$ is multiplicative and $F, G$ are determined before $s>0$ then
\begin{align*}
\E (FG)  &= \T(s)\Th_{-s}(FG) = \T(s) (\Th_{-s}F)(\Th_{-s}G)\\
&  = [\T(s)\Th_{-s}F][\T(s)\Th_{-s} G] =[\E F][\E G].
\end{align*}
Thus, $\E$ is multiplicative for all functions from the set $\cD$ of Corollary \ref{c.dense}. As this set is
$\beta_0$-dense in $C_b(\cX)$, the above equality extends to arbitrary $F, G \in C_b(\cX)$. 

If we conversely
assume that $\E$ is multiplicative, then for $F,G\in C_b(\cX, \cF_0)$, we have
\[
\T(t) (FG) = \E \Th_t(FG) = \E (\Th_t F)(\Th_t G) = [\E \Th_t F] [\E \Th_t G] = [\T(t)F][\T(t)G].
\]

(iii) $\Leftrightarrow$ (iv). If $\E$ is multiplicative, it follows from \cite[Theorem  II.2.3]{cooper} that there exists a continuous
$\varphi$ such that $[\E F](\x) = F(\varphi(\x))$. As $\E F = \E F\circ\tau$ by Proposition \ref{p.expectation_operator} and Lemma \ref{l.measurability}, we must have $\varphi=\varphi\circ\tau$ and thus (i) in Definition \ref{def.evolutionmap}. Moroever,
$\E F = F$ for all $F\in C_b(\cX, \cF_{0})$. 
Thus, for $F\in C_b(\cX, \cF_{0})$, 
\[
(\E F)(\x) = F(\varphi(\x)) = F(\tau(\varphi(\x))) = F(\tau(\x)) = F(\x).
\]
As $F$ was arbitrary, $\tau\circ\varphi=\tau$ follows. Last, as $\E$ is homogeneous, Lemma \ref{l.expectationop}
with $\k(\x, \cdot) = \delta_{\varphi(\x)}$ yields
\[
\one_A(\varphi(\x)) = \k(\x, A) = \k(\vt_t\varphi(\x), \vt_t(A)) = \one_{\vt_tA}(\varphi(\vt_t\varphi(\x)))
= \one_A(\vt_{-t}(\varphi(\vt_t\varphi(\x)))).
\]
This proves $\varphi(\x) = \vt_{-t}(\varphi(\vt_t\varphi(\x)))$ which is equivalent with $\varphi(\vt_t\varphi(\x)) =\vt_t\varphi(\x)$.
Thus, $\varphi$ is an evolution map. The converse is trivial.
\end{proof}

\begin{rem}
    \label{r.continuous}
    Let us consider for a moment the case of c\`adl\`ag paths, 
    i.e.\ $\cX=\cXd$. In view of Corollary \ref{c.continuous}, for every evolution map
    $\varphi$ that induces an evolutionary $C_b$-semigroup,  
    $\varphi(\x)$ is necessarily continuous at $0$ for every
    $\x\in \cX$. But then condition (iii) of Definition
    \ref{def.evolutionmap} implies that $\varphi(\x)$
    is continuous at every $t\geq 0$.
\end{rem}

With the help of Theorem \ref{t.continuouscb}(b), we can now determine a $\beta_0$-core for the $C_b$-generator of a semigroup $\T$
that satisfies the equivalent conditions of Proposition \ref{p.delay}. In what follows,
we write $\varphi(\x, t)$ shorthand for $[\varphi(\x)](t)$. We recall that $\D_0 F_a	^b (f) = F_{b}(f) - F_a(f)$. 
Applying $\E$, we obtain
for $0\leq a < b$ and $f\in C_b(X)$,
\begin{align*}
[\E F_a^b(f)](\x) & = \int_a^b f(\varphi (\x, s))\, ds \in D(\A)\quad\mbox{and}\quad\\
[\A \E F_a^b(f)](\x) & = [\E (F_{b}(f) - F_a(f))](\x) = f(\varphi(\x, b)) - f(\varphi(\x, a)).
\end{align*}

Via approximation, we obtain more elements of $D(\A)$. Compare the following Proposition with \cite[Proposition 2.12]{fk2020} which concerns
deterministic equations without delay.

\begin{prop}\label{p.f0delay}
Let $\T$ be an evolutionary semigroup satisfying the equivalent conditions of Proposition \ref{p.delay}. Denote its $C_b$-generator by $\A$ and its evolution map by $\varphi$. For $f\in C_b(X)$ and $G\in C_b(\cX, \cF_0)$ the following
are equivalent:
\begin{enumerate}
[\upshape (i)]
\item $F_0(f) \in D(\A)$ and $\A F_0(f) = G$;
\item For every $\x \in \cX$ we have
\[
\lim_{t\to 0}\frac{f(\varphi(\x, t)) - f(\x (0))}{t} = G(\x).
\]
\end{enumerate}
\end{prop}

\begin{proof}
(i) $\Rightarrow$ (ii). Condition (i) is equivalent with
\[
f(\varphi(\x, t)) - f(\x(0) ) = [\T (t)F_0(f) - F_0(f)](\x) = \int_0^t [\T(s)G](\x)\, ds,
\]
for all $\x \in \cX$ and $h>0$. Dividing by $t$ and taking the limit as $t\to 0$, using the continuity of the integrand on the right, (ii) follows.\smallskip

(ii) $\Rightarrow$ (i). We have $t^{-1}F_0^t(f) \in D(\D)$ and $t^{-1}F_0^t(f) \to F_0(f)$ pointwise as $t\to 0$. 
Moreover, $\D t^{-1}F_0^t(f) =
t^{-1}(F_t(f) - F_0(f))$. Applying the evolution operator $\E$, it follows from Proposition \ref{p.generator} that
\[
U_t \coloneqq \E t^{-1} F_0^t(f) \in D(\A) \quad\mbox{and}\quad U_t(\x) \to [F_0(f)](\x)
\]
as $t\to 0$ for all $\x\in \cX$.
Moreover,
\[
[\A U_t](\x) = [\E t^{-1}(F_t(f) - F_0(f))](\x) = \frac{ f(\varphi(\x, t)) - f(\x(0))}{t} \to G(\x)
\]
by (ii). The closedness of $\A$ yields (i).
\end{proof}

A class of examples that falls into the situation described in this section concerns \emph{delay differential equations}. 
Let us discuss this example in more detail.  We limit ourselves to a finite time horizon here. 
A treatment of an infinite time horizon is also possible
but in that case it is more appropriate to consider a slightly different path space, see Example \ref{ex.limitpathspace}.
Thus, fix $h\in (0,\infty)$ and $d\in \N$ and define
\[
\Ch \coloneqq C([-h,0]; \R^d),
\]
which is endowed with the supremum norm.

Given a continuous function
$y: [-h, \infty) \to \R$ and $t\in [0, \infty)$, we define $y_t \in \Ch$ by setting $y_t(s) \coloneqq y(t+s)$. 
Then $y_t$ is the \emph{past at time} $t$.
Given a continuous map $b: \Ch \to \R^d$ and $\xi \in \Ch$, we consider the following delay differential equation:
\begin{equation}
    \label{eq.dde}
    \left\{
\begin{aligned}
    y'(t) & = b(y_t)\quad\mbox{for }t\ge 0,\\
    y_0 & = \xi.
\end{aligned}
    \right.
\end{equation}

A \emph{solution} of \eqref{eq.dde} is a continuously differentiable function $y: [-h, \infty) \to \R^d$, such that \eqref{eq.dde} is satisfied.
It follows from \cite[Theorem II.4.3.1]{bddm07}, that if $b$ is Lipschitz continuous, then for every $\xi\in \Ch$, the delay differential equation \eqref{eq.dde} has a unique solution $y^\xi$. The proof of \cite[Theorem II.4.3.1]{bddm07}, which is based on Banach's fixed point theorem, immediately yields continuity of the map $\xi \mapsto y^\xi$. 

Setting $X\coloneqq \R^d$ and $\cX \coloneqq C(\R; X)$, we now construct an evolutionary $C_b$-semigroup 
$\T$ associated to \eqref{eq.dde}. To that end, define the map $\varphi: \cX \to \cX$ by setting
\begin{equation}
    \label{eq.evolutionmapdde}
    [\varphi(\x)](t) \coloneqq
    \left\{
    \begin{aligned}
        \x(t), & \quad \mbox{if } t\leq 0,\\
        y^{\x_0}(t), & \quad \mbox{if } t>0.
    \end{aligned}
    \right.
\end{equation}

As $\xi\mapsto y^\xi$ is continuous, it follows that $\varphi$ is continuous. 
That $\varphi$ satisfies conditions (i) and (ii) in Definition \ref{def.evolutionmap} is obvious and condition (iii) follows from uniqueness
of solutions to \eqref{eq.dde}. Thus, $\varphi$ is an evolution map and it follows from Proposition \ref{p.delay} that $[\E F](\x) \coloneqq F(\varphi(\x))$ is the expectation operator of a multiplicative evolutionary semigroup $\T$. We also note that $t\mapsto \varphi(\x, t)$ is 
differentiable on $[0,\infty)$.

\begin{prop}
    \label{p.dde}
    Let $b: \Ch\to \R^d$ be Lipschitz continuous, define $\varphi$ by \eqref{eq.evolutionmapdde} and set $[\E F](\x) =F(\varphi(\x))$.
    We also define the semigroup $\T$ by setting $\T(t)=\E\Th_t$ and denote its $C_b$-generator by $\A$.
    \begin{enumerate}
    [\upshape (a)]
        \item For every $t>0$, it is $\T(t)B_b(\cX, \cF([-h,0])) \subset B_b(\cX, \cF([-h,0]))$.
        \item For every $f\in C^1_c(\R^d)$, it is $F_0(f) \in D(\A)$ and
        \begin{equation}
            \label{eq.generatordde}
            [\A F_0(f)](\x) = (\nabla f)(\x (0))\cdot b(\x_0).
        \end{equation}
        
        \item $\T$ is uniquely determined by {\upshape (b)} in the following sense. 
        If $\S$ is a multiplicative evolutionary semigroup with generator $\B$, and for every 
        $f \in C_c^1(\R^d)$, it holds $F_0(f) \in D(\B)$ with $\B F_0(f) = \A F_0(f)$ as given by \eqref{eq.generatordde}, then $\S = \T$.
    \end{enumerate}
\end{prop}

\begin{proof}
    (a). Clearly, $\varphi(\x)|_{[0,\infty)}$ is $\cF([-h, 0])$-measurable. Thus, (a) follows from Lemma \ref{l.invariance}.\smallskip
    
    (b). As $[0,\infty) \ni t\mapsto \varphi(\x, t)$ solves \eqref{eq.dde}, the chainrule shows that 
    $t\mapsto f(\varphi(\x,t))$ is differentiable with derivative $(\nabla f)(\varphi(\x, t))\cdot b(\varphi(\x)_t)$. In particular,
    the derivative at 0 is given by $(\nabla f)(\x(0))\cdot b(\x_0)$. Note that the latter is a bounded function, as $b$ is Lipschitz
    continuous and $f$ has compact support. Thus, (b) follows from Proposition \ref{p.f0delay}.\smallskip

    (c). Let $\S$ and $\B$ as in the statement. By Proposition \ref{p.delay}, the expectation operator of $\S$ is induced by an evolution map $\psi$.
    We only have to prove that $\psi=\varphi$. By assumption, $F_0(f) \in D(\B)$
    with $[\B F_0(f)](\x) = (\nabla f)(\x(0))\cdot b(\x_0)$ for all $f\in C_c^1(\R^d)$. Applying this for functions $f$ satisfying
    $f(x) = x_j$ on $[-n,n]^d$ for some $n\in \N$ and $j=1, \ldots, d$, the characterization of the $C_b$-generator in (iv) of Theorem \ref{t.genchar} implies that for every $\x\in \cX$, the map $t\mapsto \psi(\x, t)$ is differentiable from the right in $0$ with derivative $b(\psi(\x)_0)$. Now, fix $t>0$ and set $\y= \vt_t \psi(\x)$. Note that $\y(0) = \psi(\x, t) = \vt_t\psi(\x, 0)$. 
    By Condition (iii) in Definition \ref{def.evolutionmap}, it is
    \begin{align*}
    \frac{\psi(\x, t+\eps) - \psi(\x, t)}{\eps} & = \frac{\vt_{t+\eps}\psi(\x, 0) - \vt_t\psi(\x,0)}{\eps}\\
    & = \frac{\vt_\eps (\psi (\vt_t\psi(\x), 0) - \psi(\vt_t\psi(\x), 0)}{\eps}
    = \frac{\psi(\y, \eps) - \psi(\y, 0)}{\eps},
    \end{align*}
    for all $\eps>0$. Letting $\eps\to 0$, it follows that $t\mapsto \psi (\x, t)$ is differentiable from the right at every point 
    $t\geq 0$ with derivative $b(\psi(\x)_t)$.
    This shows that $t\mapsto \psi(\x, t)$ is a $W^{1, \infty}_{\mathrm{loc}}$-solution of \eqref{eq.dde} with $\psi = \x_0$. Thus
    \cite[Theorem II.4.3.1]{bddm07} (which works also for this weaker type of solutions) yields $\psi(\x) = \varphi(\x)$.
\end{proof}

\begin{rem}
    It is worth to point out that $D(\A)\setminus D_-(\D)$ also contains elements that are not measurable with respect
    to $\cF(\{0\})$. For example, arguing similar to the proof of (ii) $\Rightarrow$ (i) in Proposition \ref{p.f0delay}, one can show
    that for every $f\in C^1_c(\R^d)$ and $t>0$, the function $F_{\varphi, t}(f): \x\mapsto f(\varphi(\x, t))$ belongs to $D(\A)$ and
    \[
        [\A F_{\varphi,t}(f)](\x) = (\nabla f)(\varphi(\x, t)) \cdot b(\varphi(\x)_t).
    \]
    Then $F_{\varphi, t}$ is $\cF([-h,0])$-measurable, but not, in general, $\cF(\{0\})$-measurable.
\end{rem}

Let us compare the evolutionary semigroup $\T$ constructed above with other semigroup approaches to the delay differential equation 
\eqref{eq.dde}. To the best of our \hyphenation{know-ledge} knowledge, alternative semigroups describing the dynamic in \eqref{eq.dde} are only available if the map
$b$ is linear. 
    
Following \cite{bp05} (see also \cite[Section VI.6]{en00}), 
there is a semigroup on the space $\Ch$ whose generator is given by
    \begin{equation}\label{eq.engen}
        A u = u'\quad\mbox{for}\quad u\in D(A) = \{ v \in \Ch : v'(0) = b(v)\}.
    \end{equation}
Note that for $D(A)$ to be a linear space, it is essential that $b$ is linear. We can identify $C_b(\Ch)$ with $C_b(\cX, \cF([-h,0]))$ as follows. We define
the extension map $\rho: \Ch \to \cX = C(\R; \R^d)$ by 
\[
[\rho(\x)](t) = \begin{cases}
    \x(-r), &  \mbox{if }  t\leq r,\\
    \x(t), &  \mbox{if } t\in (-r, 0),\\
    \x(0), & \mbox{if } t\geq 0.
\end{cases}
\]
The map $F\mapsto \tilde{F}\coloneqq F\circ \rho$ is a homeomorphism between $\Ch$ and $C_b(\cX, \cF([-h,0]))$. Note
that the latter is invariant under $\T$ by Proposition \ref{p.dde}. Using \cite[Corollary VI.6.3]{en00}, we see that the relationship
between $\T$ and $T$ is given by
\[
(\T(t)\tilde{F})(\rho(\x)) = F(T(t)\x)\qquad \mbox{for all } \x \in \Ch.
\]

In \cite[Theorem II.4.2]{bddm07}, a  semigroup is constructed on the product space 
\[
M^p \coloneqq \R^d\times L^p((-h,0); \R^d)
\]
for $1\leq p < \infty$. Thus, an element $u$ of $M^p$ has two components $u_1\in \R^d$
and $u_2 \in L^p((-h,0); \R^d)$. Note that $u_2$ need not have a trace in $0$, but one should think of $u_1$ as a replacement for that trace.
The generator of this semigroup is given by
    \begin{equation}
        \label{eq.bddmgen}
        \left\{
        \begin{aligned}
        A(u_1,u_2) & = (b(u_2), u_2')\\
        D(A) & = \{v\in M^p : u_2\in W^{1,p}((-h,0);\R^d) \mbox{ and } u_1=u_2(0)\}.
        \end{aligned}
        \right.
    \end{equation}
Again, linearity of $b$ is necessary for the linearity of $A$.
    
We note that in both semigroup approaches, 
the \emph{time derivative} plays an important role. In \eqref{eq.engen}, the generator $A$ itself is a realization
of the derivative, whereas in \eqref{eq.bddmgen}, it is the second component of the generator. It is well-known that the generator of (various versions of) the shift semigroup on $C([-h,0]; \R^d)$ or $L^p((-h, 0); \R^d)$ is a suitable realization of this time derivative. Thus, in our
setting, the appropriate generalization of the time derivative is the operator $\D$, 
which  enters the generator $\A$ of $\T$ through the requirement $D_-(\D_0)\subset D(\A)$ and $\A F = \D F$ for $F\in D_-(\D_0)$, see Proposition  \ref{p.generator}. On the other hand, the delay differential equation \eqref{eq.dde} only appears `at the point $0$'. 
In \eqref{eq.engen} via the boundary
condition in the generator, in \eqref{eq.bddmgen} as the first component of $A$ and in $\A$ via Proposition \ref{p.dde}(b).

\subsection{Markov Processes}\label{sub.markov}

In this section, we show how the concept of a classical Markov process (with continuous or c\`adl\`ag paths) fits into
our general framework. Throughout this subsection, we consider $\cX \in \{\cXc, \cXd\}$. We recall that
$\nstar = 0$ if $\cX = \cXc$ and $\nstar = 0-$ if $\cX = \cXd$.
We note that
\[
\sigma(\pi_{\nstar}) = \bigcap_{\eps>0} \cF([-\eps, 0)).
\]
With slight abuse of notation, we put $\cF(\{\nstar\}) = \sigma(\pi_{\nstar})$. 
Note that a function
$F\in B_b(\cX)$ is $\cF(\{\nstar\})$-measurable if and only if it is of the form $F_{\nstar}(f)$
for some $f\in B_b(X)$.

\begin{thm}\label{t.diffusion}
Let $\T$ be an evolutionary $C_b$-semigroup with expectation operator $\E$. Then the following are equivalent:
\begin{enumerate}
[\upshape (i)]
\item $\T(t) F_{\nstar}(f)$ is $\cF(\{\nstar\})$-measurable for all $t\geq 0$ and  $f\in B_b(X)$;
\item $\T(t) F_{\nstar}(f)$ is $\cF(\{\nstar\})$-measurable for all $t\geq 0$ and $f\in C_b(X)$;
\item $\E F$  is $\cF(\{\nstar\})$-measurable for all $F\in \bb{[0,\infty)}$;
\item $\E F$ is $\cF(\{\nstar\})$-measurable for all $F\in \cb{[0,\infty)}$.
\end{enumerate}
If these equivalent conditions are satisfied, the following hold true:
\begin{enumerate}
[\upshape (a)]
\item $\T$ induces a $C_b$-semigroup $T$ on $C_b(X)$ in the sense that
$\T(t) F_{\nstar}(f) = F_{\nstar}(T(t)f)$ for all $t>0$ and $f\in C_b(X)$. 
\item Denote the kernel associated to $\E$ by $\k$ and write $\P^\x$ for the probability measure $\k(\x,\cdot)$.
For every $\x\in \cX$, under the measure $\P^\x$, the canonical process $(Z_t)_{t\ge 0}$ given by $Z_t(\x):=\x(t)$
is a Markov process
with transition semigroup $T$ starting at $\x(\nstar)$.

\end{enumerate}
\end{thm}

\begin{proof}
The implication (i) $\Rightarrow$ (ii) is trivial, while the implication (ii) $\Rightarrow$ (i) follows from a bp-closedness argument.
Taking Lemma \ref{l.0-0} into account, the equivalence (ii) $\Rightarrow$ (iii) follows similar to the proof of Lemma \ref{l.invariance}.
As $\T$ is assumed to be an evolutionary $C_b$-semigroup either Theorem \ref{t.continuouscb} or Theorem \ref{t.continuousgen} yields $\E F \in \hat{C}_b(\cX^-)$ for all $F \in C_b(\cX)$. Taking this into account the implication (iii) $\Rightarrow$ (iv) is trivial and the converse implication
follows, once again, by a bp-closedness argument.\smallskip

Now assume that the equivalent conditions are satisfied.

(a). By assumption, given $t>0$ for every $f\in C_b(X)$, there exists a unique element $T(t)f \in C_b(X)$
such that $\T(t)F_{\nstar}(f) = F_{\nstar}(T(t)f)$. As $\T(t)F_{\nstar}(f) = \E \Th_tF_{\nstar}(f)$, the fact that $\E$ is a kernel operator implies that
$T(t)$ is a kernel operator and from the semigroup property of $\T$ one easily deduces the semigroup property of $T$.
Next, let $\iota : X \to \cX$ be defined by $[\iota(x)](t) \equiv x$.
Then $\iota$ is continuous and $T(t)f(x) = [\T(t)F_{\nstar}(f)](\iota(x))$. Thus, the continuity of the map $(t,x) \mapsto T(t)f(x)$
follows from that of $(t, \x) \mapsto [\T(t)F_{\nstar}(f)](\x)$. This proves (a).\smallskip

(b). Applying Proposition \ref{p.markov}
to $F= F_{\nstar}(f)$ for some $f\in C_b(X)$, it follows from Lemma \ref{l.0-0} that for every $s,t\ge 0$,
\begin{align*}
\big[\E^\x[ f(Z_{t+s}) |\cF_t]\big](\y) &= \big[\E^\x [\Th_{t+s}F_0(f)|\cF_t]\big](\y)= [\T(s)F_{\nstar}(f)](\vt_t\y) = (T(s)f)(Z_t(\y))
\end{align*}
for $\P^\x$-almost all $\y\in \cX$. %Here, we have used Lemma \ref{l.0-0} in the last equality.

\end{proof}

It is a natural question, if every Markovian $C_b$-semigroup $T$ on $C_b(X)$ can be lifted to an evolutionary semigroup $\T$.
For this, it is certainly necessary (by Theorem \ref{t.diffusion}(b)) that the associated Markovian process
can be realized with continuous, resp.\ c\`adl\`ag paths. We will prove next that this condition is also sufficient. To that end,
we put $\cX^+ \coloneqq C([0,\infty); X)$ if $\cX = \cX_{\mathrm{C}}$ and $\cX^+ \coloneqq D([0,\infty); X)$ if
$\cX = \cX_\mathrm{D}$. Let $T=(T(t))_{t\geq 0}$ be  a Markovian $C_b$-semigroup on $X$ and denote by $p_t$ the kernel associated to $T(t)$. 

\begin{defn}\label{def-markov}
    Let $T$ be a $C_b$-semigroup on $C_b(X)$. We say that the associated Markov process \emph{can be realized with paths in $\cX^+$}
    if
    \begin{enumerate}
        [\upshape (i)]
        \item for every $x\in X$ we find a measure $\bP^x$ on $\cX^+$ such that, under this measure, the canonical process is a Markov process
        with transition semigroup $T$ starting at $x$;
        \item the map $x\mapsto \bP^x$ is weakly continuous.
    \end{enumerate}
\end{defn}

\begin{rem}
    Assume that $X$ is locally compact and that $T$ is a \emph{Feller semigroup}, i.e.\ it leaves the space $C_0(X)$ of continuous 
    functions vanishing at infinity invariant and is a strongly continuous semigroup on that space. Then the associated
    Markov process can be realized with paths in $\cX^+ = D([0,\infty); X)$. Indeed, the existence of the measure $\bP^x$ follows
    from \cite[Theorem 4.2.7]{ek} and the continuity of the map $x\mapsto \bP^x$ follows from \cite[Theorem 4.2.5]{ek}.
\end{rem}

\begin{thm}\label{t.converse}
Let $T= (T(t))_{t\geq 0}$ be a Markovian $C_b$-semigroup such that the associated Markov process can be realized with paths
in $\cX^+$. Then there exists an evolutionary $C_b$-semigroup $\T = (\T(t))_{t\geq 0}$ such that
$\T(t)F_0(f) = F_0(T(t)f)$ for all $f\in C_b(X)$.
\end{thm}

\begin{proof}
We denote by $(\bP^x)_{x\in X}$ the family of probability measures on $\cX^+$ such that under these measures the canonical process
on $\cX^+$ is a Markov process with transition semigroup $T$ starting at $x$. Motivated by Remark \ref{rem-SDE1}, we want 
to construct a kernel $\k$ with $\k_+(\x, \cdot)=\bP^{\x(\nstar)}$. To that end, we let $A_- \in \cF((-\infty, 0))$ and
$A_+ \in \cF ([0,\infty))$. Note that we can (and shall) identify these with elements of $\bo(\cX^-)$ and $\bo(\cX^+)$ respectively.
Given $\x\in \cX$, we set
\[
\k(\x, A_-\cap A_+) \coloneqq \delta_{\x}(A_-)\bP^{\x(\nstar)}(A_+).
\]
Note that, in particular, every cylinder set can be written as an intersection $A_-\cap A_+$ with $A_-, A_+$ as above. Arguing
as in the proof of Fubini's theorem (see \cite[Theorem 3.3.1]{bogachev2007}), $\k(\x, \cdot)$ can be extended to a measure
on all of $\bo(\cX)$. Also, it is straightforward to show that $\x\mapsto \k(\x, A)$ is measurable for every $A\in \bo(\cX)$,
so $\k$ is a kernel. 

We prove that $\k$ satisfies Proposition 
\ref{p.expectation_operator}(iii). To that end, let $A_-, B_-\in \cF_{0-}$
and $B_+ \in \cF([0,\infty))$. Setting $A= B_-\cap B_+$, it follows that
\begin{align*}
\k(\x, A_-\cap A) & = \k(\x, (A_-\cap B_-) \cap B_+) = \delta_{\x}(A_-\cap B_-)\bP^{\x(\nstar)}(B_+)\\
& = \delta_{\x}(A_-)\delta_{\x}(B_-)\bP^{\x(\nstar)}(B_+)\\
& = \delta_{\tau(\x)}(A_-)\delta_{\tau(\x)}(B_-)\bP^{(\tau(\x))(\nstar)}(B_+)\\
& = \delta_{\tau(\x)}(A_-)\k(\tau(\x), A).
\end{align*}
By a monotone class argument, this equality generalizes to arbitrary $A\in \bo(\cX)$. Thus, by Proposition \ref{p.expectation_operator},
$\k$ is the kernel of an expectation operator. To conclude that $\E$ is the expectation operator of an evolutionary semigroup,
it remains to verify that $\E$ is homogeneous. This is easily verified using Lemma \ref{l.expectationop} (note that it suffices
to consider a cylinder set $A$ in that Lemma) and the Chapman--Kolmogorov equations for the transition kernels $p_t$ of $T(t)$.
In the case of c\`adl\`ag paths, one additionally uses the observation that under $\bP^x$ the measure of a cylinder set does not change,
if a condition at a time $t>0$ is instead imposed at time 
$t^\star =t-$. That the evolutionary semigroup associated to
the expectation operator $\E$ is a $C_b$-semigroup
follows from Theorem \ref{t.continuouscb} in case of continuous paths
and from Theorem \ref{t.continuousgen} in case of c\`adl\`ag paths. Indeed
for $F\in C_b(\cX, \cF([0,\infty))) \simeq C_b(\cX^+)$, we have
\[
[\E F](\x) = \int_{\cX^+} F(\y) \bP^{\x(0-)}(d\y)
\]
and the latter depends continuously on $\x(0-)$.
\end{proof}

\begin{example}
\label{ex.sde}
An example that fits into the setting described in Theorem \ref{t.converse} is given by stochastic differential equations. Consider
the equation 
\begin{equation}
    \label{eq-SDE3}
    \left\{
    \begin{aligned}
     dY(t) &= b(Y(t)) dt + \sigma (Y(t)) dB(t),\quad\mbox{for }t\ge 0,\\
     Y(0) & = y. 
    \end{aligned}
    \right.
\end{equation}
Here,  $(B(t))_{t\geq 0}$ is an $m$-dimensional Brownian motion, defined on a fixed probability space $(\Omega, \Sigma, \bP)$ and the
functions $b\colon \R^d\to \R^d$ and $\sigma\colon \R^d\to \R^{d\times m}$ are globally Lipschitz continuous. 

Under this assumption it is well-known (see, e.g., \cite[Theorem~5.2.1]{Oeksendal03}) that for every $y\in \R^d$, there exists
a unique strong solution $(Y^y(t))_{t\ge 0}$ of \eqref{eq-SDE3} which is adapted to the filtration $(\cG_t)_{t\in [0,T]}$ generated by 
$\{B(s): s\le t\}$. The process $(Y^y(t))_{t\ge 0}$ is called a (time-homogeneous) \emph{It\^o diffusion process}. 

It follows from \cite[Theorem~7.1.2]{Oeksendal03}, that $(Y^y(t))_{t\geq 0}$ is Markovian. 
Defining for each $y\in \R^d$ the measure  $\bP^y$ on $\cX^+\coloneqq C([0,\infty); \R^d)$ by $\bP^y(A_+) \coloneqq
\bP(Y^y\in A_+)$, \cite[Lemma~8.1.4]{Oeksendal03} implies weak continuity of the map $y\mapsto \bP^y$. Thus, by Theorem \ref{t.converse},
we can construct an evolutionary $C_b$-semigroup from the measures $(\bP^y)_{y\in \R^d}$. To describe the evolution operator $\E$, 
we use the representation from Remark \ref{rem-SDE1} with $\k_+(\x, \cdot) = \bP^{\x(0)}$. By Equation \eqref{eq-SDE2}, we have
\begin{equation}\label{eq-SDE5}
    (\E F)(\x) =\int_{\cX^+} F(\x_-\oplus_0 \y_+) \bP^{\x(0)} (d\y_+)
= \expect F\big(\x_-\oplus_0 Y^{\x(0)}\big).
\end{equation} 
Note that the  integrals are well-defined as the process $(Y^y(t))_{t\ge 0}$ has $\bP$-almost surely continuous paths.
\end{example}

For the rest of this subsection, let $\T$ be an evolutionary $C_b$-semigroup satisfying the equivalent conditions of Theorem \ref{t.diffusion}
and denote by $T=(T(t))_{t\geq 0}$ the induced semigroup on $C_b(X)$. It is a rather natural question, how the $C_b$-generator $\A$ of $\T$ and the 
$C_b$-generator $A$ of $T$ are related. To study this question, we introduce the following notation. 
Denote the transition kernels of $(T(t))_{t\geq 0}$ by $(p_t)_{t\geq 0}$. Given $f_1, \ldots f_n \in C_b(X)$, we put
\[
F(s_1, \ldots, s_n, x) = \int_{X^n} p_{s_1}(x, dy_1)\prod_{j=2}^n p_{s_j}(y_j, dy_{j-1})\prod_{k=1}^n f_k(y_k).
\]
Moreover, for $n\in \mathds{N}$ and $0\leq a < b <\infty$, we put
\[
D_n(a,b) \coloneqq \{ (s_1, \ldots, s_n) \in [0,\infty)^n : a \leq s_1 \leq s_2 \leq \cdots \leq s_n \leq b\}.
\]

\begin{lem}\label{l.simplex}
With the above notation, the function
\[
x\mapsto \int_{D_n(a,b)}F(s_1, \ldots, s_n; x)
\]
belongs to the domain of the generator $A$ of $T$ and
\begin{align*}
&A \int_{D_n(a,b)}F(s_1, \ldots, s_n)\\ &= \int_{D_{n-1}(a,b)} F(s_1, \ldots, s_{n-1}, b) - \int_{D_{n-1}(a,b)} F(a, s_2, \ldots, s_n).
\end{align*}
\end{lem}

\begin{proof}
Making use of the semigroup law and substitution, we see that
\begin{align*}
T(h) \int_{D_{n}(a,b)} F(s_1, \ldots, s_n) & = \int_{D_{n}(a,b)} F(s_1+h, s_2, \ldots, s_n)\\
& = \int_{D_n(a+h, b+h)}F(s_1, \ldots, s_n).
\end{align*}
We now define the sets $\overline{D}_n$ and $\underline{D}_n$ by
\begin{align*}
\overline{D}_n(a,b,h) & \coloneqq \{ (s_1, \ldots, s_n) : a \leq s_1 \leq \cdots \leq s_{n-1} \leq b \leq s_n \leq b+h\}\\
\underline{D}_n(a,b,h) & \coloneqq \{ (s_1, \ldots, s_n) : a \leq s_1 \leq a+h,  s_1 \leq s_2 \cdots \leq s_{n} \leq b \}.
\end{align*}
Using these, we see that
\begin{align*}
(T(h) - I) \int_{D_n(a,b)} & = \int_{D_n(a+h, b+h)} - \int_{D_n(a,b)}\\
& = \int_{D_n(a+h, b+h)} - \int_{D_n(a+h, b)} - \int_{\underline{D}_n(a,b,h)}\\
& = \int_{\overline{D}_n(a,b,h)} - \int_{\underline{D}_n(a,b,h)} + o(h).
\end{align*}
Using dominated convergence, it is easy to see that
\[
\frac{1}{h} \int_{\overline{D}_n(a,b,h)} F(s_1, \ldots, s_n) \to \int_{D_{n-1}(a,b)} F(s_1, \ldots, s_{n-1}, b)
\]
and
\[
\frac{1}{h} \int_{\underline{D}_n(a,b,h)} F(s_1, \ldots, s_n) \to \int_{D_{n-1}(a,b)} F(a, s_2, \ldots, s_{n}).
\]
Combining this with the above, the claim follows.
\end{proof}

\begin{thm}\label{t.generatormarkovian}
Let $\T$ be an evolutionary $C_b$-semigroup satisfying the equivalent conditions of Theorem \ref{t.diffusion}.
We denote by $\A$ the $C_b$-generator of $\T$ and by $A$ the $C_b$-generator of the induced semigroup $T$ on $C_b(X)$. 
\begin{enumerate}
[\upshape (a)]
\item For every $u\in D(A)$, we have
$F_{\nstar}(u)\in D(\A)$ and $\A F_{\nstar}(u) = F_{\nstar}(Au)$. 
\item Functions of the form
\begin{equation}\label{eq.markoviangen}
U = \sum_{j=1}^n U_j F_{\nstar}(u_j)
\end{equation}
where $U_1, \ldots, U_n \in D_-(\D)$ and $u_1, \ldots, u_n \in D(A)$ form 
\begin{enumerate}
[\upshape (i)]
\item a $\beta_0$-core for $\A$ in the case $\cX=\cXc$;
\item a bp-core for $\Af$ in the case $\cX = \cXd$.
\end{enumerate}
\end{enumerate}
\end{thm}

\begin{proof}
(a). 
Let $u\in D(A)$ with $Au=f$. By Proposition \ref{p.awf}, this is equivalent with $T(t)u-u=\varphi_t \coloneqq \int_0^tT(s)f\, ds$.
Consequently,
\begin{align*}
\T(t) F_{\nstar}(u) - F_{\nstar}(u) & = F_{\nstar}(T(t)u) - F_{\nstar}(u) = F_{\nstar}(\varphi_t)\\
& = \E \int_0^tF_{s^\star}(f)\, ds
= \int_0^t \E F_{s^\star}(f)\, ds = \int_0^t \T(s)F_{\nstar}(f)\, ds.
\end{align*}
By Proposition \ref{p.awf}  for $\T$, it follows that $F_{\nstar}(u) \in D(\A)$ and $\A F_{\nstar}(u) = F_{\nstar}(Au)$.\smallskip

(b). Here, we make use of Theorem \ref{t.continuouscb}(b) (in the case $\cX=\cXc$) and Proposition \ref{p.bpcorefora}(b) (in the case
$\cX=\cXd$). In both cases, it suffices to show that every function in $D(\A_0)$
is of the form in \eqref{eq.markoviangen}. To that end, it suffices to compute $\E V$ and $\E \D V$ for $V\in D_+(\D_0)$, i.e.\
when $V$ is a product of functions of the form $F_a^b(f)$ (see Proposition \ref{p.derivation}). Note that
any such $V$ can be written as the sum of integrals of the form
\[
I(\x) = \int_{[a,b]^n} \prod_{k=1}^n f_k(\x (s_k)),
\]
where the parameters $n \in \N$ and $0\leq a < b$ may vary from summand to summand. By linearity,
it suffices to consider a single integral $I$. By Proposition \ref{p.derivation},
\[
\D I(\x) = \sum_{k=1}^n (f_k(\x(b)) - f_k(\x(a))) \int_{[a,b]^{n-1}} \prod_{j\neq k} f_j(\x (s_j))
= \sum_{k=1}^n W_k(\x).
\]
To compute $\E I$ and $\E \D I$, the order of the variables $s_1, \ldots, s_n$ has to be taken into account.
To that end, we make use of the symmetric group $S_n$ and decompose
\[
I(\x) = \sum_{\sigma \in S_n} \int_{D_n(a,b)} \prod_{k=1}^n f_k(\x(s_{\sigma(k)})).
\]
It follows that
\[
\E I(\x) = \sum_{\sigma \in S_n} \int_{D_n(a,b)} F(s_{\sigma (1)}, \ldots, s_{\sigma (n)}; \x (\nstar)).
\]
It is a consequence of Lemma \ref{l.simplex} that, as a function of $\x (\nstar)$, this is an element of $D(A)$ and
\begin{align*}
(A \E I)(\x (\nstar) ) & = \sum_{\sigma \in S_n} \int_{D_{n-1}(a,b)} F(s_{\sigma( 1)}, \ldots, s_{\sigma(n-1)}, b; \x (\nstar))\\
& \qquad  -
\int_{D_{n-1}(a,b)} F(a, s_{\sigma( 2)}, \ldots, s_{\sigma(n)}; \x (\nstar))\\
& = \sum_{k=1}^n \sum_{\substack{\sigma \in S_n\\ \sigma(n) = k}}\int_{D_{n-1}(a,b)} F(s_{\sigma( 1)}, \ldots, s_{\sigma(n-1)}, b; \x (\nstar))\\
& \qquad - \sum_{\substack{\sigma \in S_n\\ \sigma(1) = k}}\int_{D_{n-1}(a,b)} F(a, s_{\sigma( 2)}, \ldots, s_{\sigma(n)}; \x (\nstar))\\
& = \sum_{k=1}^n \E W_k = \E \D I.
\end{align*}
Here we have used in the last line that as $\sigma$ runs trough all of $S_n$ with either the first or the last entry fixed, we run through all of $S_{n-1}$ in the remaining entries. Thus, decomposing the integral in $W_k$ as the integral in $I$ above, the equality follows.
\end{proof}

\begin{example}
In the situation of Example \ref{ex.sde}, the generator of the semigroup $T$ on $C_b(\R^d)$ related to the It\^o diffusion $Y^y$ is given by 
\[ (Au)(y) = b(y)\cdot \nabla u(y) + \frac12 \tr \Big( (\sigma\sigma^\top)(y) D^2 u(y)\Big)\quad \mbox{for } y\in\R^d, \]
for sufficiently smooth $u$, e.g. for $u\in C_c^2(\R^d)$ (see \cite[Theorem~7.3.3]{Oeksendal03}). Theorems \ref{t.diffusion} and \ref{t.generatormarkovian} relate the `classical' semigroup $T$ and its generator $A$ to the evolutionary semigroup $\T$ and its $C_b$-generator $\A$ as follows:
\begin{alignat*}{4}
    \T(t) F_{0}(f) &= F_{0}(T(t)f) &\quad& \mbox{for  } t>0\mbox{ and }f\in C_b(X),\\
    \A F_{0}(u) & = F_{0}(Au) &&\mbox{for } u\in D(A).
\end{alignat*}
\end{example}

\begin{rem}
Although the underlying dynamics are Markovian, the extension 
of the transition semigroup to the path space allows to consider path-dependent functions and their evolutions. For example, in finance, this enables the dynamic pricing of financial derivatives, whose payoffs depend not only on the final value of the underlying asset but also on its entire historical trajectory. 

Indeed, if the risk-neutral dynamics $(Z_t)_{t\geq 0}$ are Markovian, then the fair price of a derivative $f(Z_\tau)$ with maturity $\tau \geq 0$ at time $0 \leq t \leq \tau$ is given by  
\[
u(t,Z_t(\y)) = \big[\E^\x[f(Z_\tau(\y)) \mid \cF_t] \big](Z_t(\y))
\]  
for $\P^\x$-almost all $\y \in \cX$. Under appropriate assumptions, the value function $u$ satisfies the final value problem  
\[
\left\{
\begin{aligned}
\partial_t u(t,x) &= -A u(t,x), \\
u(\tau,x) &= f(x),
\end{aligned}
\right.
\]  
for all $0 \leq t \leq\tau$ and $x \in X$. In particular, the value function is given by $u(t, \cdot) = T(\tau - t) f$, where $T$ denotes the associated transition semigroup with generator $A$. 

Using the extension $\T$ on $\cX^+$ from Theorem~\ref{t.converse},  
the value function  
\[
\Theta_t U(t) = \E^{\x}[F \mid \cF_t]
\]  
corresponding to the fair price of a path-dependent derivative $F \in B_b(\cX, \cF_\tau)$ is given by  
$U(t) = \T(\tau - t) F$.
For further details, we refer to our forthcoming article \cite{DKK25+}, where we provide a detailed analysis of the connection between such value functions, martingales, final value problems, and path-dependent PDEs.

Path-dependent functions also play an important role in control theory, where path-dependent stochastic optimization problems are usually formulated by means of backward stochastic differential equations \cite{MR1037747}. 
If the controlled function is path-dependent, then the value function satisfies a semi-linear path-dependent parabolic equation. For a detailed survey on path-dependent PDEs we refer to \cite{MR4672222}, in particular to \cite{MR3417511} for Sobolev solutions and \cite{MR4337713, MR3531674} for viscosity solutions. 
%The investigation of (semi-)linear path-dependent equations using semigroups on path spaces is subject of future research.
\end{rem}

\subsection{Stochastic delay equations}\label{sec:SDE}

In this section, we show that stochastic delay equations give rise to evolutionary $C_b$-semigroups on the path space 
$((\cX, \md), \tau) = ((\cXc, \md_{\mathrm{C}}), \tau_{\mathrm{C}})$ of continuous functions. We fix $h\in (0,\infty)$ and use the space $\Ch \coloneqq C([-h,0]; \R^d)$ introduced in Section \ref{sub.deterministic}. Also, for a stochastic process $(Y(t))_{t\ge -h}$, we denote the history at time $t$ by $Y_t \in \Ch$.

Given an $m$-dimensional Brownian motion $(B(t))_{t\geq 0}$ on the probability space $(\Omega, \Sigma, \bP)$, consider the stochastic delay equation 
\begin{equation}
    \label{eq-SDE4}
    \left\{
    \begin{aligned}
     dY(t) &= b(Y_t) dt + \sigma (Y_t) dB(t),\quad\mbox{for }t\ge 0,\\
     Y_0 & = \xi,
    \end{aligned}
    \right.
\end{equation}
where $\xi\in \Ch$ is given and $b\colon  \Ch\to \R^d$ and  $\sigma\colon   \Ch\to \R^{d\times m}$
are Lipschitz continuous.
Again, it is well-known that for every $\xi \in \Ch$, the stochastic delay equation \eqref{eq-SDE4} has a unique strong solution $(Y^\xi(t))_{t\ge 0}$ which is 
adapted  to the filtration $(\cG_t)_{t\ge 0}$ generated by $(B(t))_{t\geq 0}$ 
and has almost surely continuous paths (see, e.g., \cite[Chapter~5, Theorem~2.2]{Mao11}).

In the notation of Remark~\ref{rem-SDE1}, we define 
\[ \bP^{\xi}(A_+) := \bP(Y^\xi \in  A_+)\]
for $\xi\in \Ch$ and $A_+\in \bo(\cX^+)$ as well as (analogously to
Equation \eqref{eq-SDE5})
\begin{align*}
    (\E F)(\x) & \coloneqq \expect^{\x_0}\big[ F(\x_-\oplus_0 \cdot)\big] = \int_{\cX^+} F(\x_-\oplus_0\y_+) \bP^{\x_0}(d\y_+)  = \expect F(\x_-\oplus_0Y^{\x_0}).
\end{align*} 
Obviously, $\E$ is a Markovian kernel operator satisfying condition (ii) of Proposition~\ref{p.expectation_operator} and therefore an expectation operator. \smallskip 

To prove that $\E$ is homogeneous, we impose an additional non-degeneracy condition:
For every $\xi\in \Ch$, the matrix $\sigma(\xi)$ is invertible and $\sup_{\xi\in \Ch} \|\sigma(\xi)^{-1}\| <\infty$. Under this assumption,
it is proved in  \cite{Butkovsky-Scheutzow17} that the stochastic delay equation \eqref{eq-SDE4} induces a Markov process on 
the state space $\Ch$. We can use this to establish the
homogeneity of $\E$. To that end, we  introduce the following notation: For $\xi\in \Ch$ and $t\ge 0$, let $Y^{t,\xi}$ be the unique strong solution of
\begin{equation}\label{eq-SDE6}
\left\{
\begin{aligned} Y^{t,\xi}(s) & = \xi(0) + \int_t^s b( Y^{t,\xi}_r) dr + \int_t^s \sigma( Y^{t,\xi}_r) dB(r),\quad \mbox{for } s\ge t,    \\
Y^{t,\xi}_t & = \xi.
\end{aligned}
\right.
\end{equation}

We also set
\[ \big(\x\oplus_tY^{t,\x_t}\big)(\omega, s) := \begin{cases}
    \x(s), & \mbox{if } s\le t,\\
    Y^{t,\x_t}(\omega,s), & \mbox{if }s\ge t.
\end{cases}\]
With this notation, we have $(\E F)(\x) = \expect F(\x\oplus_0Y^{0,\x_0})$. Also, for $t\ge 0$, we have
\[ (\E_t F)(\x)= \expect F \big[\vartheta_{-t}(\vartheta_t\x\oplus_0 Y^{0,(\vartheta_t\x)_0})\big]=  \expect F (\x\oplus_t Y^{t,\x_t}) . \]
The second equality holds due to the uniquenss  of the solution of \eqref{eq-SDE6} and the fact that $B(t+\cdot)-B(t)$ and $B(\cdot)$ have the same distribution. Using this, we obtain
\[ (\E\E_t F)(\x)
= \expect\Big[  \expect F(\y\oplus_t Y^{t,\y_t})\Big|_{\y=\x\oplus_0 Y^{0,\x_0}}\Big].
\]
It was shown in \cite[Proof of Proposition~4.1]{Butkovsky-Scheutzow17} that 
\[ \expect F(\y\oplus_t Y^{t,\y_t})\Big|_{\y=\x\oplus_0 Y^{0,\x_0}}
= \expect \Big[ F(\x\oplus_0 Y^{0,\x_t})\Big|\cG_t \Big]\quad \bP\text{-a.s.}\]
From this, we obtain
\[ (\E\E_t F)(\x)= \expect \Big[\expect\big[ F(\x\oplus_0 Y^{0,\x_0}) | \cG_t\big]\Big] = \expect F(\x\oplus_0 Y^{0,\x_0}) = (\E F)(\x)\]
by the tower property. Hence, $\E$ is homogeneous and therefore generates
an evolutionary semigroup on $C_b(\cX)$.

We remark that the above SDDE can be treated with a semigroup approach in the product space $M^2 = \R^d \times L^2((-h,0); \R^d)$ similar to the case of deterministic
delay equations (see the discussion at the end of Subsection~\ref{sub.deterministic}). 
This approach was used, e.g., in \cite{Gozzi-Marinelli06}, \cite{deFeo-Federico-Swiech24} to solve stochastic optimal control problems for SDDEs. 
Here, the 
construction of the semigroup on the space $M^2$ allows to use results from \cite{daPrato-Zabczyk14} on abstract SDEs in Hilbert spaces. However, as in the deterministic case, the linearity of the coefficients 
is necessary for the linearity of the generator of the semigroup, in contrast to our approach.

\subsection{Stochastic flows driven by L\'evy processes}
\label{sub.flow}
In this subsection, we consider a generalization of the deterministic equations considered in Section \ref{sub.deterministic}, where
the evolution of our system is influenced by a random parameter which lives in an external probability space. We will work
on the c\`adl\`ag path space $((\cX, \md), \tau) = ((\cXd, \md_{J_1}), \tau_D)$ 
endowed with the $J_1$-topology with values in a complete separable metric space $X$. 

\begin{defn}\label{def.randomevolution}
Let $\Omega \coloneqq D([0,\infty); \R^d)$ be the space of all c\`adl\`ag functions 
$\omega\colon [0,\infty)\to \R^d$ endowed with the $\sigma$-algebra $\mathscr{G}:=\sigma(\omega(s)\colon s\in[0,\infty))$ and the filtration  $\mathscr{G}_t:=\sigma(\omega(s)\colon s\in[0,t])$, $t\in[0,\infty)$. A \emph{random evolution map} is a measurable function
$\phi : \Omega\times \cX \to \cX$ such that for fixed $\omega\in \Omega$, the map $\x \mapsto \phi(\omega, \x)$ is continuous
on $\cX^-$ and for every $\omega\in \Omega$ and $\x\in \cX$, the following conditions are satisfied:
\begin{enumerate}
[\upshape (i)]
\item $\phi (\omega,\x) = \phi (\omega,\tau(\x))$;
\item $\tau(\phi( \omega,\x)) = \tau(\x)$;
\item $\phi(\omega,\x)=\phi(\omega+c,\x)$ for all $c\in\R^d$;
\item $\omega\mapsto \tau^t \phi(\omega,\x)$ is $\mathscr{G}_t$-measurable for all $t\in[0,\infty)$;
\item $\phi (\vt_t \omega, \vt_t \phi (\omega,\x))=\vt_{t}\phi (\omega,\x)$ for all $t\in[0,\infty)$.
\end{enumerate}
\end{defn}

\begin{prop}\label{prop:Sphi}
Let $\phi$ be a random evolution map and $\mathbf{P}$ be a probability measure on $(\Omega, \mathscr{G})$ such that 
the canonical process $Z_t(\omega):=\omega(t)$ is a L\'evy process starting at 0.
Then, the operator $\E\colon B_b(\cX)\to B_b(\cX)$, given by
$[\E F](\x):=\mathbf{E}\left[F(\phi(Z,\x))\right]$, is a homogeneous expectation operator. 
The induced evolutionary semigroup is a $C_b$-semigroup.
\end{prop}

\begin{proof}
First, we show that $\E F$ is $\cF_{0-}$-measurable for all $F\in B_b(\cX)$. Indeed, condition (i) implies $[\E F](\x)=\expect\left[F(\phi(Z,\x))\right]=\expect\left[F(\phi(Z,\tau(\x)))\right]=[\E F](\tau(\x))$ for all $\x\in\cX$, and the claim follows directly from Lemma~\ref{l.measurability}(c).
    
Second, we note that $\E F=F$ whenever $F\in B_b(\cX, \cF_{0-})$. Indeed, since $F=F\circ\tau$, it follows from condition (ii) that for every $\x\in\cX$,
\[
[\E F](\x)=\expect \left[F\left(\tau(\phi(Z,\x))\right)\right]=\expect\left[F(\tau(\x))\right]=F(\tau(\x))=F(\x).
\]

Third, we verify that $\E$ is homogeneous.
Let $F\in B_b(\cX)$. By definition of $\E_t$, we have 
\[
[\E_t F](\x)=[\Theta_t\E\Theta_{-t}F](\x)=\expect\left[F(\vartheta_{-t}\phi(Z,\vartheta_t\x))\right]
\]
and therefore 
\[
[\E\E_t F](\x)=\expect\Big[\tilde\expect\big[F\big(\vartheta_{-t}\phi\big(\tilde Z,\vartheta_t\phi(Z,\x)\big)\big)\big]\Big],
\]
where $\tilde Z$ is an independent copy of $Z$. 
Moreover, using first the tower property of conditional expectation, then condition (v) and then condition (i), it follows that
\begin{align*}
[\E F](\x)&=\expect\left[F(\phi(Z,\x))\right]\\
&=\expect\big[\expect [F(\phi(Z,\x))\mid \mathscr{G}_t]\big]\\
&=\expect\Big[\expect\big[F\big(\vt_{-t}\phi ( Z_{t+\cdot}, \vt_t \phi (Z,\x))\big)\mid \mathscr{G}_t\big]\Big]\\
&=\expect\Big[\expect\big[F\big(\vt_{-t}\phi (Z_{t+\cdot}-Z_t + Z_t, \vt_t\tau^t \phi (Z,\x))\big)\mid \mathscr{G}_t\big]\Big]\\
&=\expect\Big[\tilde \expect\big[F\big(\vartheta_{-t}\phi\big(\tilde Z+Z_t,\vartheta_t\tau^t\phi(Z,\x)\big)\big)\big]\Big]\\
&=\expect\Big[\tilde \expect\big[F\big(\vartheta_{-t}\phi\big(\tilde Z,\vartheta_t\phi(Z,\x)\big)\big)\big]\Big]=[\E\E_t F](\x),
\end{align*}
where $\tilde Z$ is an independent copy of $Z$.
Here, in the fifth equality, we used that the mapping $\omega\mapsto \vartheta_t\tau^t\phi(Z(\omega),\x)$ is $\mathscr{G}_t$-measurable due to condition (iv) and that $Z_{t+\cdot}-Z_t$ is independent of $\mathscr{G}_t$.
The sixth equality is valid due to condition (iii). 

To verify that the induced evolutionary semigroup is a $C_b$-semigroup, we use Theorem \ref{t.continuousgen}. If $(\x_n)_{n\in\mathbb{N}}$ is a sequence
in $\cX^-$ converging to $\x \in \cX^-$, then by assumption, we have $\phi(\omega, \x_n) \to \phi(\omega, \x)$. Given $F\in C_b(\cX)$,
it follows from the dominated convergence theorem  that
\[
[\E F](\x_n) = \expect [F(\phi (Z, \x_n))] \to \expect [F(\phi(Z,\x))] = [\E F](\x). \qedhere
\] 
\end{proof}

As an application, we consider a stochastic delay equation, driven by additive L\'evy noise. To that end, let
\[
\Dh \coloneqq D((-h,0); \R^d).
\]
For fixed
$\omega \in \Omega$ and $\xi\in \Dh$, we consider the (deterministic) delay equation
\begin{equation}
    \label{eq.sdde}
    \left\{
    \begin{aligned}
        dy(t) & = b(y_t)dt + d\omega(t),\quad\mbox{for } t\ge 0,\\
        y_0 & = \xi.
    \end{aligned}
    \right.
\end{equation}
We interpret this as an integral equation and require the solution to be continuous at 0. Thus, \eqref{eq.sdde} is equivalent to
$y(t) = \xi(t)$ for $t\leq 0$ and 
\begin{equation}
    \label{eq.sddeintegral}
    y(t) = \xi(0-) + \int_0^t b(y_s)\, ds + \omega(t) - \omega(0)
\end{equation}
for all $t>0$. If $b$ is Lipschitz continuous, then, arguing similar as in \cite[Theorem II.4.3.1]{bddm07}, we see that
\eqref{eq.sdde} has a unique solution $y^\xi(\omega) \in D((-h, \infty); \R^d)$. 
Given $\omega\in \Omega$ and $\x \in \cX$, we define
\[
[\phi(\omega, \x)](t) \coloneqq \begin{cases}
    \x(t), & \mbox{if } t<0,\\
    y^{\x_0}, & \mbox{if }t\geq 0.
\end{cases}
\]
We claim that $\phi$ is a random evolution map in the sense of Definition \ref{def.randomevolution}. Here conditions (i) -- (iv) can be directly verified. %are more or less obvious. 
To prove (v), we fix $t>0$ and set $\y = \vt_t\phi(\omega, \x)$. Then we have $\y(s) = [\phi (\omega, \x)](t+s)$. It follows that
\begin{align}
    \y(s) & = \x(0-) + \int_0^{t+s}b(\phi(\omega, \x)_r)\, dr + \omega(t+s) - \omega(0)\notag\\
    & = \x(0-) + \int_0^t b(\phi(\omega, \x)_r)\, dr + \int_t^{t+s} b(\phi(\omega, \x)_r)\, dr + \omega(t+s)-\omega(0)\notag\\
    & = [\phi(\omega, \x)](t-) + \int_0^sb(\phi(\omega, \x)_{t+r})\, dr + \omega(t+s) - \omega(t)\notag\\
    & = \y(0-) + \int_0^s b(\y_r)\, dr + (\vt_t\omega)(s) - (\vt_t\omega)(0).\label{eq.randomevolution}
\end{align}
By uniqueness, $\phi(\vt_t\omega, \y) = \y$, which is exactly (v).

\begin{example}\label{ex.levydelay}
Let $X=\R^d$ and $b: \R^d \to \R^d$ be a bounded, continuous function. For fixed $\omega\in \Omega$ and $\x\in \cX$, 
consider the delay equation
\begin{equation}
\label{eq.stochdelay}
dy(t) = b(y(t-1))dt + d\omega(t), \quad y_0 = \x_0. 
\end{equation}
Of course, we interpret this equation as an integral equation and require the solution to be continuous at $0$ (see Remark
\ref{r.continuous}). Thus, \eqref{eq.stochdelay} is equivalent to
$y(t) = \x(t)$ for $t\leq 0$ and
\[
y(t) = \x(0-) + \int_0^t b(y(s-1))\, ds + \omega(t) -\omega(0)
\]
for all $t>0$.

This equation can be solved iteratively, yielding the following random evolution map:
\[
\phi(\omega, \x, t) = \begin{cases}
\x(t), & \mbox{if } t < 0,\\
\x(0-) + \int_0^t b(\x(s-1))\, ds + \omega(t)-\omega(0), & \mbox{if }t\in [0,1),\\
\phi(\omega, \x, 1-) + \int_1^t b(\phi(\omega, \x, s-1))\, ds + \omega(t)-\omega(0), & \mbox{if } t\in [1, 2),\\
\ldots
\end{cases}
\]
This is indeed a random evolution map. Conditions (i) -- (iv) can be directly verified %are more or less obvious 
and for (v) one can argue as in Equation \eqref{eq.randomevolution}.\smallskip

Now, let $\bP$ be a measure on $(\Omega, \mathscr{G})$ such that the coordinate process $Z_t(\omega) \coloneqq \omega(t)$
is a L\'evy process starting at $0$. We denote by $(S(t))_{t\geq 0}$ the transition semigroup of this process on $B_b(\R^d)$.
It is well known, that $S$ is a $C_b$-semigroup (it is actually \emph{Feller}, in the sense that it restricts to a strongly continuous
semigroup on $C_0(\R^d)$) and we denote its $C_b$-generator by $A$. Moreover, we denote the expectation operator
associated to $\phi$ and $\bP$ via Proposition \ref{prop:Sphi} by $\E$ and the induced evolutionary semigroup by $\T$.

We also define the map $\varphi : \cX \to \cX$ by setting $\varphi (\x, t) \coloneqq \x(t)$ for $t<0$
and $\varphi(\x, t) \coloneqq \x(0-) +\int_0^t b(\x(s-1))\, ds$ for $t\in [0,1)$. Afterwards, we define $\varphi$ recursively, setting
$\varphi(\x, t) \coloneqq \varphi(\x, n-) + \int_n^t b(\varphi(\x, s-1))\, ds$ for $t\in [n, n+1)$. 

It follows for $t\in (0,1)$ and $\x\in \cX$ that $\phi(Z, \x, t) = \varphi(\x, t) + Z_t$ $\bP$-almost surely, so that for any $f\in B_b(\R^d)$,
\[
\T(t)F_{\nstar}(f)(\x) = \expect [f(\phi(Z, \x, t))] = \expect f(\varphi(\x, t) + Z_t) = (S(t)f)(\varphi(\x, t)).
\]
Note that $t\mapsto \varphi(\x, t)$ is differentiable  with derivative
$b(\varphi(\x, t-1))$ at any point where $\x$ is continuous. 
If $f\in C_c^\infty(\R^d)$, then $t\mapsto S(t)f$ is differentiable with respect to $\|\cdot\|_\infty$
and the derivative is $S(t)Af$. With the chain rule, it follows that 
\begin{align*}
\partial_t \T(t)F_{\nstar}(f)(\x) & = S(t)Af(\varphi(\x, t)) + S(t)\big[\nabla f(\varphi(\x,t))\cdot b(\varphi(\x, t-1))\big]\\
& = \T(t)\Big[F_{\nstar}(Af) + \sum_{j=1}^d F_{\nstar}(\partial_jf)F_{-1}(b_j)\Big](\x)
\end{align*}
for all but finitely many $t\in (0,1)$. Proceeding recursively, one can prove the same result for almost all $t>0$. Now
Proposition \ref{p.awf} shows that
\[
[F_{\nstar}(f), F_{\nstar}(Af) + \sum_{j=1}^d F_{\nstar}(\partial_j f) F_{-1}(b_j)\big] \in \Af,
\]
where $\Af$ denotes the full generator of $\T$.
\end{example}

\begin{rem}
\label{rem.notfeller}
The stochastic delay equation from Example \ref{ex.levydelay} is a special case of the equations considered in \cite{rrg06}, namely we here
only consider additive L\'evy noise whereas the authors of \cite{rrg06} also allow a multiplicative noise term depending on the past. 
In \cite{rrg06}, the authors construct the transition semigroup of the solution on the state space $\cX_{[-h,0]} =D([-h, 0])$ of c\`adl\`ag paths
on $[-h,0]$. As they discuss on page 1416, their semigroup is not a $C_b$-semigroup. In fact, it 
neither leaves the space $C_b(\cX_{[-h,0]})$ invariant 
nor does it satisfy the continuity condition from Definition \ref{d.cbsem}. Inspecting their argument, one sees that the problem 
is that the function $\pi_{-h}$, which is continuous on $\cX_{[-h,0]}$ because of the special role of the point $-h$, 
is mapped by the shift semigroup to the non-continous function $\pi_{-h+t}$. In our setting, this problem does not occur
as $\pi_{-h}$ is not continuous on $\cX^-$. Indeed, Proposition \ref{prop:Sphi} shows that we obtain a $C_b$-semigroup and using
Lemma \ref{l.invariance} one can easily see that it leaves $C_b(\cX, \cF([-h,0)))$ invariant.
\end{rem}

\appendix

\section{Transition semigroups}\label{sect.appendix}

In many cases, the concept of a strongly continuous semigroup, see for example \cite{en00}, is too strong to study transition semigroups
of Markovian processes. Over the years, several suggestions have been made for alternative semigroup theories that can be used instead
but, so far, no single theory has prevailed. We mention weakly integrable semigroups \cite{jefferies86}, 
weakly continuous semigroups \cite{cerrai94}, $\pi$-semigroups \cite{priola99}, bi-continuous semigroups \cite{kuehnemund03, farkas04}
and sequentially equicontinuous semigroups \cite{kms22}.
Let us also mention \cite{gk01, gnr24} where transition semigroups of certain stochastic processes were studied; here the so-called \emph{strict}
or \emph{mixed} topology plays an important role.

In this article, we use a very flexible semigroup theory, that can be seen
as a special case of the theory of `semigroups on norming dual pairs' introduced in \cite{kunze09, kunze11}. On the other hand, we 
have also opportunity to make use of the `full generator' of a semigroup that goes back to the book of Ethier and Kurtz \cite[Section 1.5]{ek}.
As there is no single reference available where all needed
results can be found, we introduce in this appendix the relevant notions and state the results that are used in the main part; there
are also some new results.\smallskip

Throughout this appendix, $X$ is a Polish space and $\bo(X)$ denotes its Borel $\sigma$-algebra. 
The spaces of bounded real valued  Borel measurable functions and bounded continuous
functions on $X$ are denoted by $B_b(X)$ and $C_b(X)$ respectively. Both are Banach spaces with respect to the supremum
norm $\|\cdot\|_\infty$.  The space of bounded signed measures is denoted by $\cM(X)$. We note that $\cM(X)$
can be isometrically identified with a closed subspace of the norm dual 
 $B_b(X)^*$ and also a closed subspace of $C_b(X)^*$ by 
means of the identification
\[
\langle f, \mu\rangle \coloneqq \int_X f(x) d\mu(x).
\]

\subsection{Kernel operators and bp-convergence}\label{s.a1}

A \emph{kernel} on $X$ is a map $k: X\times \bo(X)\to \R$ such that
\begin{enumerate}
[(i)]
\item the map $x\mapsto k(x, A)$ is measurable for every $A\in \bo(X)$;
\item $k(x,\cdot)$ is a signed measure for every $x\in X$;
\item $\sup_x |k|(x, X)<\infty$ where $|k|(x, \cdot)$ refers to the total variation of $k(x, \cdot)$.
\end{enumerate}
A kernel $k$ is called \emph{positive} (\emph{Markovian}) if every measure $k(x, \cdot)$ is a positive measure (a probability measure).

Given a kernel $k$ on $X$, we can define a bounded linear operator $K$ on $B_b(X)$ by setting
\begin{equation}
\label{eq.assop}
[Kf](x) \coloneqq \int_X f(y) k(x, dy) \quad \mbox{for } f\in B_b(X)\mbox{ and } x\in X.
\end{equation}
We call an operator of this form a \emph{kernel operator} and $k$ the kernel \emph{associated} to $K$ and, conversely, $K$ the operator
associated to $k$. 

It may happen, that in \eqref{eq.assop} the function $Kf$ is continuous, whenever $f\in C_b(X)$. In fact, this is the case if and
only if the map $x\mapsto k(x, \cdot)$ is continuous with respect to the weak topology $\sigma(C_b(X), \cM(X))$. In this case,
$K$ defines a bounded linear operator on $C_b(X)$. We call such an operator \emph{kernel operator on $C_b(X)$}. We note that
any kernel operator on $C_b(X)$ can be extended to a kernel operator on $B_b(X)$ (by merely using \eqref{eq.assop} for general
$f\in B_b(X)$ instead of $f\in C_b(X)$). We will not distinguish between kernel operators on $C_b(X)$ and their (unique) extension
to $B_b(X)$.\smallskip

We now treat the cases $C_b(X)$ and $B_b(X)$ simultaneously and let $V$ denote either of these spaces. 
It turns out that a bounded linear operator
on $V$ is a kernel operator if and only if it is continuous with respect to the weak topology $\sigma\coloneqq \sigma(V, \cM(X))$, see
e.g.\ \cite[Proposition 3.5]{kunze11}. We write $\cL(V, \sigma)$ for the space of $\sigma$-continuous linear operators on $V$.
For sequences, $\sigma$-convergence is nothing else than \emph{bp-convergence} (bp is 
short for \emph{b}ounded and \emph{p}ointwise). Indeed, by dominated convergence, bp-convergence implies $\sigma$-convergence
and the converse follows by testing against Dirac measures and using the uniform boundedness principle.

\begin{lem}\label{l.kernelop}
Let $V\in \{C_b(X), B_b(X)\}$ and $K \in \cL(V)$. The following are equivalent:
\begin{enumerate}
[\upshape (i)]
\item $K$ is a kernel operator.
\item $K$ is $\sigma$-continuous, i.e.\ $K\in \cL(V, \sigma)$.
\item $K$ is \emph{bp-continuous}, i.e.\ if $(f_n)_{n\in \N} \subset V$ is a sequences that bp-converges to $f\in V$, then 
$Kf_n$ bp-converges to $Kf$.
\item the norm-adjoint $K^*$ leaves the space $\cM(X)$ invariant.
\end{enumerate}
In this case, the $\sigma$-adjoint of $K$ is given by $K'\coloneqq K^*|_{\cM(X)}$.
\end{lem}

\begin{proof}
The equivalence of (ii) and (iv) (as well as the addendum) follow from \cite[Proposition 3.1]{kunze11} 
applied to the norming dual pair $(V, \cM(X))$. As already mentioned, the equivalence of (ii) and (i) is \cite[Proposition 3.5]{kunze11}.
For the equivalence of (i) and (iii), see \cite[Lemma A.1]{bk24} in the case $V=C_b(X)$ and \cite[Lemma 5.1]{akk16} in the case
$V=B_b(X)$.
\end{proof}

\begin{rem}
    In the case $V=B_b(X)$, the statements of Lemma \ref{l.kernelop} remain valid also for general measurable spaces.
\end{rem}

Let us also briefly recall the notion of \emph{bp-closure}, see \cite[Section 3.4]{ek}. A subset $M\subset B_b(X)$
is called \emph{bp-closed} if whenever $(f_n)_{n\in\N}\subset M$ bp-converges to $f\in B_b(X)$, it follows that
$f\in M$. The \emph{bp-closure} of a set $S\subset B_b(X)$ is the smallest bp-closed set that contains $S$. 
If the bp-closue of $S$ is all of $B_b(X)$, we say that $S$ is \emph{bp-dense} in $B_b(X)$.
By \cite[Proposition 3.4.2]{ek}, $C_b(X)$ is bp-dense in $B_b(X)$. 

These notions help overcome a weakness of working with the weak topology $\sigma$. As $C_b(X)$
is $\sigma$-dense in $B_b(X)$, for every $f\in B_b(X)$ there is a net $(f_\alpha)\subset C_b(X)$ converging to $f$ with
respect to $\sigma$. However, there need not be a sequence $(f_n)_{n\in \N}\subset C_b(X)$ that bp-converges
(i.e.\ $\sigma$-converges) to $f$. Using bp-density and bp-closures allows us to work only with sequences
nevertheless.

\subsection{The full generator of a transition semigroup}\label{s.a2}

\begin{defn}\label{d.transsg}
A \emph{transition semigroup} on $X$ is a family of Markovian kernel operators $T=(T(t))_{t\geq 0}\subset \cL(B_b(X), \sigma)$
such that
\begin{enumerate}
[(i)]
\item $T(t+s) = T(t)T(s)$ for all $t,s\geq 0$ and $T(0) = I$;
\item For every $f\in X$ the map $(t,x) \mapsto [T(t)f](x)$ is jointly measurable.
\end{enumerate}
If the same is true for a family $T=(T(t))_{t\in\R}\subset \cL(B_b(X), \sigma)$ with index set $\R$, we call $T$ a
\emph{transition group}.
\end{defn}

If $T$ is a transition semigroup, then for every $f\in B_b(X)$ and $\mu\in \cM(X)$ the map
$t\mapsto \langle T(t)f, \mu\rangle$ is measurable and, given $\lambda > 0$,
there exists an operator $R(\lambda) \in \cL(B_b(X), \sigma)$ such that
\begin{equation}
\label{eq.resolvent}
\langle R(\lambda)f, \mu\rangle = \int_0^\infty  e^{-\lambda t}\langle T(t)f, \mu \rangle dt \quad  \mbox{for } f\in B_b(X)\mbox{ and } \mu \in \cM(X).
\end{equation}
This shows that, in the terms of \cite[Definition 5.1]{kunze11}, a transition semigroup is an integrable semigroup on 
the norming dual pair $(B_b(X), \cM(X))$. By \cite[Proposition 5.2]{kunze11}, the family $(R(\lambda))_{\lambda >0}$ is a 
\emph{pseudo-resolvent}, i.e.\ it satisfies the resolvent identity
\[
R(\lambda) - R(\mu) = (\mu-\lambda)R(\lambda)R(\mu)\quad \mbox{for } \lambda, \mu >0.
\]
For more information on pseudo-resolvents we refer to \cite[Section III.4.a]{en00}.

In general, the operators $(R(\lambda))_{\Re\lambda>0}$ are not injective and hence are not
the resolvent of a single-valued operator. However, there is a multi-valued operator $\Lf$ (called the  \emph{full generator}) 
such that $(\lambda - \Lf)^{-1} = R(\lambda)$ for $\lambda >0$, see \cite[Appendix A]{haase} (also for more
information about multi-valued operators). The \emph{domain} of $\Lf$ is given by
$D(\Lf) \coloneqq \{ f\in B_b(X) : (f,g) \in \Lf \mbox{ for some } g\in B_b(X)\}$.

\begin{prop}\label{p.awf}
Let $T$ be a transition semigroup with full generator $\Lf$ and $f,g\in B_b(X)$. The following
are equivalent:
\begin{enumerate}
[\upshape (i)]
\item $(f,g)\in \Lf$.
\item For all $t>0$ and $x\in X$ it is
\[
T(t)f(x) - f(x) = \int_0^t (T(s)g)(x)\, ds.
\]
\item For all $x\in X$ the map $t\mapsto T(t)f(x)$ belongs to $W^{1,\infty}_{\mathrm{loc}}([0,\infty))$ and
$\partial_t T(t)f(x)$ $= T(t)g(x)$ for almost every $t$.
\end{enumerate}
\end{prop}

\begin{proof}
The equivalence of (i) and (ii) is proved in \cite[Proposition 5.7]{kunze11}. Note that this shows that our terminology
is consistent with that used in the book of Ethier--Kurtz \cite[Section 1.5]{ek} who use (ii) for the definition. In view of the
fundamental theorem of calculus for Sobolev functions, (iii) is merely a reformulation of (ii).
\end{proof}

\begin{lem}
\label{l.fullgenprop}
Let $T$ be a transition semigroup with full generator $\Lf$. The following hold true:
\begin{enumerate}
[\upshape (a)]
\item $\Lf$ is bp-closed, i.e.\ if  $((f_n, g_n))_{n\in \N}\subset \Lf$ bp-converges to $(f,g) \in B_b(X)\times B_b(X)$,
then $(f,g)\in \Lf$;
\item If $D(\Lf)$ is bp-dense in $B_b(X)$, then $T$ is uniquely determined by $\Lf$.
\end{enumerate}
\end{lem}

\begin{proof}
(a). Let $((f_n, g_n))_{n\in \N} \subset \Lf$. By Proposition \ref{p.awf}, this is equivalent to
\begin{equation}
\label{eq.awfapprox}
T(t)f_n(x) - f_n(x) = \int_0^t (T(s)g_n)(x)\, ds
\end{equation}
for all $t>0$ and $x\in X$. If $(f_n, g_n)$ bp-converges to $(f,g)$ then the bp-continuity of the operators $T(t)$ imply that the left 
hand side of \eqref{eq.awfapprox} converges to $T(t)f(x) - f(x)$ and the right-hand side to 
$\int_0^t T(s)g(x)\, ds$. Using Proposition \ref{p.awf} again, (a) follows.\smallskip

(b). It follows from Proposition \ref{p.awf} that $\|T(t)f - T(s)f\|_\infty \leq |t-s|\|g\|_\infty$ for $(f,g)\in \Lf$ and $t,s>0$. 
This implies that, for $f\in D(\Lf)$ the orbit $t\mapsto T(t)f$ is $\|\cdot\|_\infty$-continuous and thus in particular pointwise continuous.
Now let $T_1$ and $T_2$ be transition semigroups with the same full generator $\Lf$ and fix $f\in D(\Lf)$. It follows that
\[
\int_0^\infty e^{-\lambda t}(T_1(t)f)(x)\, dt = \int_0^\infty e^{-\lambda t}(T_2(t)f)(x)\, dt = (\lambda - \Lf)^{-1}f(x)
\]
for all $\lambda >0$ and $x\in X$. By the uniqueness theorem for the Laplace transform \cite[Theorem 1.7.3]{abhn},
$(T_1(t)f)(x) = (T_2(t)f)(x)$ for almost every $t>0$. By continuity, $T_1(t)f(x) = T_2(t)f(x)$ for all $t>0$. Thus,
for every $t>0$, the operators $T_1(t)$ and $T_2(t)$ coincide on the set $D(\Lf)$. Now fix $t>0$. 
As the operators $T_1(t), T_2(t)$ are bp-continuous,
the set of $f$ for which $T_1(t)f= T_2(t)f$ is bp-closed. Since $D(\Lf)$ is bp-dense in $B_b(X)$
, $T_1(t) = T_2(t)$ follows.
\end{proof}

Following \cite{kk22}, we say that a subset $C\subset \Lf$ is a \emph{bp-core}, if $\Lf$ is the bp-closure of $C$. 
We next give a criterion for a bp-core.

\begin{lem}\label{l.bpcore}
Let $(T(t))_{t\geq 0}$ be a transition semigroup with full generator $\Lf$. Moreover, let
$C\subset \Lf$ be a subspace with the following properties:
\begin{enumerate}
[\upshape (i)]
\item If $(f,g)\in C$ then $(T(t)f, T(t)g)\in C$ for all $t>0$;
\item The set $S \coloneqq \{ f\in B_b(X) :  (f,g)\in C \mbox{ for some } g\in B_b(X)\}$ is bp-dense in $B_b(X)$;
\item For every $x\in X$ and $g\in B_b(X)$ such that $(f,g)\in C$ for some $f\in B_b(X)$, the map
$t\mapsto T(t)g(x)$ is continuous at almost every point in $(0,\infty)$.
\end{enumerate}
Then $C$ is a bp-core for $\Lf$.
\end{lem}

\begin{proof}
We denote the bp-closure of $C$ by $M$. Since $C$ is a vector space, so is $M$. Moreover, $\Lf$ is 
bp-closed by Lemma \ref{l.fullgenprop}(a), so $M \subset\Lf$. 
We thus only need to prove that $\Lf\subset M$.\smallskip

Fix $(f,g)\in C$ and $t_0>0$. We claim that for every $\lambda >0$
\[
\Big(\int_0^{t_0} e^{-\lambda t}T(t)f\, dt , \int_0^{t_0}e^{-\lambda t} T(t)g\, dt \Big) \in M.
\]
To see this, let $x\in X$ and define the functions $\varphi(t) \coloneqq \one_{(0,t_0)}(t) e^{-\lambda t} [T(t)f](x)$ and $\varphi_n(t) \coloneqq
\sum_{k=1}^n \one_{[\frac{k-1}{n}t_0, \frac{k}{n}t_0)}(t) e^{-\lambda\frac{k-1}{n}}\big[T(\frac{k-1}{n}t_0)g\big](x)$ and the functions
$\psi$ and $\psi_n$ similarly, replacing $f$ with $g$. It is a consequence of condition (i) and the fact that $M$ is a vector space
that $(\int_0^{t_0} \varphi_n(t)\, dt, \int_0^{t_0} \psi_n(t)\, dt) \in M$. Since $f\in D(\Lf)$, the orbit $t\mapsto T(t)f$ is $\|\cdot\|$-continuous
(see the proof of Lemma \ref{l.fullgenprop}(b))
and hence pointwise continuous. It follows that $\varphi$ is Riemann-integrable and $\int_0^{t_0}\varphi_n(t)\, dt$, which is nothing else
than Riemannian sums for the integral of $\varphi$ converge to $\int_0^{t_0} \varphi(t)\, dt$. It is a consequence of condition (iii), that
also the function $\psi$ is Riemann integrable and thus the integrals of $\psi_n$ converge to that of $\psi$, proving the claim.\smallskip

Next note that since $(f,g) \in \Lf$, we have
\[
\int_0^{t_0} e^{-\lambda t}T(t)g\, dt = \lambda \int_0^{t_0} e^{-\lambda t}T(t)f\, dt - f + e^{-\lambda t_0} T(t_0)f,
\]
see, e.g., \cite[Proof of Proposition 5.7]{kunze11}.
For $t_0\to \infty$, the last term tends (in norm) to $0$ and it follows that
\[
(R(\lambda, \Lf) f, \lambda R(\lambda, \Lf)f - f) \in M
\]
for every $f\in S$. Since $S$ is bp-dense in $B_b(X)$ by (ii), the same is true for arbitrary $f\in B_b(X)$. Now let $(f_0, g_0) \in \Lf$. 
Note that $R(\lambda,\Lf)(\lambda f_0-g_0) = f_0$. Thus, applying the above to $f= \lambda f_0 - g_0$, it follows that
$(f_0, g_0)\in M$, finishing the proof.
\end{proof}

Sometimes, it is more convenient to work with single-valued operators. To that end, we introduce the following terminology:

\begin{defn}\label{def.slice}
Let $\cL$ be a (potentially) multi-valued operator on $B_b(X)$. A \emph{slice of} $\cL$ is a single-valued operator $L_0: D(L_0) \to B_b(X)$
such that $(f, L_0f) \in \cL$ for every $f\in D(L)$.
\end{defn}

From Lemma \ref{l.bpcore}, we immediately obtain:

\begin{cor}\label{c.slice}
Let $(T(t))_{t\geq 0}$ be a transition semigroup with full generator $\Lf$ and $L_0$ be a slice of $\Lf$ such that:
\begin{enumerate}
[\upshape (i)]
\item For every $t >0$ and $f\in D(L)$ it is $T(t)f\in D(L_0)$ and $L_0T(t)f = T(t)L_0f$;
\item $D(L_0)$ is bp-dense in $B_b(X)$;
\item for every $f\in D(L_0)$and $x\in X$ the map $t\mapsto T(t)L_0f(x)$ is continuous at almost every point in $(0,\infty)$.
\end{enumerate}
Then $\{ (f, L_0f) : f\in D(L_0)\}$ is a bp-core for $\Lf$. By slight abuse of notation, we will say that $D(L_0)$ \emph{is a bp-core}
for $\Lf$.
\end{cor}

\subsection{\texorpdfstring{$C_b$-}{}semigroups and the \texorpdfstring{$C_b$-}{}generator}\label{s.a3}

\begin{defn}\label{d.cbsem}
A transition semigroup $T$ is called \emph{$C_b$-semigroup} if
\begin{enumerate}
[(i)]
\item for every $t\geq 0$ the operator $T(t)$ leaves the space $C_b(X)$ invariant;
\item for every $f\in C_b(X)$ the map $(t,x) \mapsto [T(t)f](x)$ is continuous on $[0,\infty)\times X$.
\end{enumerate}
A transition group $T$ is called \emph{$C_b$-group} if the above hold true for every $t\in \R$.
\end{defn}

Note that the continuity assumption in Definition \ref{d.cbsem} implies via a bp-closedness argument the measurability assumption
in Definition \ref{d.transsg}.

We can give an equivalent description of $C_b$-semigroups using an additional locally convex topology on $C_b(X)$, the
so-called \emph{strict topology} $\beta_0$. This topology is defined as follows: We denote by $\cF_0(X)$ the space of all functions
$\varphi: X \to \R$ that \emph{vanish at infinity}, i.e.\ given $\eps>0$ we find a compact subset $K$ of $X$ such that
$|\varphi(x)|\leq \eps$ for all $x\in X\setminus K$. The strict topology $\beta_0$ is generated by the seminorms 
$\{p_\varphi : \varphi\in \cF_0(X)\}$, where $p_\varphi(f) = \|\varphi f\|_\infty$. This topology is consistent with the duality 
$(C_b(X), \cM(X))$, i.e.\ the dual space $(C_b(X), \beta_0)'$ is $\cM(X)$, see \cite[Theorem 7.6.3]{jarchow81}. In fact, it is the 
Mackey topology of the dual pair $(C_b(X), \cM(X))$, i.e.\ the finest locally convex topology on $C_b(X)$ which yields
$\cM(X)$ as a dual space. This is a consequence of Theorems 4.5 and 5.8 of \cite{sentilles}. Consequently, a kernel operator
on $C_b(X)$ (i.e.\ a $\sigma$-continuous operator by Lemma \ref{l.kernelop}) is automatically $\beta_0$-continuous.
By \cite[Theorem 2.10.4]{jarchow81} $\beta_0$ coincides on $\|\cdot\|_\infty$-bounded sets with the topology of uniform
convergence on compact subsets of $X$.

We recall the following characterization from \cite[Theorem 4.4]{kunze09}.

\begin{thm}
\label{t.beta0char}
Let $T$ be a transition semigroup such that every operator $T(t)$ leaves $C_b(X)$ invariant.
The following are equivalent:
\begin{enumerate}
[\upshape (i)]
\item $T$ is a $C_b$-semigroup, i.e.\ the map $(t,x) \mapsto [T(t)f](x)$ is continuous on $[0,\infty)\times X$
for every $f\in C_b(X)$.
\item $T$ is $\beta_0$-continuous on $C_b(X)$, i.e.\ for every $f\in C_b(X)$, 
the orbit $t\mapsto T(t)f$ is $\beta_0$-continuous on $[0,\infty)$.
\item $T$ is $\beta_0$-continuous on $C_b(X)$ and locally $\beta_0$-equicontinuous.
\end{enumerate}
\end{thm}

We note that many transition semigroups are in fact $C_b$-semigroups, see, e.g.,  
the references \cite{gnr24, akk16, kk22, bk24} for concrete examples.

Using the dominated convergence theorem, it is easy to see that if $T$ is a transition semigroup such that every $T(t)$ leaves the space
$C_b(X)$ invariant, then also its Laplace transform $R(\lambda)$ leaves the space $C_b(X)$ invariant. If additionally the continuity
assumption in Definition \ref{d.cbsem}(ii) is satisfied, $R(\lambda)|_{C_b(X)}$ is injective and thus the resolvent of a unique
(single-valued) operator.

\begin{defn}
Let $T$ be a $C_b$-semigroup and $(R(\lambda))_{\Re\lambda>0}$ be its Laplace transform. The \emph{$C_b$-generator}
of $T$ is the unique single-valued operator $L$ such that $(\lambda - L)^{-1} = R(\lambda)|_{C_b(X)}$ for all $\Re\lambda >0$.
\end{defn}

As the resolvent of $L$ is merely the restriction of that of $\Lf$ to $C_b(X)$, the $C_b$-generator $L$ is the \emph{part} of $\Lf$ in 
$C_b(X)$, i.e.\ $f\in D(L)$ and $Lf=g$ if and only if $(f,g)\in \Lf$ and $f,g\in C_b(X)$.
We give some alternative characterizations of the $C_b$-generator.

\begin{thm}\label{t.genchar}
Let $T$ be a $C_b$-semigroup with $C_b$-generator $L$. For $f,g\in C_b(X)$ the following are equivalent:
\begin{enumerate}
[\upshape (i)]
\item $f\in D(L)$ and $Af=g$;
\item $t^{-1}(T(t)f-f) \to g$ with respect to $\sigma$ as $t\to 0$;
\item $t^{-1}(T(t)f-f) \to g$ with respect to $\beta_0$ as $t\to 0$;
\item $\sup_{t\in (0,1)}\|t^{-1}(T(t)f-f)\|_\infty < \infty$ and $t^{-1}(T(t)f(x) - f(x))\to g(x)$
for all $x\in X$.
\end{enumerate}
\end{thm}

\begin{proof}
See \cite[Theorem A.5]{bk24}.
\end{proof}

The full generator can be reconstructed from the $C_b$-generator. In fact, we have:

\begin{cor}\label{c.genbpcore}
Let $(T(t))_{t\geq 0}$ be a $C_b$-semigroup with $C_b$-generator $L$ and full generator $\Lf$. Then $D(L)$ is a bp-core
for $\Lf$.
\end{cor}

\begin{proof}
If $f\in D(L)$ and $x\in X$, then $Lf\in C_b(X)$ and the map $t\mapsto [T(t)Lf](x)$ is continuous on all of $(0,\infty)$ proving (iii) in Corollary
\ref{c.slice}. Condition (i) can easily be obtained from \cite[Proposition 5.7]{kunze11}. It remains to prove condition (ii). To that end, note
that for every $f\in C_b(X)$, it holds $\lambda R(\lambda, L)f \in D(L)$ and $\lambda R(\lambda, L)f \to f$ pointwise 
as $\lambda \to \infty$, see \cite[Theorem 2.10]{kunze09}. This proves that the bp-closure of $D(L)$ contains
$C_b(X)$ and hence $B_b(X)$. Now Corollary \ref{c.slice} yields the claim.
\end{proof}

Note that the concept of bp-core is only appropriate on the space $B_b(X)$. We also introduce a suitable concept of a core for the 
$C_b$-generator. If $T$ is a $C_b$-semigroup with $C_b$-generator $L$ and $D\subset D(L)$, we say that $D$ is a \emph{$\beta_0$-core}
 for $L$
if for every $f\in D(L)$ there is a net $(f_\alpha)\subset D$ such that $f_\alpha \to f$ and $Lf_\alpha \to Lf$ with respect to 
$\beta_0$.

\begin{lem}
\label{l.cbcore}
Let $T$ be a $C_b$-semigroup with $C_b$-generator $L$ and $D\subset D(L)$ be such that
\begin{enumerate}
[\upshape (i)]
\item $D$ is dense in $C_b(X)$ with respect to $\beta_0$ and
\item $T(t)D\subset D$ for all $t>0$.
\end{enumerate}
Then $D$ is a $\beta_0$-core for $L$.
\end{lem}

\begin{proof}
This follows along the lines of \cite[Proposition 2.3]{kunze13} which is concerned with the full generator.
\end{proof}

\section{Examples of path spaces}\label{ap.pathspace}

\subsection{Continuous paths}\label{sub.continuous}

Let $(X, d)$ be a separable, complete metric space and put $\cXc\coloneqq C(\R; X)$, the set of all continuous functions from $\R$ to $X$, and
\[
\md_{\mathrm{C}}(\x, \y)\coloneqq \sum_{n=1}^\infty 2^{-n} \big[ 1\wedge \sup_{t\in [-n,n]} d(\x(t), \y(t))\big].
\]
We also set
\[
[\tau_\mathrm{C}(\x)](t) \coloneqq \begin{cases}
\x(t), & \mbox{if } t < 0,\\
\x(0), & \mbox{if } t\geq 0.
\end{cases}
\]

\begin{prop}\label{p.continuous}
The pair $((\cXc, \md_{\mathrm{C}}), \tau_\mathrm{C})$ is a path space. Moreover, in this case the map $\tau_\mathrm{C}$
and every evaluation map $\pi_t$, for $t\in \R$, is continuous.
\end{prop}

\begin{proof}
The evaluation maps $\pi_t$ are clearly continuous from $\cXc$ to $X$ and it is well known that they generate the Borel $\sigma$-algebra, 
proving (P1). We point out that, by the continuity of the paths, we have $\cF_t = \cF_{t-}$ in this case. In what follows,
we work with $\cF_t$ instead of $\cF_{t-}$ whenever $\cX = \cXc$.
\medskip

Note that $\tau_{\mathrm{C}}$ is continuous and thus certainly measurable. Moreover, as $\cXc^-$ is closed
in $\cXc$, it is complete with respect to the original metric $\md_{\mathrm{C}}$. To finish the proof of (P2), it remains to prove the measurability 
requirement. To that end,  let $t_1 < t_2 < \ldots < t_k \leq 0 < t_{k+1} < \ldots < t_n$ and 
$A_j \in \bo (X)$ for $j=1, \ldots, n$. Consider
the cylinder set
\[
A = \{ \x : \x (t_j) \in A_j \mbox{ for } j =1, \ldots, n\} \in \bo (\cXc).
\]
Then
\[
(\tau_{\mathrm{C}} )^{-1}(A) = \Big\{ \x : \x (t_j) \in A_j \mbox{ for } j=1, \ldots, k \mbox{ and } \x(0)\in \bigcap_{j=k+1}^n A_j \Big\} \in \cF_0.
\]
As the cylinder sets generate $\bo (\cXc)$, it follows that $\tau_\mathrm{C}$ is $\cF_0$-measurable. Note that in the above
situation, the cylinder set $A$ belongs to $\cF_0$ if and only if $k=n$. In that case $\tau_{\mathrm{C}}^{-1}(A) = A$. Once again,
cylinder sets of this form generate $\cF_0$ so that
$(\tau_{\mathrm{C}})^{-1}(A) = A$ for all $A\in \cF_0$, proving (P2).\medskip

To prove (P3), let $t_n \to t_\infty$ and $\x_n \to \x_\infty$.
Given $\eps >0$, pick $N_0$ so large  that $\sum_{n\geq N_0} 2^{-n} \leq \eps$ and then $N_1$ so large that
$t+t_n \in [-N_1, N_1]$ for all $t\in [-N_0, N_0]$ and $n\in \N\cup\{\infty\}$.
 
As the set $\{ \x_n : n\in \N\cup\{\infty\}\}$ is compact, it follows from the Arzel\`a--Ascoli Theorem that we find $\delta>0$
such that $d(\x_n(t), \x_n(s)) \leq \eps$ for all $n\in \N\cup\{\infty\}$ and $t, s \in [-N_1, N_1]$ with $|t-s|\leq \delta$. Next pick
$n_0$ so large, that $|t_n - t_\infty|\leq \delta$ and $\sup_{t\in [-N_1, N_1]}d(\x_n(t), \x_\infty(t)) \leq \eps$
for all $n\geq n_0$. Then for $n\geq n_0$
\begin{align*}
d((\vt_{t_n}\x_n)(t), &(\vt_{t_\infty}\x_\infty)(t))  = 
d(\x_n(t+t_n), \x_\infty(t+t_\infty))\\
& \leq d(\x_n(t+t_n), \x_n(t+t_\infty)) + d(\x_n(t+t_\infty), \x_\infty(t+t_\infty)) \leq 2\eps 
\end{align*}
for all $t\in [-N_0, N_0]$. Altogether, it follows that $\md_{\mathrm{C}}(\vt_{t_n}\x_n, \vt_{t_\infty}\x_{\infty}) \leq 3\eps$
for $n\geq n_0$. This proves that $(t, \x)\mapsto \vt_t\x$ is continuous and thus (P3).
\end{proof}

\subsection{C\`adl\`ag paths in the \texorpdfstring{$J_1$}{}-topology}\label{sub.cadlag}

Let $(X, d)$ be a complete, separable metric space. Without loss of generality, we assume that $d(x,y) \leq 1$ for all
$x,y\in X$. We consider the set $\cXd = D(\R; X)$ of all \emph{c\`adl\`ag functions} $\x : \R \to X$, i.e.\ for every $t \in \R$,
we have
\[
\x(t) = \lim_{s\downarrow t} \x(s) \quad\mbox{and}\quad \x(t-) \coloneqq \lim_{s\uparrow t}\x(s) \mbox{ exists}.
\]
As is well known (see e.g.\ \cite[Lemma 3.5.1]{ek}), any $\x\in \cXd$ has at most countably many discontinuities. There are several ways
to extend Skorohod's $J_1$-topology, originally defined for c\`adl\`ag functions on $[0,1]$, to unbounded time intervals.
Billingsley \cite{billingsley} and Lindvall \cite{lindvall73} use, similar to the case of continuous functions, a series over
the distance of restrictions of the functions to compact time intervals. We note that in the $J_1$-topology,
convergence of a c\`adl\`ag function on a compact
time interval $[a,b]$ implies convergence of the values at the endpoints $a$ and $b$. In the approach of Billingsley and Lindvall, 
the restrictions to the compact time intervals have to be slightly modified so that convergence with respect to the metric on the unbounded
time interval does not entail convergence of the values in these (technically chosen) endpoints.

On the other hand, Whitt \cite{whitt80} and Ethier and Kurtz \cite{ek} use an integral instead of a series and no such alteration
is needed. We will here follow this second approach. We should also point out that
\cite{whitt80} is one of the few references where also the time interval $\R$ (instead of $[0,\infty)$) is considered.

Given $-\infty < a < b < \infty$,  we denote by $\Lambda_a^b$ the collection of all strictly increasing, 
bijective and Lipschitz continuous functions $\lambda$ from $[a, b]$ to $[a, b]$ so that
\[
\|\lambda\|_L \coloneqq \sup_{a \leq s<t\leq b} \Big|\log \frac{\lambda (t) - \lambda (s)}{t-s} \Big| < \infty.
\]
On $D([a, b], X)$, the space of $X$-valued c\`adl\`ag functions on the interval $[a,b]$, we define the metric
\begin{equation}\label{eq.bddskor}
\md_a^b (\x, \y) = 
\inf_{\lambda\in \Lambda_a^b } \big[ \|\lambda \|_L \vee \sup_{t\in [a, b]} d(\x(t), \y(\lambda(t)))\big\}.
\end{equation}
This is a complete metric that induces Skorohod's $J_1$-topology on $D([a, b], X)$, see \cite[Sect.\ 12]{billingsley}.
We note that $\md_a^b(\x,\y)\leq 1$, as $d(x,y) \leq 1$ for all $x,y\in X$ and the choice $\lambda(t) = t$ yields 
$\|\lambda\|_L = 0$.

Given $\x, \y\in D(\R; X)$, we denote their restrictions to $[a, b]$ by $\x|_a^b$ and $\y|_a^b$ respectively and define
\[
\md_{J_1}(\x, \y) \coloneqq \int_{-\infty}^0\int_0^\infty 
e^{s-t}\md_s^t (\x|_s^t, \y|_s^t)\, dt\, ds.
\]
This integral is well defined by \cite[Lemma 2.4]{whitt80}.
By \cite[Theorem 2.6]{whitt80}  $(\cXd, \md_{\mathrm{D}})$ is a complete separable metric space (see \cite[Theorem 3.5.6]{ek}) which
induces Skorohod's $J_1$-topology on $D(\R; X)$ in the sense that $\x_n \to \x$ in $(\cXd, \md_{\mathrm{D}})$ if and only if
$\x_n|_a^b \to \x|_a^b$ in $D([a,b], X)$ whenever $a$ and $b$ are continuity points of $\x$, see \cite[Theorem 2.5]{whitt80}.

To simplify notation, we will not distinguish between a function $\x \in \cXd$ and its restriction to some compact
interval $[a,b]$ in what follows. It will be clear from the context what  the domain of definition is. \smallskip

On $\cXd$, we define the stopping map $\tau_\mathrm{D}$ by
\[
[\tau_\mathrm{D}(\x)](t) = \begin{cases}
\x(t), & \mbox{if } t <0,\\
\x(0-), & \mbox{if } t\geq 0.
\end{cases}
\]

To prove that also the c\`adl\`ag functions form a path space, we make use of the following `\emph{modulus of continuity}':
\[
\omega'(\x, \delta, T) \coloneqq \inf_{(t_0, \ldots, t_n)} \max_{j=1, \ldots, n} \sup_{s,t\in [t_{j-1}, t_j)} d(\x(s), \x(t)),
\]
where the infimum is taken over all partitions $-T = t_0 < t_1 < \ldots < t_n = T$ satisfying $\min_j (t_j - t_{j-1})\geq \delta$. 
Similar to the Arzel\`a--Ascoli theorem, this modulus of continuity can be used to characterize compact subsets of 
$\cXd$ (see \cite[Theorem 3.6.3]{ek}).

In order to prove that $((\cXd, \md_{J_1}), \tau_\mathrm{D})$ is a path space, we first establish a Lemma.

\begin{lem}\label{l.simpleshift}
Let $\x \in D(\R; X)$ be a piecewise constant function with finitely many jumps, i.e.\
\[
\x(t) = x_0\one_{(-\infty, r_0)}(t) + \sum_{j=1}^n x_j \one_{[r_{j-1}, r_j)}(t) + x_{n+1}\one_{[r_n, \infty)}(t),
\]
where $-\infty < r_0 < r_1 < \ldots < r_n < \infty$ satisfies $\min_j(r_j-r_{j-1}) \geq \delta>0$ and $x_0, \ldots, x_{n+1} \in X$.
For $|t| < \delta$ we put
\[
f_\delta(t) \coloneqq \max \Big\{ \log \frac{\delta}{\delta - |t|}, \log \Big(1 + \frac{|t|}{\delta}\Big)\Big\}.
\]
Then for $|t|< \delta/2$ we have $\md_{J_1}(\x, \vt_t\x) \leq f_\delta(t)$. In particular, $\vt_t\x \to \x$ as
$t\to 0$.
\end{lem}

\begin{proof}
Fix real numbers $a<b$ and put
\[
\{t_0, \ldots, t_k\} \coloneqq (\{r_0, \ldots, r_n\} \cap [a,b])\cup\{a,b\},
\]
so that $a=t_0 < t_1 < \cdots < t_k = b$.  Now, let $|t| < \delta/2$. Then, for every $j=2, \ldots, k-2$ we have
$t_j-t \in (t_{j-1}, t_{j+1})$ but if $t>0$ it is possible that $t_1-t< t_0$ and in the case where $t<0$ it is possible
that $t_{k-1}-t > t_k$. Define $s_0 < \ldots < s_k$ by setting $s_0=t_0 = a$ and $s_k=t_k = b$ and
\[
s_j = \begin{cases}
t_1 - t, & \mbox{if } j=1 \mbox{ and } t_1-t > t_0,\\
t_1, & \mbox{if } j=1 \mbox{ and } t_1-t \leq t_0,\\
t_j - t, & \mbox{if } j = 2, \ldots , k-2,\\
t_{k-1} - t, & \mbox{if } j=k-1 \mbox{ and } t_{k-1}-t < t_k,\\
t_{k-1}, & \mbox{if } j=1 \mbox{ and } t_{t-1}-t \geq t_k.
\end{cases}
\]
Now define $\lambda_t \in \Lambda_a^b$ in such a way, that $\lambda_t(s_j) = t_j$ and $\lambda_t$ is piecewise linear, i.e.
\[
\lambda_t(s) = t_j + \frac{t_j-t_{j-1}}{s_j-s_{j-1}} (s- s_{j-1})= t_j + d_j(t) (s-s_{j-1})
\]
for $s\in [s_{j-1}, s_j]$
We note that $\lambda_t'(s) = d_j(t)$ for $s\in (s_{j-1}, s_j)$ and $d_j(t)\equiv 1$ for $j=2, \ldots, k-2$. Also in cases $j=1$ and $j=k-1$ it is
possible that $d_j(t)=1$ and this is the case when $s_j = t_j$. However, a brief computation shows that in any case
$|\log d_j(t)| \leq f_\delta(t)$ for $j=1$ and $k-1$ so that $\|\lambda\|_L \leq f_\delta(t)$.

Noting that $\vt_t\one_{[t_{j-1}, t_j)} = \one_{[t_{j-1}-t, t_j-t)}$ one sees that $(\vt_t\x)(\lambda_t(s)) = \x(s)$ for all $s\in [a,b]$
so that $\md_a^b(\vt_t\x, \x) \leq \|\lambda_t\|_L \leq f_\delta(t)$. As $a, b$ were arbitrary, the claim follows from the definition
of $\md_{J_1}$.
\end{proof}

We are now ready to prove:

\begin{prop}\label{p.cadlag}
The pair $((\cXd, \md_{J_1}), \tau_\mathrm{D})$ is a path space.
\end{prop}

\begin{proof}
If $\x_n\to \x$, then $\pi_t(\x_n) \to \pi_t(\x)$ whenever $t$ is a continuity point of $\x$ (\cite[Proposition 3.5.2]{ek}). The fact
that the Borel $\sigma$-algebra is generated by the evaluation maps $\pi_t$ follows from \cite[Proposition 3.7.1]{ek}.
The measurability requirements in  (P2) can be proved in the same way as in Proposition \ref{p.continuous}. To see that $\cXd^-$ is Polish, define for $\x, \y\in \cXd^-$
\[
\md^-_{J_1}(\x, \y) \coloneqq \int_0^\infty e^{-t}\md_{-t}^0(\x, \y)\, dt.
\]
As every $\x\in \cXd^-$ is continuous at $0$, it follows from the results of \cite[Sect.\ 2]{whitt80} already used above that
$\md^-_{J_1}$ defines a complete metric on $\cXd^-$ such that $\x_n\to \x$ if and only if the restrictions to
$[a, 0]$ converge in $D([a,0], X)$ for every $a<0$ that is a continuity point of $\x$. 
Noting that every element of $\cXd^-$ is constant on $[0, b]$ for every $b>0$
it is easy to see that $\x_n \to \x$ in $D([a,0], X)$ if and only if $\x_n \to \x$ in $D([a,b], X)$ for all $b>0$; to see that
convergence in $D([a,b], X)$ implies convergence in $D([a,0]; X)$ use the continuity of $\x$ in $0$ and \cite[Lemma 2.2]{whitt80}.
This proves that a sequence $(\x_n)\subset \cXd^-$ converges to some $\x\in \cXd^-$ with respect to
$\md^-_{J_1}$ if and only if it converges with respect $\md_{J_1}$. Consequently,
$\md^-_{J_1}$ induces the same topology on $\cXd^-$ as $\md_{J_1}$ the proof of (P2) is finished.\smallskip

To prove (P3), we show that for any compact subset $K$ of $\cXd$ it is $\vt_t\x\to \vt_s\x$ as $t\to s$ uniformly for $\x\in K$.
We may assume without loss of generality that $s=0$. Given $\eps>0$, pick $T>0$ such that $\int_{T-1}^\infty e^{-t}\, dt \leq \eps$.
As $K$ is compact, there is $\delta >0$ such that $\omega'(\x, \delta, T) \leq \eps$ for all $\x\in K$, see \cite[Theorem 3.6.3]{ek}.
For fixed $\x \in K$ we find a partition $\pi = \pi(\x)$ of $[-T, T]$, say $\pi = (-T=t_0 < t_1 < \ldots < t_n=T)$, with
$\min_j(t_j-t_{j-1}) \geq \delta$ and $\max_j\sup_{t,s\in [t_{j-1}, t_j)}d(\x(t), \x(s)) \leq 2\eps$.
Define $\x_\pi$ by setting
\[
\x_\pi(s) = \x(t_0-)\one_{(-\infty, t_0)}(s) + \sum_{j=1}^n\x(t_{j-1})\one_{[t_{j-1}, t_j)}(s) + \x(t_n)\one_{[t_n, \infty)}.
\]
Then $d(\x_\pi(s), \x(s)) \leq 2\eps$ for every $s\in [-(T-1), T-1]$ and, using $\lambda(s) = s$, it follows that
$\md_a^b(\x_\pi(s), \x(s)) \leq 2\eps$ for all $-(T-1)\leq a < 0 < b \leq T-1$ and thus altogether $\md_{J_1}(\x_\pi, \x) \leq 2\eps$.

Basically the same calculation shows $\md_{J_1}(\vt_t\x_\pi, \vt_t\x) \leq 2\eps$ whenever $|t|\leq 1$. Taking Lemma \ref{l.simpleshift} into account, it follows that
\[
\md_{J_1}(\vt_t\x, \x) \leq \md_{J_1}(\vt_t\x, \vt_t\x_\pi) + \md_{J_1}(\vt_t\x_\pi, \x_\pi) + 
\md_{J_1}(\x_\pi, \x)
\leq 6\eps + f_\delta(t)
\]
for all $|t|\leq \delta/2$. Picking $|t|$ so small that $f_\delta(t) \leq \eps$ (note that this only depends on $\delta$ and may thus be
done indepently of the particular $\x\in K$), $\md_{J_1}(\vt_t\x, \x) \leq 7\eps$. As $\eps>0$ was arbitrary, (P3) is proved.
\end{proof}

\subsection{C\`adl\`ag paths as proper path space}\label{sub.proper}

We now verify that $\cXd$, the space of all c\`adl\`ag paths endowed with the $J_1$-topology is indeed a proper
path space so that the results of Section \ref{sub.general} apply. We start with a Lemma.

\begin{lem}\label{l.compacteminus}
Let $K \subset \cXd^-$ be compact. Then, for any $\eps>0$ there exists $\delta>0$ such that
$d(\x(t), \x(0)) \leq \eps$ for all $t\in [-\delta, 0]$ and $\x\in K$.
\end{lem}

\begin{proof}
If the conclusion is wrong, we find $\eps_0>0$, a sequence $(\x_n)\subset K$ and a sequence $t_n\uparrow 0$ with 
$d(\x_n(t_n), \x(0))\geq \eps_0$ for all $n\in \N$. As $K$ is compact, passing to a subsequence, we may and shall assume
that $\x_n$ converges to some $\x\in K$. In particular, $\x$ is continuous at $0$ so that $\x_n(0) \to \x(0)$
by \cite[Proposition 3.5.2]{ek}. But then \cite[Proposition 3.6.5]{ek} implies $\x_n(t_n)\to \x(0)$ which yields $\x_n(t_n) \to \x(0)$.
This is a contradiction.
\end{proof}

\begin{prop}
The space $((\cXd, \md_{J_1}), \tau_\mathrm{D})$ is a proper path space.
\end{prop}

\begin{proof}
Let $\x_n\to \x$. By (P3)
$\vt_t\x_n\to \vt_t\x$ for all $t\in \R$. Note that $\vt_t\x$ is continuous at $0$ for almost all $t\in \R$. Using \cite[Proposition 3.6.5]{ek},
it is easy to see that $\tau(\vt_t\x_n) \to \tau(\vt_\x)$ whenever $\vt_t\x$ is continuous at $0$. This proves (P4).\smallskip

The proof of (P5) is similar to that of (P3). Given $\eps>0$ and $K\subset \cXd^-$, Lemma \ref{l.compacteminus} allows us to
 pick $\delta>0$ so that
$d(\x(t), \x(0)) \leq \eps$ for all $t\in [-\delta, 0]$ and $\x\in K$. Next, choose $T>0$ so large that
$\int_{T-1}^\infty e^{-t}\, dt \leq \eps$. Picking a smaller $\delta$ if necessary, we may and shall assume that
$\omega'(\x, 2\delta, T) \leq \eps$ for all $\x\in K$.

Now, fix $\x\in K$ and pick a partition $\pi =(t_0< \ldots < t_n)$ of $[-T, T]$ such that $\min_j (t_j- t_{j-1}) \geq 2\delta$ and
$\max_j\sup_{t,s\in [t_{j-1}, t_j)}d(\x(t), \x(s)) \leq 2\eps$. As $\x(t) = \x(0) = \x(0-)$ for all $\x\in \cXd^-$ we may
and shall assume that $t_{n-1}\leq 0$ so that the only partition point larger than $0$ is $t_n = T$. Next, $\pi$ is modified
to $\tilde\pi$ as follows. If $|t_{n-1}| < \delta$, we replace $t_{n-1}$ by $\tilde{t}_{n-1}\coloneqq - \delta$ whereas all other
partition points are unchanged. This results in a partition $\tilde\pi$ satisfying $\min_j (\tilde t_j - \tilde t_{j-1})\geq \delta$
and $\max_j\sup_{t,s\in [\tilde t_{j-1}, \tilde t_j)}d(\x(t), \x(s)) \leq 2\eps$. Repeating the arguments from above,
$\md_{\mathrm{D}}(\x_{\tilde{\pi}}, \x) \leq 3\eps$.

Now let $0\leq t\leq \delta/2$. Note that $(\vt_{-t}\x_{\tilde\pi})(s) = \x_{\tilde\pi}(s-t) = \x(\tilde{t}_{n-1})$ for all $s\geq \tilde{t}_{n-1}+t$
and thus, in particular, for all $s\geq 0$. It follows that $\tau_\mathrm{D}(\vt_{-t}\x_{\tilde\pi}) = \vt_{-t}\x_{\tilde\pi}$ and Lemma \ref{l.simpleshift} yields
$\md_{\mathrm{D}}(\x_{\tilde\pi}, \tau_\mathrm{D}(\vt_{-t}\x_{\tilde\pi})) \leq f_\delta(t)$.

Finally, we estimate $\md_{\mathrm{D}}(\tau_\mathrm{D}(\vt_{-t}\x_{\tilde\pi}), \tau_\mathrm{D}(\vt_{-t}\x))$. If $-T<s<0$ it is
$\tau_\mathrm{D}(\vt_{-t}\x)(s) = \x(s-t)$ and thus $d(\tau_\mathrm{D}(\vt_{-t}\x_{\tilde\pi})(s), \tau_\mathrm{D}(\vt_{-t}\x)(s)) = d(\x_{\tilde{\pi}}(s-t), \x(s-t)) \leq 2\eps$ by the above.
On the other hand, for $s\geq 0$ it is $d(\tau_\mathrm{D}(\vt_{-t}\x_{\tilde\pi})(s), \tau_\mathrm{D}(\vt_{-t}\x)(s)) = d(\x(\tilde{t}_{n-1}), \x(-t-)) \leq \eps$
by the initial choice of $\delta$. Thus $\md_{\mathrm{D}}(\tau_\mathrm{D}(\vt_{-t}\x_{\tilde\pi}), \tau_\mathrm{D}(\vt_{-t}\x)) \leq 3\eps$. Altogether,
\[
\md_{\mathrm{D}} (\x, \tau_\mathrm{D}(\vt_{-t}\x)) \leq 6\eps + f_\delta(t)
\]
whenever $\x\in K$ and $0\leq t\leq \delta/2$. This implies (P5).\smallskip

Finally, we prove (P6). As $\md_a^b$ and thus $\md_{J_1}$ is bounded by 1 we find $t^\star>0$ such that
\[
\md_{t^\star}(\x,\y) \coloneqq \int_{-\infty}^0\int_0^{t^\star} e^{t-s} \md_s^t(\x,\y)\, dt\, ds \geq \frac{1}{2}\md_{J_1}(\x,\y)
\]
for all $\x,\y\in \cX$. For $t_0>t^\star$, Lemma \ref{l.measurability}(e) yields $\md_{t^*}(\tau^{t_0}(\x), \tau^{t_0}(\y)) = \md_{t^\star}(\x, \y)$
and thus
\[
\md(\tau^{t_0}(\x), \tau^{t_0}(\y)) \geq \md_{t^\star}(\x,\y) \geq \frac{1}{2}\md_{J_1}(\x,\y).
\]
This finishes the proof.
\end{proof}

\end{document}